\documentclass[11pt,a4paper,abstract=yes]{scrartcl}
\usepackage[utf8]{inputenc}
\usepackage[T1]{fontenc}
\usepackage{hyperref}
\usepackage{amsmath}
\usepackage{amsthm}
\usepackage{amssymb}
\usepackage{mathbbol}
\usepackage{graphicx}
\usepackage{tikz}
\usepackage{tkz-graph, tkz-berge}
\usetikzlibrary{arrows.meta}
\usetikzlibrary{shapes, graphs}
\usepackage{longtable}
\usepackage[framemethod=TikZ]{mdframed}
\mdfdefinestyle{examplestyle}{%
linecolor=brown!25,linewidth=1pt,%
frametitlerule=false,%
frametitlebelowskip=1pt,%
frametitlebackgroundcolor=brown!2,
backgroundcolor=brown!1,
innertopmargin=0pt,
roundcorner=5pt,
nobreak=false,
}
\mdtheorem[style=examplestyle]{mexample}[theorem]{Example}
\mdfdefinestyle{defstyle}{%
linecolor=blue!5,linewidth=1pt,%
frametitlerule=false,%
frametitlebackgroundcolor=blue!1,
backgroundcolor=blue!1,
innertopmargin=1pt,
roundcorner=5pt,
nobreak=true,
}
\theoremstyle{plain}
\newtheorem{theorem}{Theorem}[section]
\newtheorem{lemma}[theorem]{Lemma}
\newtheorem{definition}[theorem]{Definition}
\newtheorem{corollary}[theorem]{Corollary}

\newtheorem{question}[theorem]{Question}
\newtheorem{remark}[theorem]{Remark}

\newcommand{\QQ}{\mathbb{Q}}

\newcommand{\RR}{\mathbb{R}}
\newcommand{\Prob}{\mathbb{Pr}}
\newcommand{\probnull}{\mathbb{nu}}
\newcommand{\problin}[1]{\mathbb{l}_{#1}}
\newcommand{\problins}{\mathbb{l}}
\newcommand{\linexts}[1]{\mathcal{L}_{#1}}

\newcommand{\prpolytope}[1]{\mathcal{Q}_{#1}}
\newcommand{\orderpolytope}[1]{\mathcal{O}({#1})}
\newcommand{\orderpolytopebounds}[1]{\widehat{\mathcal{O}}({#1})}
\newcommand{\arel}{\le_{E_+}}
\newcommand{\brel}{\le_{E_-}}
\newcommand{\erel}{\le_{E}}
\newcommand{\transerel}{\preceq_{E}}
\newcommand{\twoanti}[1]{\mathcal{A}({#1})}
\newcommand{\tarel}{\le_{A_+}}
\newcommand{\tbrel}{\le_{A_-}}
\newcommand{\terel}{\le_{A}}
\newcommand{\transterel}{\preceq_{A}}
\newcommand{\Pairs}[1]{\underline{#1}}
\newcommand{\Ppairs}{\Pairs{P}}

\newcommand{\PairsQuotient}[1]{\widehat{\Pairs{#1}}}
\newcommand{\Pqu}{\PairsQuotient{P}}

\newcommand{\pqurel}{\preceq}
\author{Jan Snellman\thanks{Department of Mathematics, Linköping University, 58183 Linköping, Sweden; jan.snellman@liu.se.}}
\date{2025-01-25}
\title{The Polytope of Probability Functions on a Finite Poset}
\hypersetup{
 pdfauthor={Jan Snellman},
 pdftitle={The Polytope of Probability Functions on a Finite Poset},
 pdfkeywords={},
 pdfsubject={},
 pdfcreator={Emacs 30.0.93 (Org mode 9.7.19)}, 
 pdflang={English}}
\usepackage[backend=bibtex,style=numeric]{biblatex}
\begin{document}

\maketitle

\begin{abstract}
Kim, Kim, and Neggers \autocite{neggers2019probfunc} defined probability functions on a poset,
by listing some very natural conditions that a function \(\pi: P \times P \to [0,1]\)
should satisfy in order to capture the intuition of ``the likelihood that \(a\) precedes \(b\) in \(P\)''.
In particular, this generalizes the common notion of poset probability for finite posets, where
\(\pi(a,b)\) is the proportion of linear extensions of \(P\) in which \(a\) precedes \(b\).

They constructed a family of such functions for posets embedded in the ordered plane;
that is two say, for posets of order dimension at most two.

We study probability functions of a finite poset \(P\) by constructing an
ancillary poset \(\twoanti{P}\), that we call \textbf{probability functions posets}. The relations of this
new poset encodes the restrictions imposed on probability functions of the original poset by
the conditions of the definition. Then, we define the probability functions polytope \(\prpolytope{P}\),
which parameterizes the probability functions on \(P\), and show that it can be realized
as the \emph{order polytope} \(\orderpolytope{\twoanti{P}}\) of \(\twoanti{P}\) intersected by a certain affine subspace.

We give a partial description of the vertices of \(\prpolytope{P}\) and show that, in contrast to
the order polytope, it is not always a lattice polytope.
\end{abstract}
\section{Introduction}
\label{sec:org29ca3b5}
\subsection{Preliminaries}
\label{sec:org048367c}
A \emph{partially ordered set}, or \emph{poset} for short, is an ordered pair \((X,\le)\) where \(X\) is a set and
\(\le\) is reflexive, anti-symmetric, and transitive. Two elements \(x,y \in X\) are \emph{parallel},
written \(x \parallel y\), if neither \(x \le y\) nor \(y \le x\). On the other hand, if
\(x \le y\) or \(y \le x\) then \(x,y\) are \emph{comparable}. The \emph{diagonal} of the relation \(\le\)
is the set
\[\Delta = \{(x,x) \mid x \in X\},\]
and the \emph{strict part} of \(\le\) is
\[< \, = \, (\le \setminus \Delta);\]
in other words, \(x < y\) if and only if \(x \le y\) and \(x \neq y\).
We say that \(y\) \emph{covers} \(x\), written
\(x \lessdot y\), if \(x < y\) but there is no \(z \in X\) with \(x < z < y\).

If \((P,\le)\) and \((Q,\preceq)\) are two posets, then a map \[\phi: P \to Q\] is \emph{isotone}, or \emph{order-preserving},
if \(x \le y\) implies \(\phi(x) \preceq \phi(y)\). The dual poset \(Q^{op}\) has the same elements as \((Q,\preceq)\)
but the reverse ordering. A map \[\tau: P \to Q\] is \emph{antitone}, or \emph{order-reversing}, if
\(\tau: P \to Q^{op}\) is isotone. An isotone bijection with isotone inverse is called an \emph{isomorphism}.

When \((X,\le)\) is a poset,
a subset \(S \subseteq X\) is an \emph{order ideal} if
\[y \in S, x \in X, x \le y \quad  \implies \quad x \in S.\]
Dually, an \emph{order filter} \(T \subseteq X\) has the defining property that
\[y \in T, x \in X, y \le x \quad \implies \quad x \in T.\]
A subset \(A \subseteq X\) is an \emph{antichain} if
\[x,y \in A, x \neq y \quad \implies \quad x \parallel y. \]
A subset \(C \subseteq X\) is a \emph{chain} if \(x,y \in C\)  implies that \(x\) and \(y\) are comparable.
If \(Z \subseteq X\) then the \emph{induced subposet} is
\((Z,R)\) where \(a R b\) if and only if \(x \le y\).

The poset \((X,\le)\) is a \emph{linear order}, or \emph{total order},  if any two elements in \(X\) are
comparable (thus the induced subposet of a chain is a linear order). For a positive integer \(n\),
any two linear orders with \(n\) elements are isomorphic;
we denote this poset by \(C_{n}\). Similarly, any two \emph{antichain posets} with \(n\) elements
are isomorphic.

For a poset \((P,\le)\), a poset \((P, \preceq)\) such that \(\le \, \subseteq \, \preceq\) is called an \emph{extension}
of the former; in other words,
\[x \le y \quad \implies \quad x \preceq y.\]
An extension to a linear order is a \emph{linear extension}. This concept is mostly used
for finite posets \(P\). In this case, we denote then set of linear extensions of \(P\)
by \(\linexts{P}\), and the cardinality of this set by \(e(P)\).

If \((P,\le)\), \((Q, \preceq)\) are posets, we denote by \(P + Q\) their \emph{disjoint union}. The underlying set
is the disjoint union \(P \sqcup Q\), and the order relation is that \(x \le_{P + Q} y\) if
\(x,y \in P\) and \(x \le_{P} y\), or if
\(x,y \in Q\) and \(x \le_{Q} y\).
The \emph{ordinal sum} \(P \oplus Q\) has the same underlying set, but here \(x \le_{P \oplus Q} y\) if
\(x \le_{P + Q} y\) or if \(x \in P, y \in Q\).

A (closed) \emph{polyhedron} is the intersection of finitely many closed affine halfspaces
\[\{\vec{x} \mid \ell(\vec{x}) \ge c\} \subset \RR^{n}.\]
A \emph{polytope} is a bounded polyhedron. It is the
irredundant convex hull of finitely many points, its \emph{vertices}. The irredundant \emph{bounding hyperplanes}
cut out the \emph{facets} of the polytope. The \emph{dimension} of a polytope is the dimension of its
\emph{affine hull}.
\subsection{Poset probability and the 1/3-2/3 conjecture.}
\label{sec:orgd28a5a6}
For a finite poset \((P,\le)\), a much investigated notion of poset probability
is to define
\begin{equation}
\label{eq-sortingprob}
  \pi_{P}(x,y) = \frac{|\linexts{P}(x < y)|}{|\linexts{P}|}
\end{equation}
where \(\linexts{P}(x < y)\) denotes the set of linear extensions of \(P\) that place \(x\) before \(y\).

Study of this probability function has been driven by interest in the
 \textbf{1/3-2/3 conjecture}, proposed by Kislitsyn \autocite{Kislitsyn13} in 1968
(see also \autocite{Fredman13}\autocite{Linial13}\autocite{Brightwell1999}\autocite{olson20181}) which states that
every finite poset \(P\) which is not totally ordered has a \(1/3\)-balanced pair
\((x,y)\) with respect to the sorting probability \(\pi_{P}\), meaning that
\(1/3 \le \pi_{P}(x,y) \le 2/3\); equivalently,
that \(\min(\pi_{P}(x,y), \pi_{P}(y,x)) \ge 1/3\).
Another way of expressing this is that
\begin{equation}
\label{eq-13maxmin}
  \max_{x,y \in P} \min (\Prob_{P}(x,y), \Prob_{P}(y,x))
  \ge \frac{1}{3}
\end{equation}
\subsection{The linear extensions polytope}
\label{sec:org7c0b6f8}
To a finite poset \((X,\le)\) one can associate its \emph{linear extension polytope}
\autocites{FacetsLinordpoly}[][]{FioriniLinextpolytope}
as the convex hull of all characteristic vectors of linear extensions of \(P\).
In more detail, put
\begin{equation}
\label{eq-linexA}
A = \{(i,j) \mid i,j \in X, i \neq j\}
\end{equation}
and for any linear extension \(L=(X,\preceq)\)  define the vector
\(\chi^{L}\) so that \(\chi_{{ij}}^{L}=1\) if \(i \prec j\) and \(\chi_{{ij}}^{L}=0\) otherwise.
Then
\[\mathbf{P}_{{\mathrm{LO}}}(P) = \mathrm{conv}(\{ \chi^{L} \mid L \in \linexts{P} \})\]
is the linear extension polytope of \(P\).

Schulz \autocite{SchulzPolytopes} showed that the affine hull of
\(\mathbf{P}_{{\mathrm{LO}}}(P)\) is defined by
\begin{align}
\label{eq-linex-eq}
  x_{ij} &= 1, \text{ whenever } i < j, \\
  x_{ij} &= 0, \text{ whenever } i > j, \\
  x_{ij} + x_{ji} &= 1, \text{ whenever } i \parallel j.
\end{align}
It follows that the dimension of \(\mathbf{P}_{{\mathrm{LO}}}(P)\) is equal to the number
of unordered pairs of incomparable elements in \(P\).

The probability functions polytope the we will define in this article will be somewhat similar to the linear extensions
polytope. We will not investigate the relationship between these two polytopes, but it seems
a promising avenue for further research.
\subsection{The order polytope}
\label{sec:orgdec1c22}
To a finite poset \((Q,\le)\) one can furthermore associate its \emph{order polytope} \(\orderpolytope{Q}\)
of order-preserving maps \(f: Q \to [0,1]\). The seminal article by Stanley \autocite{StanleyTwoPosetPolytopes}
lists some important properties of order polytopes, and the interplay between combinatorial
properties of \(Q\) and geometric properties of \(\orderpolytope{Q}\). There has since been considerably
developments in this field \autocite{HibiCutting}\autocite{HibiEdgesOrderpolytope}\autocite{vonBellTriangulationsOrderpolytopes}\autocite{AhmadOrderChain}\autocite{Freij-Hollanti_Lundström_2024}\autocite{HibiUnimodularOrderpolytopes}, including the study of double posets and double order polytopes \autocite{https://doi.org/10.37236/8381}\autocite{https://doi.org/10.1137/16m1091800}, and so-called
order-chain polytopes \autocite{https://doi.org/10.1016/j.ejc.2016.06.007}.

We will realize our probability functions polytope as the intersection of certain order polytopes with an affine subspace
cut out by the equations
\[  x_{ij} + x_{ji} = 1, \text{ whenever } i \parallel j,\]
similar to the equations (\ref{eq-linex-eq}) for the  linear extensions polytope.
\subsection{Outline of the article}
\label{sec:org693ca3b}
In section 2, we review Kim, Kim, and Neggers definition of probability functions on a poset.
We make some technical tweaks, like restricting the domain of definition to \((P \times P) \setminus \Delta\),
similar to what we saw in equation (\ref{eq-linexA}).

In section 3 we define the ``derived poset'' \((\Ppairs, \transerel)\), and its quotient
\((\Pqu, \pqurel)\) to capture the conditions on probability functions on \(P\) imposed by the
axioms of Definition \ref{def-probfunc}. Since the relation on parallel pairs \(x \parallel y\) is determined,
it is enough to specify the relations on \emph{ordered pairs of antichains}; to this end,
we define the poset \((\twoanti{P},\transterel)\) which encodes the same information as \((\Pqu, \pqurel)\).
We give an example that show why it is not feasible to reduce the size of the derived poset
further by considering \emph{unordered pairs of antichains}.

Section 4 introduces the main object of study for this article, the probability functions polytope \(\prpolytope{P}\)
which parameterizes all probability functions on \(P\). We first give a straightforward definition
by encoding a probability function on \(P\) as a square matrix; the set of these matrices are easily seen
to form a convex polytope. By considering only non-parallel pairs, and taking into account the relations
\begin{equation}
\label{eq-for-H}
\pi(x,y) + \pi(y,x)=1 \quad \forall x,y \in P, x \parallel y
\end{equation}
for such pairs, we reduce the embedding dimension, resulting in a more tractable object.

In fact, we show that the probability functions polytope \(\prpolytope{P}\) can be described as the intersection of the order polytope
\(\orderpolytope{\twoanti{P}}\) with the affine subspace \(H\) cut out by the equations
(\ref{eq-for-H}). Since the face-lattice structure of order polytopes is known
(from the work of Geissinger \autocite{geissingerpolytope}), and in particular there are explicit descriptions
of the vertices of order polytopes
\autocite{StanleyTwoPosetPolytopes}\autocite{HibiEdgesOrderpolytope}\autocite{HibiCutting}, we can describe
some of the vertices of \(\prpolytope{P}\). Somewhat surprisingly, we find an example where the vertices
of \(\prpolytope{P}\) are non-integral, showing that \(\prpolytope{P}\) need not be a lattice polytope.

This example that we found is for a poset \(P\) with 8 vertices. In the last section of this article,
we list the probability functions polytopes (or rather, their edge graphs and \(f\)-vectors)
for posets with three or four elements. All these probability functions polytopes are lattice polytopes, so it is possible
that the example found (\(P=C_{2} \times C_{2} \times C_{2}\)) is in fact the smallest counterexample. An exhaustive
search for posets with \(\le 7\) elements does not seem completely infeasible, but would require either
more clever programming
(we use the polytope functionality in Sagemath \autocite{sagemath} in a rather amateurish fashion)
or computer resources beyond what we currently have access to.
\section{Probability functions on a poset}
\label{sec:orgc60bae7}
\subsection{Definition}
\label{sec:org7a53a1b}
In \autocite{neggers2019probfunc}, Kim, Kim, and Neggers defined the notion of a probability function on a partially ordered set \((P,\le)\) as follows:

\begin{definition}
A function \(\pi: P \times P \to [0,1]\) is called a \textbf{probability function} on \(P\) if for all \(x,y \in P\),
\begin{enumerate}
\item \(\pi(x,x)=1\),
\item \(x \le y\) implies \(\pi(x,y)=1\),
\item \(x \neq y\) implies \(\pi(x,y) + \pi(y,x)=1\),
\item \(y < z\) implies \(\pi(x,y) \le \pi(x,z)\).
\end{enumerate}
\label{def-probfunc}
\end{definition}
The value \(\pi(x,y)\) should be thought of as \(\Prob_{\pi}(x \preceq y)\), the probability
that \(x\) precedes \(y\) (though \(x\) and \(y\) might be incomparable).

The authors proved:
\begin{lemma}
If \((P,\le\) is a poset, and \(x,y,z \in P\), then
\[
y < z \quad \implies \quad   \pi(z,x) \le \pi(y,x).
\]
\end{lemma}

Implicit in the same article is the following:
\begin{lemma}
Let \((P,\le)\) be a poset, and let \(\pi_{1},\dots,\pi_{n}\) be probability functions on \(P\).
Then any convex combination
\[
\pi = \sum_{j=1}^{n} c_{j}\pi_{j}, \qquad \sum_{j=1}^{n} c_{j} =1, \qquad \forall j: \, c_{j} \ge 0
\]
is a probability function on \(P\).
\label{lemma-isconvex}
\end{lemma}
\begin{proof}
Take \(x,y,z \in P\). Then
\[\pi(x,x) = \sum_{j=1}^{n}c_{j}\pi_{j}(x,x) = \sum_{j=1}^{n}c_{j}\cdot 1 = 1.
\]
If \(x \le y\) then
\[
\pi(x,y) = \sum_{j=1}^{n}c_{j}\pi_{j}(x,y) = \sum_{j=1}^{n}c_{j} = 1.
\]
If \(x \neq y\) then
\begin{align*}
\pi(x,y) + \pi(y,t)
& = \sum_{j=1}^{n}c_{j}\pi_{j}(x,y) + \sum_{j=1}^{n}c_{j}\pi_{j}(y,x) \\
& = \sum_{j=1}^{n}c_{j}\bigl(\pi_{j}(x,y) + \pi_{j}(y,x) \bigr) \\
& = \sum_{j=1}^{n}c_{j}\cdot 1 \\
& = 1.
\end{align*}
If \(y < z\) then
\[
\pi(x,y) = \sum_{j=1}^{n}c_{j}\pi_{j}(x,y)  \le \sum_{j=1}^{n}c_{j}\pi_{j}(x,z) = \pi(x,z).
\]
\end{proof}
\subsection{Important probability functions}
\label{sec:orga859451}
Kim, Kim, and Neggers  \autocite{neggers2019probfunc} defined the following three important probability functions.
\begin{definition}
Let \((P,\le)\) be a finite poset with \(|P|=n\), and let \(x,y \in P\).
Let \(\linexts{P}\) be the set of linear (total) extensions of \(P\),
i.e. the set of all bijections \(\ell: P \to [n] = \{1,2,\dots,n\}\) that are order-preserving,
and let \(\ell \in \linexts{P}\) is some total extension.
Define
\begin{align}
\label{eq-3probfuncs}
  \probnull(x,y) & =
                   \begin{cases}
                     1 & x \le y, \\
                     1/2 & x \parallel y, \\
                     0 & x > y,
                   \end{cases} \\
\problin{\ell}(x,y) &=
   \begin{cases}
     1 & \text{ if } \ell(x) < \ell(y), \\
     1 & \text{ if  } x = y, \\
     0 & \text{ otherwise}
   \end{cases} \\
\problins(x,y) & = \frac{1}{|\linexts{P}|} \sum_{\ell \in \linexts{P}} \problin{\ell}(x,y)
\end{align}
\end{definition}
\subsection{Reduced probability functions}
\label{sec:orgdf55db8}
Since for any function on a poset \(P\), and any \(x \in P\), it holds that \(\pi(x,x)=1\),
this we make the following technical change to the definition:
\begin{definition}
Let \((P, \le)\) be a poset. Denote by \(\Delta = \{(x,y) \in P \times P \, \mid \, x \neq y\}\),
and put  \(\Ppairs = (P \times P \setminus \Delta)\).
A \textbf{reduced} probability function on \(P\) is a function \(\pi: \Ppairs \to [0,1]\)
satisfying
\begin{enumerate}
\item \(x \le y\) implies \(\pi(x,y)=1\),
\item \(x \neq y\) implies \(\pi(x,y) + \pi(y,x)=1\),
\item \(y < z\) implies \(\pi(x,y) \le \pi(x,z)\),
\item \(y < z\) implies \(\pi(z,x) \le \pi(y,x)\).
\end{enumerate}
\label{def-probfunc-reduced}
\end{definition}
Note that Lemma (\ref{lemma-isconvex}) still holds for reduced probability functions, indeed the two concepts are interchangeable; for our purposes, though, the modified definition has some technical benefits that will
become apparent.
We will, in what follows, refer to reduced probability functions simply as probability functions.
\section{The probability functions poset of a poset}
\label{sec:org4b3e4ab}
\subsection{Relations on all non-equal pairs}
\label{sec:orgad4a4f5}
\begin{definition}
Let \((P,\le)\) be a finite poset. Suppose that \(P\) is not a chain.
Define  binary relations on \(\Ppairs\) by
\begin{equation}
\label{eq-erel}
  \begin{split}
    (x,y) & \arel (u,v) \iff x = u \text{ and } y \le v \\
    (x,y) & \brel (u,v) \iff y = v \text{ and } u \le x \\
    (x,y) & \erel (u,v) \iff (x,y) \arel (u,v) \text{ or } (x,y) \brel (u,v)
  \end{split}
\end{equation}
\end{definition}

\begin{theorem}
The transitive closure of \(\erel\) is a poset.
\label{punc-poset}
\end{theorem}
\begin{proof}

Reflexivity: \((x,y) \arel (x,y)\) since \(x=x\) and \(y \le y\).

No directed cycles: suppose that
\begin{equation}
\label{eq-erelchain}
(x_{1},y_{1}) \erel (x_{1},y_{1}) \erel \cdots \erel (x_{n},y_{n})
\end{equation}
and that each comparison is strict, i.e.
\[(x_{i},y_{i}) \neq (x_{i+1},y_{i+1}).\]
If
\[(x_{i},y_{i}) \arel (x_{i+1},y_{i+1})\] then \(y_{i} < y_{i+1}\), and if
\[(x_{i},y_{i}) \brel (x_{i+1},y_{i+1})\] then \(x_{i+1} < x_{i}\). Hence either (or both) \(x_{n} < x_{1}\) or \(y_{n} > y_{1}\), so
\[(x_{1},y_{1}) \neq (x_{n},y_{n}).\] Thus (\ref{eq-erelchain}) is no directed cycle.
\end{proof}

\begin{definition}
We denote by \((\Ppairs, \transerel)\) the poset which is the transitive closure of \(\erel\).
\end{definition}

\begin{lemma}
The map
\begin{equation}
\label{isotone}
  \begin{split}
    \tau: \Ppairs & \to \Ppairs \\
    \tau((x,y)) & = (y,x)
  \end{split}
\end{equation}
is an antitone involution on the poset \((\Ppairs, \transerel)\).
\end{lemma}
\begin{proof}
Clearly \(\tau \circ \tau\) is the identity.
If \((x,y) \arel (u,v)\) then \[(v,u) \brel (y,x),\]
and if \((x,y) \brel (u,v)\) then \[(v,u) \arel (y,x).\]
Thus if \((x,y) \erel (u,v)\) then
\[\tau((u,v)) \erel \tau((x,y)).\] So by transitivity, if
\((x,y) \transerel (u,v)\) then
\[\tau((u,v)) \transerel \tau((x,y)).\]
\end{proof}

\begin{mexample}
Take \(P= C_{2} \times C_{2}\) and as the disjoint sum \(C_{2} + C_{2}\), then
 \((\Ppairs, \transerel)\)  are as follows:
\begin{center}\begin{tabular}{c|c|c|c}
\(P\) & \(\Ppairs\) & \(P\) & \(\Ppairs\) \\
\resizebox{!}{3cm}{
 \begin{tikzpicture}[>=latex,line join=bevel,]
\node (node_0) at (20.811bp,6.5307bp) [draw,draw=none] {$1$};
  \node (node_1) at (5.8112bp,55.592bp) [draw,draw=none] {$2$};
  \node (node_2) at (35.811bp,55.592bp) [draw,draw=none] {$3$};
  \node (node_3) at (20.811bp,104.65bp) [draw,draw=none] {$4$};
  \draw [black,->] (node_0) ..controls (17.124bp,19.099bp) and (13.812bp,29.49bp)  .. (node_1);
  \draw [black,->] (node_0) ..controls (24.499bp,19.099bp) and (27.81bp,29.49bp)  .. (node_2);
  \draw [black,->] (node_1) ..controls (9.4985bp,68.161bp) and (12.81bp,78.552bp)  .. (node_3);
  \draw [black,->] (node_2) ..controls (32.124bp,68.161bp) and (28.812bp,78.552bp)  .. (node_3);
\end{tikzpicture}
}

   &
\resizebox{!}{3cm}{
 \begin{tikzpicture}[>=latex,line join=bevel,]
\node (node_0) at (84.39bp,8.3018bp) [draw,draw=none] {$\left(4, 1\right)$};
  \node (node_1) at (14.39bp,60.906bp) [draw,draw=none] {$\left(4, 2\right)$};
  \node (node_2) at (108.39bp,60.906bp) [draw,draw=none] {$\left(4, 3\right)$};
  \node (node_3) at (61.39bp,60.906bp) [draw,draw=none] {$\left(3, 1\right)$};
  \node (node_6) at (155.39bp,60.906bp) [draw,draw=none] {$\left(2, 1\right)$};
  \node (node_4) at (43.39bp,113.51bp) [draw,draw=none] {$\left(3, 2\right)$};
  \node (node_7) at (125.39bp,113.51bp) [draw,draw=none] {$\left(2, 3\right)$};
  \node (node_5) at (20.39bp,166.11bp) [draw,draw=none] {$\left(3, 4\right)$};
  \node (node_10) at (67.39bp,166.11bp) [draw,draw=none] {$\left(1, 2\right)$};
  \node (node_11) at (90.39bp,218.72bp) [draw,draw=none] {$\left(1, 4\right)$};
  \node (node_8) at (114.39bp,166.11bp) [draw,draw=none] {$\left(2, 4\right)$};
  \node (node_9) at (161.39bp,166.11bp) [draw,draw=none] {$\left(1, 3\right)$};
  \draw [black,->] (node_0) ..controls (63.757bp,24.218bp) and (46.772bp,36.496bp)  .. (node_1);
  \draw [black,->] (node_0) ..controls (90.951bp,23.136bp) and (95.962bp,33.7bp)  .. (node_2);
  \draw [black,->] (node_0) ..controls (78.102bp,23.136bp) and (73.301bp,33.7bp)  .. (node_3);
  \draw [black,->] (node_0) ..controls (105.32bp,24.218bp) and (122.55bp,36.496bp)  .. (node_6);
  \draw [black,->] (node_1) ..controls (22.361bp,75.815bp) and (28.507bp,86.539bp)  .. (node_4);
  \draw [black,->] (node_2) ..controls (112.99bp,75.589bp) and (116.43bp,85.836bp)  .. (node_7);
  \draw [black,->] (node_3) ..controls (56.523bp,75.589bp) and (52.878bp,85.836bp)  .. (node_4);
  \draw [black,->] (node_4) ..controls (37.102bp,128.34bp) and (32.301bp,138.91bp)  .. (node_5);
  \draw [black,->] (node_4) ..controls (49.951bp,128.34bp) and (54.962bp,138.91bp)  .. (node_10);
  \draw [black,->] (node_5) ..controls (41.023bp,182.03bp) and (58.008bp,194.31bp)  .. (node_11);
  \draw [black,->] (node_6) ..controls (147.1bp,75.891bp) and (140.65bp,86.775bp)  .. (node_7);
  \draw [black,->] (node_7) ..controls (122.43bp,128.12bp) and (120.24bp,138.21bp)  .. (node_8);
  \draw [black,->] (node_7) ..controls (135.39bp,128.57bp) and (143.25bp,139.62bp)  .. (node_9);
  \draw [black,->] (node_8) ..controls (107.83bp,180.95bp) and (102.82bp,191.51bp)  .. (node_11);
  \draw [black,->] (node_9) ..controls (140.46bp,182.03bp) and (123.23bp,194.31bp)  .. (node_11);
  \draw [black,->] (node_10) ..controls (73.678bp,180.95bp) and (78.479bp,191.51bp)  .. (node_11);
\end{tikzpicture}
}

   &
\resizebox{!}{2.0cm}{
 \begin{tikzpicture}[>=latex,line join=bevel,]
\node (node_0) at (5.8112bp,6.5307bp) [draw,draw=none] {$3$};
  \node (node_1) at (5.8112bp,55.592bp) [draw,draw=none] {$4$};
  \node (node_2) at (35.811bp,6.5307bp) [draw,draw=none] {$1$};
  \node (node_3) at (35.811bp,55.592bp) [draw,draw=none] {$2$};
  \draw [black,->] (node_0) ..controls (5.8112bp,19.029bp) and (5.8112bp,29.252bp)  .. (node_1);
  \draw [black,->] (node_2) ..controls (35.811bp,19.029bp) and (35.811bp,29.252bp)  .. (node_3);
\end{tikzpicture}
}

   &
\resizebox{!}{2.0cm}{
 \begin{tikzpicture}[>=latex,line join=bevel,]
\node (node_0) at (14.39bp,8.3018bp) [draw,draw=none] {$\left(4, 3\right)$};
  \node (node_1) at (61.39bp,8.3018bp) [draw,draw=none] {$\left(4, 1\right)$};
  \node (node_2) at (55.39bp,60.906bp) [draw,draw=none] {$\left(4, 2\right)$};
  \node (node_4) at (102.39bp,60.906bp) [draw,draw=none] {$\left(3, 1\right)$};
  \node (node_5) at (79.39bp,113.51bp) [draw,draw=none] {$\left(3, 2\right)$};
  \node (node_3) at (131.39bp,8.3018bp) [draw,draw=none] {$\left(3, 4\right)$};
  \node (node_6) at (190.39bp,8.3018bp) [draw,draw=none] {$\left(2, 3\right)$};
  \node (node_7) at (149.39bp,60.906bp) [draw,draw=none] {$\left(2, 4\right)$};
  \node (node_9) at (196.39bp,60.906bp) [draw,draw=none] {$\left(1, 3\right)$};
  \node (node_10) at (172.39bp,113.51bp) [draw,draw=none] {$\left(1, 4\right)$};
  \node (node_8) at (237.39bp,8.3018bp) [draw,draw=none] {$\left(2, 1\right)$};
  \node (node_11) at (284.39bp,8.3018bp) [draw,draw=none] {$\left(1, 2\right)$};
  \draw [black,->] (node_1) ..controls (59.786bp,22.834bp) and (58.608bp,32.769bp)  .. (node_2);
  \draw [black,->] (node_1) ..controls (72.905bp,23.514bp) and (82.119bp,34.886bp)  .. (node_4);
  \draw [black,->] (node_2) ..controls (61.951bp,75.74bp) and (66.962bp,86.304bp)  .. (node_5);
  \draw [black,->] (node_4) ..controls (96.102bp,75.74bp) and (91.301bp,86.304bp)  .. (node_5);
  \draw [black,->] (node_6) ..controls (178.88bp,23.514bp) and (169.66bp,34.886bp)  .. (node_7);
  \draw [black,->] (node_6) ..controls (191.99bp,22.834bp) and (193.17bp,32.769bp)  .. (node_9);
  \draw [black,->] (node_7) ..controls (155.68bp,75.74bp) and (160.48bp,86.304bp)  .. (node_10);
  \draw [black,->] (node_9) ..controls (189.83bp,75.74bp) and (184.82bp,86.304bp)  .. (node_10);
\end{tikzpicture}
}

\end{tabular}\end{center}
\end{mexample}
\subsubsection{Chains and antichains}
\label{sec:org472a157}
\begin{lemma}
If \(P\) is an antichain with \(n\) elements, then \((\Ppairs, \transerel)\)
is an antichain with \(n(n-1)\) elements.
\end{lemma}
\begin{proof}
In this case, the relations \(\arel,\brel\), and hence \(\erel\) and \(\transerel\) are all trivial,
so only idential elements are related in \((\Ppairs, \transerel)\).
\end{proof}

\begin{lemma}
If \(P=C_{n}\) is a chain with \(n\) elements, which w.l.o.g. may be taken to be \(\{1,\dots,n\}\),
with the natural order, then the elements of \((\Ppairs, \transerel)\)
are \(\{(i,j) \mid i \neq j\}\), and
the resulting poset is ranked with rank function
\begin{equation*}
r((i,j)) =
\begin{cases}
  n - i + j  & i > j \\
  n - i + j - 1 & i < j
\end{cases}
\end{equation*}
The cover relations are
\begin{enumerate}
\item \((i,j) \lessdot (i,j+1)\) and \((i,j) \lessdot (i-1,j)\), when \(i > j\), \(r((i,j)) < n - 1\),
\item \((i,j) \lessdot (i,j+2)\), if \(j+2 \le n\), and \((i,j) \lessdot (i-2,j)\), if \(i-2 \ge 1\),
when \(i > j\), \(r((i,j)) < n - 1\),
\item \((i,j) \lessdot (i,j+1)\), if \(j+1 \le n\), and \((i,j) \lessdot (i-1,j)\), if \(i \ge 1\), when \(i < j\).
\end{enumerate}
\end{lemma}
\begin{proof}
Define \(c((i,j)) = j - i\). Then
\[c(\Ppairs)=\{1-n, 2-n, \dots - 1, 1,2, \dots, n-1\}.\]
Since the cover relations in \(C_{n}\) are \(a \lessdot a+1\)
the cover relations in \(\Ppairs\) are \((i,j) \lessdot (i,j+1)\)
and \((i,j) \lessdot (i-1,j)\), provided that all these elements belong to
\[\Ppairs = \{(i,j) \mid 1 \le i,j \le n, \, i \neq j\}.\]
Each such cover relation \(u \lessdot v\) has \(c(u) = c(v)-1\).
However, there are no elements with \(c((i,j))=0\), so the elements with
\(c=-1\) are instead covered by those with \(c=+1\). The rank function \(r\) is shifted to take this
into account.
\end{proof}

\begin{mexample}
For a chain of length 4, the poset
\(\Ppairs\) is as follows.

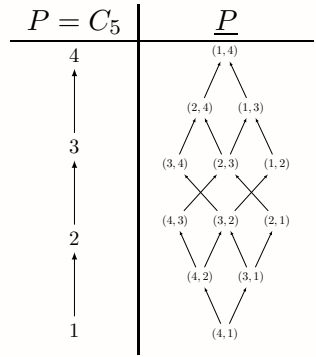
\phantomsection
\label{}
\begin{center}\begin{tabular}{c|c}
 \(P=C_5\) & \(\Ppairs\) \\ \hline
\resizebox{!}{4cm}{
 \begin{tikzpicture}[>=latex,line join=bevel,]
\node (node_0) at (5.8112bp,6.5307bp) [draw,draw=none] {$1$};
  \node (node_1) at (5.8112bp,55.592bp) [draw,draw=none] {$2$};
  \node (node_2) at (5.8112bp,104.65bp) [draw,draw=none] {$3$};
  \node (node_3) at (5.8112bp,153.71bp) [draw,draw=none] {$4$};
  \draw [black,->] (node_0) ..controls (5.8112bp,19.029bp) and (5.8112bp,29.252bp)  .. (node_1);
  \draw [black,->] (node_1) ..controls (5.8112bp,68.09bp) and (5.8112bp,78.313bp)  .. (node_2);
  \draw [black,->] (node_2) ..controls (5.8112bp,117.15bp) and (5.8112bp,127.37bp)  .. (node_3);
\end{tikzpicture}
}

   &
\resizebox{!}{4cm}{
 \begin{tikzpicture}[>=latex,line join=bevel,]
\node (node_0) at (60.39bp,8.3018bp) [draw,draw=none] {$\left(4, 1\right)$};
  \node (node_1) at (37.39bp,60.906bp) [draw,draw=none] {$\left(4, 2\right)$};
  \node (node_3) at (84.39bp,60.906bp) [draw,draw=none] {$\left(3, 1\right)$};
  \node (node_2) at (14.39bp,113.51bp) [draw,draw=none] {$\left(4, 3\right)$};
  \node (node_4) at (61.39bp,113.51bp) [draw,draw=none] {$\left(3, 2\right)$};
  \node (node_7) at (61.39bp,166.11bp) [draw,draw=none] {$\left(2, 3\right)$};
  \node (node_6) at (108.39bp,113.51bp) [draw,draw=none] {$\left(2, 1\right)$};
  \node (node_5) at (14.39bp,166.11bp) [draw,draw=none] {$\left(3, 4\right)$};
  \node (node_9) at (108.39bp,166.11bp) [draw,draw=none] {$\left(1, 2\right)$};
  \node (node_8) at (37.39bp,218.72bp) [draw,draw=none] {$\left(2, 4\right)$};
  \node (node_10) at (84.39bp,218.72bp) [draw,draw=none] {$\left(1, 3\right)$};
  \node (node_11) at (60.39bp,271.32bp) [draw,draw=none] {$\left(1, 4\right)$};
  \draw [black,->] (node_0) ..controls (54.102bp,23.136bp) and (49.301bp,33.7bp)  .. (node_1);
  \draw [black,->] (node_0) ..controls (66.951bp,23.136bp) and (71.962bp,33.7bp)  .. (node_3);
  \draw [black,->] (node_1) ..controls (31.102bp,75.74bp) and (26.301bp,86.304bp)  .. (node_2);
  \draw [black,->] (node_1) ..controls (43.951bp,75.74bp) and (48.962bp,86.304bp)  .. (node_4);
  \draw [black,->] (node_2) ..controls (27.66bp,128.8bp) and (38.375bp,140.33bp)  .. (node_7);
  \draw [black,->] (node_3) ..controls (78.102bp,75.74bp) and (73.301bp,86.304bp)  .. (node_4);
  \draw [black,->] (node_3) ..controls (90.951bp,75.74bp) and (95.962bp,86.304bp)  .. (node_6);
  \draw [black,->] (node_4) ..controls (48.12bp,128.8bp) and (37.405bp,140.33bp)  .. (node_5);
  \draw [black,->] (node_4) ..controls (74.66bp,128.8bp) and (85.375bp,140.33bp)  .. (node_9);
  \draw [black,->] (node_5) ..controls (20.678bp,180.95bp) and (25.479bp,191.51bp)  .. (node_8);
  \draw [black,->] (node_6) ..controls (95.12bp,128.8bp) and (84.405bp,140.33bp)  .. (node_7);
  \draw [black,->] (node_7) ..controls (54.829bp,180.95bp) and (49.819bp,191.51bp)  .. (node_8);
  \draw [black,->] (node_7) ..controls (67.678bp,180.95bp) and (72.479bp,191.51bp)  .. (node_10);
  \draw [black,->] (node_8) ..controls (43.678bp,233.55bp) and (48.479bp,244.11bp)  .. (node_11);
  \draw [black,->] (node_9) ..controls (101.83bp,180.95bp) and (96.819bp,191.51bp)  .. (node_10);
  \draw [black,->] (node_10) ..controls (77.829bp,233.55bp) and (72.819bp,244.11bp)  .. (node_11);
\end{tikzpicture}
}

\end{tabular}\end{center}
\end{mexample}
\subsection{The quotient poset}
\label{sec:orgcca35bf}
Recall the following definitions (see e.g. \autocite{StanleyTwoPosetPolytopes}, \autocite{AhmadOrderChain})
\begin{definition}
Let \((Q,\le)\) be a finite poset, and \(\gamma\) a set partition  on \(Q\).
We say that \(\gamma\) is \textbf{compatible} if the following relation on the blocks of \(\gamma\) has a transitive closure which is
a poset (in other words, the transitive closure should be antisymmetric):
\[
B_{\iota} \le_{\gamma} B_{\nu} \quad \iff \quad \exists a \in B_{\iota}, b \in B_{\nu}, \, a \le b.
\]
In this case, the resulting poset of blocks, ordered by \(\le_{\gamma}\), is the called a \textbf{quotient} of \(Q\),
and denoted by \(Q/\gamma\).
\label{def-set-partition-compatible}
\end{definition}
We will later need:
\begin{definition}
Let \((Q,\le)\) be a finite poset.
\begin{enumerate}
\item Let \(S \subset Q\). Then \(S\) is \emph{connected} as an induced subposet of \(Q\) (``connected inside \(Q\)'')
if the undirected graph of the Hasse diagram of the induced subposet of \(Q\) is connected.
Concretely, for any \(a,b \in S\) there is a sequence \(s_1,\dots, s_n \in S\)
such that \(s_1=a\), \(s_n=b\) and for \(1 \le i < n\), either \(s_i \le s_{i+1}\) or \(s_i \ge s_{i+1}\).
\item Let \(\gamma\) be a set partition  on \(Q\). Then \(\gamma\) is \emph{connected} if every block
is connected as an induced subposet of \(Q\).
\end{enumerate}
\label{def-set-partition-connected}
\end{definition}

We now construct a convenient  quotient of \((\Ppairs, \transerel)\).
\begin{definition}
Let \((P,\le)\) be a finite poset.
Construct a set partition \(\sigma\) on \(\Ppairs\) by
\begin{enumerate}
\item The ``bottom block''  consists of all \((x,y)\) with \(x < y\),
\item the ``middle blocks'' consists of all \((x,y)\) with \(x \parallel y\); each such element of \(\Ppairs\)
form a singleton block,
\item the ``top block'' consists of all \((x,y)\) with \(x > y\).
\end{enumerate}

This set partition is compatible with \((\Ppairs, \transerel)\), thus we can form the
quotient
\[(\Pqu, \pqurel) = \Ppairs/\sigma. \]
\label{def-quotient-poset}
\end{definition}

\begin{lemma}
If \((P,\le)\) is a finite  anti-chain, then the quotient poset \((\Pqu, \pqurel)\)
is equal to \((\Ppairs, \transerel)\) and thus an antichain.

If \(P\) is a finite chain, then \((\Pqu, \pqurel)\) is a two-element chain.

For other finite \(P\), \((\Pqu, \pqurel)\) has a unique minimal element \(\hat{0}\) and a unique maximal element \(\hat{1}\), and some other elements in between.
\label{lemma-quotient-poset}
\end{lemma}
\begin{proof}
Immediate.
\end{proof}
\subsubsection{Examples}
\label{sec:org5e69bec}
\begin{mexample}
Let again \(P=C_{2} \times C_{2}\). Then \(\Pqu\) becomes:
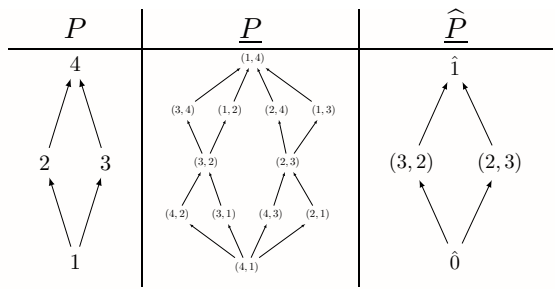
\phantomsection
\label{}
\begin{center}\begin{tabular}{c|c|c}
\(P\) & \(\Ppairs\) & \(\Pqu\) \\ \hline
\resizebox{!}{3cm}{
 \begin{tikzpicture}[>=latex,line join=bevel,]
\node (node_0) at (20.811bp,6.5307bp) [draw,draw=none] {$1$};
  \node (node_1) at (5.8112bp,55.592bp) [draw,draw=none] {$2$};
  \node (node_2) at (35.811bp,55.592bp) [draw,draw=none] {$3$};
  \node (node_3) at (20.811bp,104.65bp) [draw,draw=none] {$4$};
  \draw [black,->] (node_0) ..controls (17.124bp,19.099bp) and (13.812bp,29.49bp)  .. (node_1);
  \draw [black,->] (node_0) ..controls (24.499bp,19.099bp) and (27.81bp,29.49bp)  .. (node_2);
  \draw [black,->] (node_1) ..controls (9.4985bp,68.161bp) and (12.81bp,78.552bp)  .. (node_3);
  \draw [black,->] (node_2) ..controls (32.124bp,68.161bp) and (28.812bp,78.552bp)  .. (node_3);
\end{tikzpicture} 
}

 & 
\resizebox{!}{3cm}{
 \begin{tikzpicture}[>=latex,line join=bevel,]
\node (node_0) at (84.39bp,8.3018bp) [draw,draw=none] {$\left(4, 1\right)$};
  \node (node_1) at (14.39bp,60.906bp) [draw,draw=none] {$\left(4, 2\right)$};
  \node (node_2) at (108.39bp,60.906bp) [draw,draw=none] {$\left(4, 3\right)$};
  \node (node_3) at (61.39bp,60.906bp) [draw,draw=none] {$\left(3, 1\right)$};
  \node (node_6) at (155.39bp,60.906bp) [draw,draw=none] {$\left(2, 1\right)$};
  \node (node_4) at (43.39bp,113.51bp) [draw,draw=none] {$\left(3, 2\right)$};
  \node (node_7) at (125.39bp,113.51bp) [draw,draw=none] {$\left(2, 3\right)$};
  \node (node_5) at (20.39bp,166.11bp) [draw,draw=none] {$\left(3, 4\right)$};
  \node (node_10) at (67.39bp,166.11bp) [draw,draw=none] {$\left(1, 2\right)$};
  \node (node_11) at (90.39bp,218.72bp) [draw,draw=none] {$\left(1, 4\right)$};
  \node (node_8) at (114.39bp,166.11bp) [draw,draw=none] {$\left(2, 4\right)$};
  \node (node_9) at (161.39bp,166.11bp) [draw,draw=none] {$\left(1, 3\right)$};
  \draw [black,->] (node_0) ..controls (63.757bp,24.218bp) and (46.772bp,36.496bp)  .. (node_1);
  \draw [black,->] (node_0) ..controls (90.951bp,23.136bp) and (95.962bp,33.7bp)  .. (node_2);
  \draw [black,->] (node_0) ..controls (78.102bp,23.136bp) and (73.301bp,33.7bp)  .. (node_3);
  \draw [black,->] (node_0) ..controls (105.32bp,24.218bp) and (122.55bp,36.496bp)  .. (node_6);
  \draw [black,->] (node_1) ..controls (22.361bp,75.815bp) and (28.507bp,86.539bp)  .. (node_4);
  \draw [black,->] (node_2) ..controls (112.99bp,75.589bp) and (116.43bp,85.836bp)  .. (node_7);
  \draw [black,->] (node_3) ..controls (56.523bp,75.589bp) and (52.878bp,85.836bp)  .. (node_4);
  \draw [black,->] (node_4) ..controls (37.102bp,128.34bp) and (32.301bp,138.91bp)  .. (node_5);
  \draw [black,->] (node_4) ..controls (49.951bp,128.34bp) and (54.962bp,138.91bp)  .. (node_10);
  \draw [black,->] (node_5) ..controls (41.023bp,182.03bp) and (58.008bp,194.31bp)  .. (node_11);
  \draw [black,->] (node_6) ..controls (147.1bp,75.891bp) and (140.65bp,86.775bp)  .. (node_7);
  \draw [black,->] (node_7) ..controls (122.43bp,128.12bp) and (120.24bp,138.21bp)  .. (node_8);
  \draw [black,->] (node_7) ..controls (135.39bp,128.57bp) and (143.25bp,139.62bp)  .. (node_9);
  \draw [black,->] (node_8) ..controls (107.83bp,180.95bp) and (102.82bp,191.51bp)  .. (node_11);
  \draw [black,->] (node_9) ..controls (140.46bp,182.03bp) and (123.23bp,194.31bp)  .. (node_11);
  \draw [black,->] (node_10) ..controls (73.678bp,180.95bp) and (78.479bp,191.51bp)  .. (node_11);
\end{tikzpicture} 
}

 & 
\resizebox{!}{3cm}{
 \begin{tikzpicture}[>=latex,line join=bevel,]
\node (node_0) at (37.39bp,7.8452bp) [draw,draw=none] {$\hat{0}$};
  \node (node_1) at (14.39bp,59.992bp) [draw,draw=none] {$\left(3, 2\right)$};
  \node (node_2) at (61.39bp,59.992bp) [draw,draw=none] {$\left(2, 3\right)$};
  \node (node_3) at (37.39bp,112.14bp) [draw,draw=none] {$\hat{1}$};
  \draw [black,->] (node_0) ..controls (31.335bp,22.047bp) and (26.489bp,32.613bp)  .. (node_1);
  \draw [black,->] (node_0) ..controls (43.708bp,22.047bp) and (48.765bp,32.613bp)  .. (node_2);
  \draw [black,->] (node_1) ..controls (20.678bp,74.702bp) and (25.479bp,85.17bp)  .. (node_3);
  \draw [black,->] (node_2) ..controls (54.793bp,74.777bp) and (49.706bp,85.406bp)  .. (node_3);
\end{tikzpicture} 
}

\end{tabular}\end{center}
\end{mexample}

\begin{mexample}
We next consider \(P=C_{2} \times C_{3}\), and tabulate the Hasse diagrams of
\(P\), \(\Ppairs\), and \(\Pqu\).
\phantomsection
\label{}
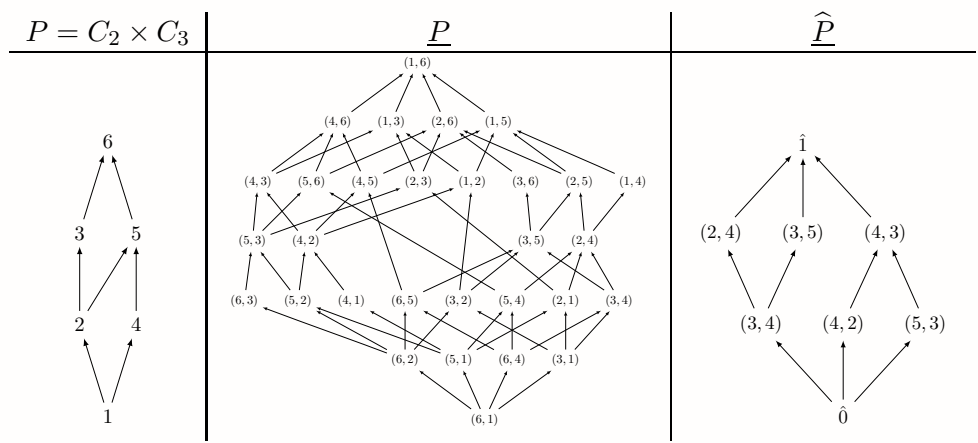
\begin{center}\begin{tabular}{c|c|c}
\(P=C_2 \times C_3\) & \(\Ppairs\) & \(\Pqu\) \\ \hline
\resizebox{!}{4cm}{
 \begin{tikzpicture}[>=latex,line join=bevel,]
\node (node_0) at (20.811bp,6.5307bp) [draw,draw=none] {$1$};
  \node (node_1) at (5.8112bp,55.592bp) [draw,draw=none] {$2$};
  \node (node_3) at (35.811bp,55.592bp) [draw,draw=none] {$4$};
  \node (node_2) at (5.8112bp,104.65bp) [draw,draw=none] {$3$};
  \node (node_4) at (35.811bp,104.65bp) [draw,draw=none] {$5$};
  \node (node_5) at (20.811bp,153.71bp) [draw,draw=none] {$6$};
  \draw [black,->] (node_0) ..controls (17.124bp,19.099bp) and (13.812bp,29.49bp)  .. (node_1);
  \draw [black,->] (node_0) ..controls (24.499bp,19.099bp) and (27.81bp,29.49bp)  .. (node_3);
  \draw [black,->] (node_1) ..controls (5.8112bp,68.09bp) and (5.8112bp,78.313bp)  .. (node_2);
  \draw [black,->] (node_1) ..controls (13.367bp,68.445bp) and (20.424bp,79.516bp)  .. (node_4);
  \draw [black,->] (node_2) ..controls (9.4985bp,117.22bp) and (12.81bp,127.61bp)  .. (node_5);
  \draw [black,->] (node_3) ..controls (35.811bp,68.09bp) and (35.811bp,78.313bp)  .. (node_4);
  \draw [black,->] (node_4) ..controls (32.124bp,117.22bp) and (28.812bp,127.61bp)  .. (node_5);
\end{tikzpicture}
}

   &
\resizebox{!}{5cm}{
 \begin{tikzpicture}[>=latex,line join=bevel,]
\node (node_0) at (225.39bp,8.3018bp) [draw,draw=none] {$\left(6, 1\right)$};
  \node (node_1) at (155.39bp,60.906bp) [draw,draw=none] {$\left(6, 2\right)$};
  \node (node_3) at (249.39bp,60.906bp) [draw,draw=none] {$\left(6, 4\right)$};
  \node (node_5) at (202.39bp,60.906bp) [draw,draw=none] {$\left(5, 1\right)$};
  \node (node_15) at (296.39bp,60.906bp) [draw,draw=none] {$\left(3, 1\right)$};
  \node (node_2) at (14.39bp,113.51bp) [draw,draw=none] {$\left(6, 3\right)$};
  \node (node_4) at (155.39bp,113.51bp) [draw,draw=none] {$\left(6, 5\right)$};
  \node (node_6) at (61.39bp,113.51bp) [draw,draw=none] {$\left(5, 2\right)$};
  \node (node_16) at (202.39bp,113.51bp) [draw,draw=none] {$\left(3, 2\right)$};
  \node (node_7) at (21.39bp,166.11bp) [draw,draw=none] {$\left(5, 3\right)$};
  \node (node_8) at (249.39bp,113.51bp) [draw,draw=none] {$\left(5, 4\right)$};
  \node (node_17) at (343.39bp,113.51bp) [draw,draw=none] {$\left(3, 4\right)$};
  \node (node_13) at (120.39bp,218.72bp) [draw,draw=none] {$\left(4, 5\right)$};
  \node (node_18) at (266.39bp,166.11bp) [draw,draw=none] {$\left(3, 5\right)$};
  \node (node_10) at (108.39bp,113.51bp) [draw,draw=none] {$\left(4, 1\right)$};
  \node (node_20) at (296.39bp,113.51bp) [draw,draw=none] {$\left(2, 1\right)$};
  \node (node_11) at (68.39bp,166.11bp) [draw,draw=none] {$\left(4, 2\right)$};
  \node (node_9) at (73.39bp,218.72bp) [draw,draw=none] {$\left(5, 6\right)$};
  \node (node_12) at (26.39bp,218.72bp) [draw,draw=none] {$\left(4, 3\right)$};
  \node (node_21) at (167.39bp,218.72bp) [draw,draw=none] {$\left(2, 3\right)$};
  \node (node_22) at (313.39bp,166.11bp) [draw,draw=none] {$\left(2, 4\right)$};
  \node (node_14) at (96.39bp,271.32bp) [draw,draw=none] {$\left(4, 6\right)$};
  \node (node_24) at (190.39bp,271.32bp) [draw,draw=none] {$\left(2, 6\right)$};
  \node (node_26) at (214.39bp,218.72bp) [draw,draw=none] {$\left(1, 2\right)$};
  \node (node_27) at (143.39bp,271.32bp) [draw,draw=none] {$\left(1, 3\right)$};
  \node (node_28) at (237.39bp,271.32bp) [draw,draw=none] {$\left(1, 5\right)$};
  \node (node_29) at (166.39bp,323.92bp) [draw,draw=none] {$\left(1, 6\right)$};
  \node (node_19) at (261.39bp,218.72bp) [draw,draw=none] {$\left(3, 6\right)$};
  \node (node_23) at (308.39bp,218.72bp) [draw,draw=none] {$\left(2, 5\right)$};
  \node (node_25) at (355.39bp,218.72bp) [draw,draw=none] {$\left(1, 4\right)$};
  \draw [black,->] (node_0) ..controls (204.76bp,24.218bp) and (187.77bp,36.496bp)  .. (node_1);
  \draw [black,->] (node_0) ..controls (231.95bp,23.136bp) and (236.96bp,33.7bp)  .. (node_3);
  \draw [black,->] (node_0) ..controls (219.1bp,23.136bp) and (214.3bp,33.7bp)  .. (node_5);
  \draw [black,->] (node_0) ..controls (246.32bp,24.218bp) and (263.55bp,36.496bp)  .. (node_15);
  \draw [black,->] (node_1) ..controls (117.24bp,75.599bp) and (69.01bp,92.906bp)  .. (node_2);
  \draw [black,->] (node_1) ..controls (155.39bp,75.437bp) and (155.39bp,85.372bp)  .. (node_4);
  \draw [black,->] (node_1) ..controls (127.05bp,77.16bp) and (102.85bp,90.192bp)  .. (node_6);
  \draw [black,->] (node_1) ..controls (168.66bp,76.193bp) and (179.37bp,87.729bp)  .. (node_16);
  \draw [black,->] (node_2) ..controls (16.272bp,128.12bp) and (17.668bp,138.21bp)  .. (node_7);
  \draw [black,->] (node_3) ..controls (221.05bp,77.16bp) and (196.85bp,90.192bp)  .. (node_4);
  \draw [black,->] (node_3) ..controls (249.39bp,75.437bp) and (249.39bp,85.372bp)  .. (node_8);
  \draw [black,->] (node_3) ..controls (277.73bp,77.16bp) and (301.94bp,90.192bp)  .. (node_17);
  \draw [black,->] (node_4) ..controls (147.56bp,137.59bp) and (134.26bp,176.81bp)  .. (node_13);
  \draw [black,->] (node_4) ..controls (187.61bp,129.2bp) and (219.54bp,143.75bp)  .. (node_18);
  \draw [black,->] (node_5) ..controls (164.24bp,75.599bp) and (116.01bp,92.906bp)  .. (node_6);
  \draw [black,->] (node_5) ..controls (215.66bp,76.193bp) and (226.37bp,87.729bp)  .. (node_8);
  \draw [black,->] (node_5) ..controls (174.05bp,77.16bp) and (149.85bp,90.192bp)  .. (node_10);
  \draw [black,->] (node_5) ..controls (230.73bp,77.16bp) and (254.94bp,90.192bp)  .. (node_20);
  \draw [black,->] (node_6) ..controls (50.216bp,128.65bp) and (41.356bp,139.85bp)  .. (node_7);
  \draw [black,->] (node_6) ..controls (63.272bp,128.12bp) and (64.668bp,138.21bp)  .. (node_11);
  \draw [black,->] (node_7) ..controls (36.227bp,181.55bp) and (48.422bp,193.42bp)  .. (node_9);
  \draw [black,->] (node_7) ..controls (22.727bp,180.64bp) and (23.709bp,190.58bp)  .. (node_12);
  \draw [black,->] (node_7) ..controls (60.469bp,180.66bp) and (111.86bp,198.47bp)  .. (node_21);
  \draw [black,->] (node_8) ..controls (207.53bp,139.06bp) and (132.82bp,182.87bp)  .. (node_9);
  \draw [black,->] (node_8) ..controls (268.16bp,129.35bp) and (283.48bp,141.46bp)  .. (node_22);
  \draw [black,->] (node_9) ..controls (79.678bp,233.55bp) and (84.479bp,244.11bp)  .. (node_14);
  \draw [black,->] (node_9) ..controls (107.23bp,234.35bp) and (142.35bp,249.54bp)  .. (node_24);
  \draw [black,->] (node_10) ..controls (97.216bp,128.65bp) and (88.356bp,139.85bp)  .. (node_11);
  \draw [black,->] (node_11) ..controls (56.595bp,181.32bp) and (47.156bp,192.7bp)  .. (node_12);
  \draw [black,->] (node_11) ..controls (83.227bp,181.55bp) and (95.422bp,193.42bp)  .. (node_13);
  \draw [black,->] (node_11) ..controls (107.47bp,180.66bp) and (158.86bp,198.47bp)  .. (node_26);
  \draw [black,->] (node_12) ..controls (47.023bp,234.63bp) and (64.008bp,246.91bp)  .. (node_14);
  \draw [black,->] (node_12) ..controls (60.227bp,234.35bp) and (95.352bp,249.54bp)  .. (node_27);
  \draw [black,->] (node_13) ..controls (113.83bp,233.55bp) and (108.82bp,244.11bp)  .. (node_14);
  \draw [black,->] (node_13) ..controls (154.23bp,234.35bp) and (189.35bp,249.54bp)  .. (node_28);
  \draw [black,->] (node_14) ..controls (117.02bp,287.24bp) and (134.01bp,299.51bp)  .. (node_29);
  \draw [black,->] (node_15) ..controls (268.05bp,77.16bp) and (243.85bp,90.192bp)  .. (node_16);
  \draw [black,->] (node_15) ..controls (309.66bp,76.193bp) and (320.37bp,87.729bp)  .. (node_17);
  \draw [black,->] (node_15) ..controls (296.39bp,75.437bp) and (296.39bp,85.372bp)  .. (node_20);
  \draw [black,->] (node_16) ..controls (221.16bp,129.35bp) and (236.48bp,141.46bp)  .. (node_18);
  \draw [black,->] (node_16) ..controls (205.06bp,137.44bp) and (209.55bp,176.09bp)  .. (node_26);
  \draw [black,->] (node_17) ..controls (320.46bp,129.58bp) and (301.27bp,142.19bp)  .. (node_18);
  \draw [black,->] (node_17) ..controls (335.1bp,128.49bp) and (328.65bp,139.38bp)  .. (node_22);
  \draw [black,->] (node_18) ..controls (265.05bp,180.64bp) and (264.07bp,190.58bp)  .. (node_19);
  \draw [black,->] (node_18) ..controls (278.19bp,181.32bp) and (287.62bp,192.7bp)  .. (node_23);
  \draw [black,->] (node_19) ..controls (240.46bp,234.63bp) and (223.23bp,246.91bp)  .. (node_24);
  \draw [black,->] (node_20) ..controls (266.44bp,138.47bp) and (213.01bp,181.22bp)  .. (node_21);
  \draw [black,->] (node_20) ..controls (300.99bp,128.19bp) and (304.43bp,138.44bp)  .. (node_22);
  \draw [black,->] (node_21) ..controls (173.68bp,233.55bp) and (178.48bp,244.11bp)  .. (node_24);
  \draw [black,->] (node_21) ..controls (160.83bp,233.55bp) and (155.82bp,244.11bp)  .. (node_27);
  \draw [black,->] (node_22) ..controls (312.05bp,180.64bp) and (311.07bp,190.58bp)  .. (node_23);
  \draw [black,->] (node_22) ..controls (325.19bp,181.32bp) and (334.62bp,192.7bp)  .. (node_25);
  \draw [black,->] (node_23) ..controls (274.89bp,234.08bp) and (239.24bp,249.37bp)  .. (node_24);
  \draw [black,->] (node_23) ..controls (287.46bp,234.63bp) and (270.23bp,246.91bp)  .. (node_28);
  \draw [black,->] (node_24) ..controls (183.83bp,286.15bp) and (178.82bp,296.72bp)  .. (node_29);
  \draw [black,->] (node_25) ..controls (321.89bp,234.08bp) and (286.24bp,249.37bp)  .. (node_28);
  \draw [black,->] (node_26) ..controls (193.46bp,234.63bp) and (176.23bp,246.91bp)  .. (node_27);
  \draw [black,->] (node_26) ..controls (220.68bp,233.55bp) and (225.48bp,244.11bp)  .. (node_28);
  \draw [black,->] (node_27) ..controls (149.68bp,286.15bp) and (154.48bp,296.72bp)  .. (node_29);
  \draw [black,->] (node_28) ..controls (216.46bp,287.24bp) and (199.23bp,299.51bp)  .. (node_29);
\end{tikzpicture}
}

   &
\resizebox{!}{4cm}{
 \begin{tikzpicture}[>=latex,line join=bevel,]
\node (node_0) at (84.39bp,7.8452bp) [draw,draw=none] {$\hat{0}$};
  \node (node_1) at (37.39bp,59.992bp) [draw,draw=none] {$\left(3, 4\right)$};
  \node (node_4) at (84.39bp,59.992bp) [draw,draw=none] {$\left(4, 2\right)$};
  \node (node_5) at (131.39bp,59.992bp) [draw,draw=none] {$\left(5, 3\right)$};
  \node (node_2) at (14.39bp,112.6bp) [draw,draw=none] {$\left(2, 4\right)$};
  \node (node_3) at (61.39bp,112.6bp) [draw,draw=none] {$\left(3, 5\right)$};
  \node (node_7) at (61.39bp,164.74bp) [draw,draw=none] {$\hat{1}$};
  \node (node_6) at (108.39bp,112.6bp) [draw,draw=none] {$\left(4, 3\right)$};
  \draw [black,->] (node_0) ..controls (72.014bp,22.05bp) and (60.921bp,33.886bp)  .. (node_1);
  \draw [black,->] (node_0) ..controls (84.39bp,21.751bp) and (84.39bp,31.682bp)  .. (node_4);
  \draw [black,->] (node_0) ..controls (96.766bp,22.05bp) and (107.86bp,33.886bp)  .. (node_5);
  \draw [black,->] (node_1) ..controls (31.102bp,74.826bp) and (26.301bp,85.391bp)  .. (node_2);
  \draw [black,->] (node_1) ..controls (43.951bp,74.826bp) and (48.962bp,85.391bp)  .. (node_3);
  \draw [black,->] (node_2) ..controls (27.887bp,128.0bp) and (39.102bp,139.96bp)  .. (node_7);
  \draw [black,->] (node_3) ..controls (61.39bp,127.08bp) and (61.39bp,137.07bp)  .. (node_7);
  \draw [black,->] (node_4) ..controls (90.951bp,74.826bp) and (95.962bp,85.391bp)  .. (node_6);
  \draw [black,->] (node_5) ..controls (125.1bp,74.826bp) and (120.3bp,85.391bp)  .. (node_6);
  \draw [black,->] (node_6) ..controls (94.893bp,128.0bp) and (83.679bp,139.96bp)  .. (node_7);
\end{tikzpicture}
}

\end{tabular}\end{center}

The bottom block \(\hat{0}\)
represents all \((x,y) \in \Ppairs\) with \(x > y\), thus \(\pi(x,y)=0\) for
all probability functions \(\pi\). Similarly,
the top block \(\hat{1}\) represents all \((x,y) \in \Ppairs\) with \(x < y\),
thus \(\pi(x,y)=1\) for
all probability functions \(\pi\); recall that \(\pi(x,y)\) represents the likelihood that \(x\) precedes \(y\).

This poset encompasses almost all information needed to parameterize probability functions;
what is missing is the information that \(\pi(x,y) + \pi(y,x)=1\).
We will proceed to amend our constructions to take this final piece of the puzzle into account.
\end{mexample}

\begin{mexample}
Before doing so, we show one more example; \(P=C_{3} \times C_{3}\) with the top and bottom elements removed.
\phantomsection
\label{}
\begin{center}\begin{tabular}{p{1cm}|c}
 \(P\)  & 
\resizebox{!}{2cm}{
 \begin{tikzpicture}[>=latex,line join=bevel,]
\node (node_0) at (20.811bp,6.5307bp) [draw,draw=none] {$3$};
  \node (node_1) at (5.8112bp,55.592bp) [draw,draw=none] {$4$};
  \node (node_4) at (35.811bp,55.592bp) [draw,draw=none] {$5$};
  \node (node_5) at (20.811bp,104.65bp) [draw,draw=none] {$6$};
  \node (node_2) at (50.811bp,6.5307bp) [draw,draw=none] {$1$};
  \node (node_3) at (65.811bp,55.592bp) [draw,draw=none] {$2$};
  \node (node_6) at (57.811bp,104.65bp) [draw,draw=none] {$7$};
  \draw [black,->] (node_0) ..controls (17.124bp,19.099bp) and (13.812bp,29.49bp)  .. (node_1);
  \draw [black,->] (node_0) ..controls (24.499bp,19.099bp) and (27.81bp,29.49bp)  .. (node_4);
  \draw [black,->] (node_1) ..controls (9.4985bp,68.161bp) and (12.81bp,78.552bp)  .. (node_5);
  \draw [black,->] (node_2) ..controls (54.499bp,19.099bp) and (57.81bp,29.49bp)  .. (node_3);
  \draw [black,->] (node_2) ..controls (47.124bp,19.099bp) and (43.812bp,29.49bp)  .. (node_4);
  \draw [black,->] (node_3) ..controls (63.857bp,68.09bp) and (62.119bp,78.313bp)  .. (node_6);
  \draw [black,->] (node_4) ..controls (32.124bp,68.161bp) and (28.812bp,78.552bp)  .. (node_5);
  \draw [black,->] (node_4) ..controls (41.286bp,68.303bp) and (46.301bp,79.031bp)  .. (node_6);
\end{tikzpicture} 
}

   \\  \hline  
 \(\Ppairs\)  & 
\resizebox{!}{3cm}{
 \begin{tikzpicture}[>=latex,line join=bevel,]
\node (node_0) at (95.39bp,8.3018bp) [draw,draw=none] {$\left(7, 3\right)$};
  \node (node_1) at (71.39bp,60.906bp) [draw,draw=none] {$\left(7, 4\right)$};
  \node (node_4) at (212.39bp,60.906bp) [draw,draw=none] {$\left(7, 5\right)$};
  \node (node_12) at (118.39bp,60.906bp) [draw,draw=none] {$\left(5, 3\right)$};
  \node (node_30) at (24.39bp,60.906bp) [draw,draw=none] {$\left(2, 3\right)$};
  \node (node_5) at (155.39bp,113.51bp) [draw,draw=none] {$\left(7, 6\right)$};
  \node (node_13) at (108.39bp,113.51bp) [draw,draw=none] {$\left(5, 4\right)$};
  \node (node_31) at (14.39bp,113.51bp) [draw,draw=none] {$\left(2, 4\right)$};
  \node (node_2) at (306.39bp,8.3018bp) [draw,draw=none] {$\left(7, 1\right)$};
  \node (node_3) at (494.39bp,60.906bp) [draw,draw=none] {$\left(7, 2\right)$};
  \node (node_14) at (400.39bp,60.906bp) [draw,draw=none] {$\left(5, 1\right)$};
  \node (node_32) at (259.39bp,60.906bp) [draw,draw=none] {$\left(2, 1\right)$};
  \node (node_15) at (494.39bp,113.51bp) [draw,draw=none] {$\left(5, 2\right)$};
  \node (node_27) at (398.39bp,166.11bp) [draw,draw=none] {$\left(3, 5\right)$};
  \node (node_33) at (202.39bp,113.51bp) [draw,draw=none] {$\left(2, 5\right)$};
  \node (node_17) at (249.39bp,166.11bp) [draw,draw=none] {$\left(5, 6\right)$};
  \node (node_34) at (155.39bp,166.11bp) [draw,draw=none] {$\left(2, 6\right)$};
  \node (node_6) at (188.39bp,8.3018bp) [draw,draw=none] {$\left(6, 3\right)$};
  \node (node_7) at (165.39bp,60.906bp) [draw,draw=none] {$\left(6, 4\right)$};
  \node (node_10) at (353.39bp,60.906bp) [draw,draw=none] {$\left(6, 5\right)$};
  \node (node_18) at (306.39bp,60.906bp) [draw,draw=none] {$\left(4, 3\right)$};
  \node (node_8) at (423.39bp,8.3018bp) [draw,draw=none] {$\left(6, 1\right)$};
  \node (node_9) at (541.39bp,60.906bp) [draw,draw=none] {$\left(6, 2\right)$};
  \node (node_19) at (447.39bp,60.906bp) [draw,draw=none] {$\left(4, 1\right)$};
  \node (node_11) at (541.39bp,113.51bp) [draw,draw=none] {$\left(6, 7\right)$};
  \node (node_20) at (588.39bp,113.51bp) [draw,draw=none] {$\left(4, 2\right)$};
  \node (node_21) at (400.39bp,113.51bp) [draw,draw=none] {$\left(4, 5\right)$};
  \node (node_38) at (202.39bp,166.11bp) [draw,draw=none] {$\left(1, 5\right)$};
  \node (node_16) at (492.39bp,166.11bp) [draw,draw=none] {$\left(5, 7\right)$};
  \node (node_23) at (539.39bp,166.11bp) [draw,draw=none] {$\left(4, 7\right)$};
  \node (node_36) at (61.39bp,113.51bp) [draw,draw=none] {$\left(1, 3\right)$};
  \node (node_24) at (108.39bp,166.11bp) [draw,draw=none] {$\left(3, 4\right)$};
  \node (node_37) at (61.39bp,166.11bp) [draw,draw=none] {$\left(1, 4\right)$};
  \node (node_25) at (447.39bp,113.51bp) [draw,draw=none] {$\left(3, 1\right)$};
  \node (node_26) at (586.39bp,166.11bp) [draw,draw=none] {$\left(3, 2\right)$};
  \node (node_40) at (445.39bp,166.11bp) [draw,draw=none] {$\left(1, 2\right)$};
  \node (node_29) at (515.39bp,218.72bp) [draw,draw=none] {$\left(3, 7\right)$};
  \node (node_41) at (397.39bp,218.72bp) [draw,draw=none] {$\left(1, 7\right)$};
  \node (node_28) at (260.39bp,218.72bp) [draw,draw=none] {$\left(3, 6\right)$};
  \node (node_39) at (190.39bp,218.72bp) [draw,draw=none] {$\left(1, 6\right)$};
  \node (node_22) at (301.39bp,166.11bp) [draw,draw=none] {$\left(4, 6\right)$};
  \node (node_35) at (350.39bp,166.11bp) [draw,draw=none] {$\left(2, 7\right)$};
  \draw [black,->] (node_0) ..controls (88.829bp,23.136bp) and (83.819bp,33.7bp)  .. (node_1);
  \draw [black,->] (node_0) ..controls (129.23bp,23.937bp) and (164.35bp,39.128bp)  .. (node_4);
  \draw [black,->] (node_0) ..controls (101.68bp,23.136bp) and (106.48bp,33.7bp)  .. (node_12);
  \draw [black,->] (node_0) ..controls (74.462bp,24.218bp) and (57.235bp,36.496bp)  .. (node_30);
  \draw [black,->] (node_1) ..controls (96.524bp,77.047bp) and (117.74bp,89.826bp)  .. (node_5);
  \draw [black,->] (node_1) ..controls (81.726bp,76.042bp) and (89.922bp,87.25bp)  .. (node_13);
  \draw [black,->] (node_1) ..controls (55.042bp,76.419bp) and (41.486bp,88.454bp)  .. (node_31);
  \draw [black,->] (node_2) ..controls (347.3bp,19.596bp) and (407.77bp,35.03bp)  .. (node_3);
  \draw [black,->] (node_2) ..controls (278.05bp,24.556bp) and (253.85bp,37.589bp)  .. (node_4);
  \draw [black,->] (node_2) ..controls (334.73bp,24.556bp) and (358.94bp,37.589bp)  .. (node_14);
  \draw [black,->] (node_2) ..controls (293.12bp,23.589bp) and (282.41bp,35.125bp)  .. (node_32);
  \draw [black,->] (node_3) ..controls (494.39bp,75.437bp) and (494.39bp,85.372bp)  .. (node_15);
  \draw [black,->] (node_4) ..controls (196.04bp,76.419bp) and (182.49bp,88.454bp)  .. (node_5);
  \draw [black,->] (node_4) ..controls (256.76bp,86.526bp) and (336.26bp,130.64bp)  .. (node_27);
  \draw [black,->] (node_4) ..controls (209.7bp,75.513bp) and (207.71bp,85.604bp)  .. (node_33);
  \draw [black,->] (node_5) ..controls (183.73bp,129.76bp) and (207.94bp,142.8bp)  .. (node_17);
  \draw [black,->] (node_5) ..controls (155.39bp,128.04bp) and (155.39bp,137.98bp)  .. (node_34);
  \draw [black,->] (node_6) ..controls (182.1bp,23.136bp) and (177.3bp,33.7bp)  .. (node_7);
  \draw [black,->] (node_6) ..controls (230.92bp,22.346bp) and (293.76bp,41.618bp)  .. (node_10);
  \draw [black,->] (node_6) ..controls (167.76bp,24.218bp) and (150.77bp,36.496bp)  .. (node_12);
  \draw [black,->] (node_6) ..controls (221.89bp,23.668bp) and (257.54bp,38.955bp)  .. (node_18);
  \draw [black,->] (node_7) ..controls (149.04bp,76.419bp) and (135.49bp,88.454bp)  .. (node_13);
  \draw [black,->] (node_8) ..controls (456.89bp,23.668bp) and (492.54bp,38.955bp)  .. (node_9);
  \draw [black,->] (node_8) ..controls (402.76bp,24.218bp) and (385.77bp,36.496bp)  .. (node_10);
  \draw [black,->] (node_8) ..controls (417.1bp,23.136bp) and (412.3bp,33.7bp)  .. (node_14);
  \draw [black,->] (node_8) ..controls (429.95bp,23.136bp) and (434.96bp,33.7bp)  .. (node_19);
  \draw [black,->] (node_9) ..controls (541.39bp,75.437bp) and (541.39bp,85.372bp)  .. (node_11);
  \draw [black,->] (node_9) ..controls (528.12bp,76.193bp) and (517.41bp,87.729bp)  .. (node_15);
  \draw [black,->] (node_9) ..controls (554.66bp,76.193bp) and (565.37bp,87.729bp)  .. (node_20);
  \draw [black,->] (node_10) ..controls (370.79bp,67.307bp) and (374.2bp,68.313bp)  .. (377.39bp,69.207bp) .. controls (435.99bp,85.646bp) and (453.69bp,87.672bp)  .. (node_11);
  \draw [black,->] (node_10) ..controls (366.66bp,76.193bp) and (377.37bp,87.729bp)  .. (node_21);
  \draw [black,->] (node_10) ..controls (318.12bp,86.013bp) and (254.69bp,129.37bp)  .. (node_38);
  \draw [black,->] (node_11) ..controls (527.48bp,128.87bp) and (516.15bp,140.57bp)  .. (node_16);
  \draw [black,->] (node_11) ..controls (540.86bp,128.04bp) and (540.46bp,137.98bp)  .. (node_23);
  \draw [black,->] (node_12) ..controls (115.7bp,75.513bp) and (113.71bp,85.604bp)  .. (node_13);
  \draw [black,->] (node_12) ..controls (135.8bp,67.276bp) and (139.21bp,68.292bp)  .. (142.39bp,69.207bp) .. controls (283.29bp,109.78bp) and (322.29bp,115.99bp)  .. (node_16);
  \draw [black,->] (node_12) ..controls (102.04bp,76.419bp) and (88.486bp,88.454bp)  .. (node_36);
  \draw [black,->] (node_13) ..controls (146.54bp,128.2bp) and (194.77bp,145.51bp)  .. (node_17);
  \draw [black,->] (node_13) ..controls (108.39bp,128.04bp) and (108.39bp,137.98bp)  .. (node_24);
  \draw [black,->] (node_13) ..controls (95.12bp,128.8bp) and (84.405bp,140.33bp)  .. (node_37);
  \draw [black,->] (node_14) ..controls (428.73bp,77.16bp) and (452.94bp,90.192bp)  .. (node_15);
  \draw [black,->] (node_14) ..controls (365.12bp,86.013bp) and (301.69bp,129.37bp)  .. (node_17);
  \draw [black,->] (node_14) ..controls (413.66bp,76.193bp) and (424.37bp,87.729bp)  .. (node_25);
  \draw [black,->] (node_15) ..controls (493.86bp,128.04bp) and (493.46bp,137.98bp)  .. (node_16);
  \draw [black,->] (node_15) ..controls (522.05bp,129.73bp) and (545.59bp,142.67bp)  .. (node_26);
  \draw [black,->] (node_15) ..controls (480.48bp,128.87bp) and (469.15bp,140.57bp)  .. (node_40);
  \draw [black,->] (node_16) ..controls (498.68bp,180.95bp) and (503.48bp,191.51bp)  .. (node_29);
  \draw [black,->] (node_16) ..controls (463.75bp,182.37bp) and (439.29bp,195.4bp)  .. (node_41);
  \draw [black,->] (node_17) ..controls (252.35bp,180.72bp) and (254.54bp,190.81bp)  .. (node_28);
  \draw [black,->] (node_17) ..controls (232.26bp,181.8bp) and (218.52bp,193.59bp)  .. (node_39);
  \draw [black,->] (node_18) ..controls (334.73bp,77.16bp) and (358.94bp,90.192bp)  .. (node_21);
  \draw [black,->] (node_19) ..controls (485.54bp,75.599bp) and (533.77bp,92.906bp)  .. (node_20);
  \draw [black,->] (node_19) ..controls (434.12bp,76.193bp) and (423.41bp,87.729bp)  .. (node_21);
  \draw [black,->] (node_19) ..controls (447.39bp,75.437bp) and (447.39bp,85.372bp)  .. (node_25);
  \draw [black,->] (node_20) ..controls (574.48bp,128.87bp) and (563.15bp,140.57bp)  .. (node_23);
  \draw [black,->] (node_20) ..controls (587.86bp,128.04bp) and (587.46bp,137.98bp)  .. (node_26);
  \draw [black,->] (node_21) ..controls (370.4bp,129.84bp) and (344.57bp,143.04bp)  .. (node_22);
  \draw [black,->] (node_21) ..controls (438.38bp,128.34bp) and (485.77bp,145.59bp)  .. (node_23);
  \draw [black,->] (node_21) ..controls (399.86bp,128.04bp) and (399.46bp,137.98bp)  .. (node_27);
  \draw [black,->] (node_22) ..controls (289.88bp,181.32bp) and (280.66bp,192.7bp)  .. (node_28);
  \draw [black,->] (node_23) ..controls (532.83bp,180.95bp) and (527.82bp,191.51bp)  .. (node_29);
  \draw [black,->] (node_24) ..controls (148.49bp,180.46bp) and (203.13bp,198.65bp)  .. (node_28);
  \draw [black,->] (node_25) ..controls (485.38bp,128.34bp) and (532.77bp,145.59bp)  .. (node_26);
  \draw [black,->] (node_25) ..controls (433.48bp,128.87bp) and (422.15bp,140.57bp)  .. (node_27);
  \draw [black,->] (node_26) ..controls (565.46bp,182.03bp) and (548.23bp,194.31bp)  .. (node_29);
  \draw [black,->] (node_27) ..controls (360.8bp,180.9bp) and (314.11bp,198.02bp)  .. (node_28);
  \draw [black,->] (node_27) ..controls (432.23bp,181.75bp) and (467.35bp,196.94bp)  .. (node_29);
  \draw [black,->] (node_30) ..controls (21.701bp,75.513bp) and (19.707bp,85.604bp)  .. (node_31);
  \draw [black,->] (node_30) ..controls (41.805bp,67.265bp) and (45.212bp,68.284bp)  .. (48.39bp,69.207bp) .. controls (102.7bp,84.986bp) and (119.0bp,88.147bp)  .. (node_33);
  \draw [black,->] (node_30) ..controls (34.726bp,76.042bp) and (42.922bp,87.25bp)  .. (node_36);
  \draw [black,->] (node_31) ..controls (52.544bp,128.2bp) and (100.77bp,145.51bp)  .. (node_34);
  \draw [black,->] (node_31) ..controls (27.66bp,128.8bp) and (38.375bp,140.33bp)  .. (node_37);
  \draw [black,->] (node_32) ..controls (243.04bp,76.419bp) and (229.49bp,88.454bp)  .. (node_33);
  \draw [black,->] (node_33) ..controls (189.12bp,128.8bp) and (178.41bp,140.33bp)  .. (node_34);
  \draw [black,->] (node_33) ..controls (241.6bp,127.92bp) and (293.87bp,145.79bp)  .. (node_35);
  \draw [black,->] (node_33) ..controls (202.39bp,128.04bp) and (202.39bp,137.98bp)  .. (node_38);
  \draw [black,->] (node_34) ..controls (165.12bp,181.17bp) and (172.76bp,192.22bp)  .. (node_39);
  \draw [black,->] (node_35) ..controls (363.66bp,181.4bp) and (374.37bp,192.94bp)  .. (node_41);
  \draw [black,->] (node_36) ..controls (61.39bp,128.04bp) and (61.39bp,137.98bp)  .. (node_37);
  \draw [black,->] (node_36) ..controls (99.544bp,128.2bp) and (147.77bp,145.51bp)  .. (node_38);
  \draw [black,->] (node_37) ..controls (97.385bp,181.23bp) and (138.92bp,197.52bp)  .. (node_39);
  \draw [black,->] (node_38) ..controls (199.16bp,180.72bp) and (196.77bp,190.81bp)  .. (node_39);
  \draw [black,->] (node_38) ..controls (219.79bp,172.52bp) and (223.2bp,173.53bp)  .. (226.39bp,174.41bp) .. controls (277.89bp,188.76bp) and (338.66bp,203.68bp)  .. (node_41);
  \draw [black,->] (node_40) ..controls (431.77bp,181.48bp) and (420.67bp,193.18bp)  .. (node_41);
\end{tikzpicture} 
}

   \\  \hline 
 \(\Pqu\)  & 
\resizebox{!}{3cm}{
 \begin{tikzpicture}[>=latex,line join=bevel,]
\node (node_0) at (213.39bp,7.8452bp) [draw,draw=none] {$\hat{0}$};
  \node (node_1) at (272.39bp,59.992bp) [draw,draw=none] {$\left(2, 3\right)$};
  \node (node_4) at (96.39bp,59.992bp) [draw,draw=none] {$\left(4, 1\right)$};
  \node (node_7) at (166.39bp,59.992bp) [draw,draw=none] {$\left(6, 2\right)$};
  \node (node_13) at (366.39bp,59.992bp) [draw,draw=none] {$\left(7, 4\right)$};
  \node (node_2) at (249.39bp,112.6bp) [draw,draw=none] {$\left(1, 3\right)$};
  \node (node_3) at (343.39bp,112.6bp) [draw,draw=none] {$\left(2, 5\right)$};
  \node (node_14) at (296.39bp,112.6bp) [draw,draw=none] {$\left(2, 4\right)$};
  \node (node_16) at (260.39bp,165.2bp) [draw,draw=none] {$\left(1, 4\right)$};
  \node (node_18) at (319.39bp,165.2bp) [draw,draw=none] {$\left(2, 6\right)$};
  \node (node_5) at (14.39bp,112.6bp) [draw,draw=none] {$\left(3, 1\right)$};
  \node (node_6) at (108.39bp,112.6bp) [draw,draw=none] {$\left(4, 5\right)$};
  \node (node_8) at (61.39bp,112.6bp) [draw,draw=none] {$\left(4, 2\right)$};
  \node (node_10) at (96.39bp,165.2bp) [draw,draw=none] {$\left(3, 2\right)$};
  \node (node_12) at (143.39bp,165.2bp) [draw,draw=none] {$\left(4, 7\right)$};
  \node (node_9) at (155.39bp,112.6bp) [draw,draw=none] {$\left(5, 2\right)$};
  \node (node_11) at (202.39bp,112.6bp) [draw,draw=none] {$\left(6, 7\right)$};
  \node (node_19) at (213.39bp,217.35bp) [draw,draw=none] {$\hat{1}$};
  \node (node_15) at (390.39bp,112.6bp) [draw,draw=none] {$\left(5, 4\right)$};
  \node (node_17) at (437.39bp,112.6bp) [draw,draw=none] {$\left(7, 6\right)$};
  \draw [black,->] (node_0) ..controls (227.11bp,20.506bp) and (243.2bp,34.18bp)  .. (node_1);
  \draw [black,->] (node_0) ..controls (192.34bp,17.867bp) and (148.86bp,36.504bp)  .. (node_4);
  \draw [black,->] (node_0) ..controls (201.01bp,22.05bp) and (189.92bp,33.886bp)  .. (node_7);
  \draw [black,->] (node_0) ..controls (239.1bp,17.272bp) and (305.28bp,38.963bp)  .. (node_13);
  \draw [black,->] (node_1) ..controls (266.1bp,74.826bp) and (261.3bp,85.391bp)  .. (node_2);
  \draw [black,->] (node_1) ..controls (293.32bp,75.908bp) and (310.55bp,88.187bp)  .. (node_3);
  \draw [black,->] (node_1) ..controls (278.95bp,74.826bp) and (283.96bp,85.391bp)  .. (node_14);
  \draw [black,->] (node_2) ..controls (252.35bp,127.2bp) and (254.54bp,137.29bp)  .. (node_16);
  \draw [black,->] (node_3) ..controls (336.83bp,127.43bp) and (331.82bp,137.99bp)  .. (node_18);
  \draw [black,->] (node_4) ..controls (71.854bp,76.134bp) and (51.146bp,88.913bp)  .. (node_5);
  \draw [black,->] (node_4) ..controls (99.617bp,74.6bp) and (102.01bp,84.691bp)  .. (node_6);
  \draw [black,->] (node_4) ..controls (86.665bp,75.053bp) and (79.025bp,86.099bp)  .. (node_8);
  \draw [black,->] (node_5) ..controls (38.926bp,128.74bp) and (59.634bp,141.52bp)  .. (node_10);
  \draw [black,->] (node_6) ..controls (118.12bp,127.66bp) and (125.76bp,138.7bp)  .. (node_12);
  \draw [black,->] (node_7) ..controls (135.24bp,76.004bp) and (106.48bp,89.865bp)  .. (node_8);
  \draw [black,->] (node_7) ..controls (163.43bp,74.6bp) and (161.24bp,84.691bp)  .. (node_9);
  \draw [black,->] (node_7) ..controls (176.39bp,75.053bp) and (184.25bp,86.099bp)  .. (node_11);
  \draw [black,->] (node_8) ..controls (71.115bp,127.66bp) and (78.755bp,138.7bp)  .. (node_10);
  \draw [black,->] (node_8) ..controls (85.926bp,128.74bp) and (106.63bp,141.52bp)  .. (node_12);
  \draw [black,->] (node_9) ..controls (138.26bp,128.29bp) and (124.52bp,140.07bp)  .. (node_10);
  \draw [black,->] (node_10) ..controls (132.67bp,181.75bp) and (174.39bp,199.63bp)  .. (node_19);
  \draw [black,->] (node_11) ..controls (185.26bp,128.29bp) and (171.52bp,140.07bp)  .. (node_12);
  \draw [black,->] (node_12) ..controls (165.23bp,181.85bp) and (184.96bp,195.98bp)  .. (node_19);
  \draw [black,->] (node_13) ..controls (345.76bp,75.908bp) and (328.77bp,88.187bp)  .. (node_14);
  \draw [black,->] (node_13) ..controls (372.95bp,74.826bp) and (377.96bp,85.391bp)  .. (node_15);
  \draw [black,->] (node_13) ..controls (387.32bp,75.908bp) and (404.55bp,88.187bp)  .. (node_17);
  \draw [black,->] (node_14) ..controls (286.39bp,127.66bp) and (278.53bp,138.7bp)  .. (node_16);
  \draw [black,->] (node_14) ..controls (302.68bp,127.43bp) and (307.48bp,137.99bp)  .. (node_18);
  \draw [black,->] (node_15) ..controls (354.35bp,127.63bp) and (312.03bp,144.1bp)  .. (node_16);
  \draw [black,->] (node_16) ..controls (246.89bp,180.6bp) and (235.68bp,192.57bp)  .. (node_19);
  \draw [black,->] (node_17) ..controls (403.89bp,127.96bp) and (368.24bp,143.25bp)  .. (node_18);
  \draw [black,->] (node_18) ..controls (285.35bp,182.31bp) and (249.84bp,199.1bp)  .. (node_19);
\end{tikzpicture} 
}

\end{tabular}\end{center}
\end{mexample}
\subsection{Ordering antichains}
\label{sec:org01df65e}
\begin{definition}
Let
\[\twoanti{P} = \{(x,y) \in P \times P \mid x \parallel y\}\]
and define the following binary relations on \(\twoanti{P}\):
\begin{equation}
\label{eq-trel}
  \begin{split}
    (x,y) & \tarel (u,v) \iff x = u \text{ and } y \le v \\
    (x,y) & \tbrel (u,v) \iff y = v \text{ and } u \le x \\
    (x,y) & \terel (u,v) \iff (x,y) \tarel (u,v) \text{ or } (x,y) \tbrel (u,v)
  \end{split}
\end{equation}
\label{def-rel-antichains}
\end{definition}

\begin{theorem}
The transitive closure \(\transterel\) of the relation \(\terel\) is the restriction of \(\transerel\) to \(\twoanti{P}\);
thus \((\twoanti{P}, \transterel)\) is the induced subposet of \((\Ppairs, \transerel)\).
Furthermore, when \(P\) is neither a chain nor an antichain,
\((\twoanti{P},  \transterel)\) is isomorphic to \((\Pqu,\pqurel)\) with the top and bottom elements removed.
\label{thm-induced}
\end{theorem}
\begin{proof}
The ``bottom block'' of \(\Ppairs\) consists of ordered pairs \((x,y)\) with \(x \le y\),
the ``top block''  consists of ordered pairs \((x,y)\) with \(x \ge y\),
and the ``middle block''  consists of ordered pairs \((x,y)\) with \(x \parallel y\).

In \(\twoanti{P}\) there are only the elements of the middle block, with the induced order.

In \(\Pqu\) there are the elements of the middle block, with the induced order,
as well as the class of the top block, an element which is above all elements in the middle block,
and also the class of the bottom block, which is below all elements of the middle block.
The ordering of the elements in the middle block is once again induced from \(\Ppairs\).

Hence \(\twoanti{P}\) with \(\hat{0},\hat{1}\) adjoined is isomorphic to \(\Pqu\);
equivalently,
\[\Pqu \setminus \{\hat{0}, \hat{1}\} \simeq \twoanti{P}.\]
\end{proof}
\subsubsection{Examples}
\label{sec:org304a1f8}
\begin{mexample}
We continue the previous example, where \(P=C_{3} \times C_{3}\) with the top and bottom elements removed.
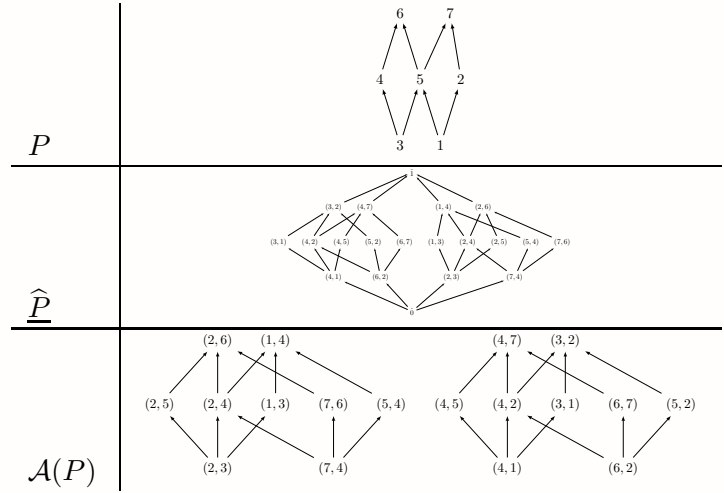
\phantomsection
\label{}
\begin{center}\begin{tabular}{p{1cm}|c}
 \(P\)  & 
\resizebox{!}{2cm}{
 \begin{tikzpicture}[>=latex,line join=bevel,]
\node (node_0) at (20.811bp,6.5307bp) [draw,draw=none] {$3$};
  \node (node_1) at (5.8112bp,55.592bp) [draw,draw=none] {$4$};
  \node (node_4) at (35.811bp,55.592bp) [draw,draw=none] {$5$};
  \node (node_5) at (20.811bp,104.65bp) [draw,draw=none] {$6$};
  \node (node_2) at (50.811bp,6.5307bp) [draw,draw=none] {$1$};
  \node (node_3) at (65.811bp,55.592bp) [draw,draw=none] {$2$};
  \node (node_6) at (57.811bp,104.65bp) [draw,draw=none] {$7$};
  \draw [black,->] (node_0) ..controls (17.124bp,19.099bp) and (13.812bp,29.49bp)  .. (node_1);
  \draw [black,->] (node_0) ..controls (24.499bp,19.099bp) and (27.81bp,29.49bp)  .. (node_4);
  \draw [black,->] (node_1) ..controls (9.4985bp,68.161bp) and (12.81bp,78.552bp)  .. (node_5);
  \draw [black,->] (node_2) ..controls (54.499bp,19.099bp) and (57.81bp,29.49bp)  .. (node_3);
  \draw [black,->] (node_2) ..controls (47.124bp,19.099bp) and (43.812bp,29.49bp)  .. (node_4);
  \draw [black,->] (node_3) ..controls (63.857bp,68.09bp) and (62.119bp,78.313bp)  .. (node_6);
  \draw [black,->] (node_4) ..controls (32.124bp,68.161bp) and (28.812bp,78.552bp)  .. (node_5);
  \draw [black,->] (node_4) ..controls (41.286bp,68.303bp) and (46.301bp,79.031bp)  .. (node_6);
\end{tikzpicture} 
}

   \\  \hline  
 \(\Pqu\)  & 
\resizebox{!}{2cm}{
 \begin{tikzpicture}[>=latex,line join=bevel,]
\node (node_0) at (213.39bp,7.8452bp) [draw,draw=none] {$\hat{0}$};
  \node (node_1) at (272.39bp,59.992bp) [draw,draw=none] {$\left(2, 3\right)$};
  \node (node_4) at (96.39bp,59.992bp) [draw,draw=none] {$\left(4, 1\right)$};
  \node (node_7) at (166.39bp,59.992bp) [draw,draw=none] {$\left(6, 2\right)$};
  \node (node_13) at (366.39bp,59.992bp) [draw,draw=none] {$\left(7, 4\right)$};
  \node (node_2) at (249.39bp,112.6bp) [draw,draw=none] {$\left(1, 3\right)$};
  \node (node_3) at (343.39bp,112.6bp) [draw,draw=none] {$\left(2, 5\right)$};
  \node (node_14) at (296.39bp,112.6bp) [draw,draw=none] {$\left(2, 4\right)$};
  \node (node_16) at (260.39bp,165.2bp) [draw,draw=none] {$\left(1, 4\right)$};
  \node (node_18) at (319.39bp,165.2bp) [draw,draw=none] {$\left(2, 6\right)$};
  \node (node_5) at (14.39bp,112.6bp) [draw,draw=none] {$\left(3, 1\right)$};
  \node (node_6) at (108.39bp,112.6bp) [draw,draw=none] {$\left(4, 5\right)$};
  \node (node_8) at (61.39bp,112.6bp) [draw,draw=none] {$\left(4, 2\right)$};
  \node (node_10) at (96.39bp,165.2bp) [draw,draw=none] {$\left(3, 2\right)$};
  \node (node_12) at (143.39bp,165.2bp) [draw,draw=none] {$\left(4, 7\right)$};
  \node (node_9) at (155.39bp,112.6bp) [draw,draw=none] {$\left(5, 2\right)$};
  \node (node_11) at (202.39bp,112.6bp) [draw,draw=none] {$\left(6, 7\right)$};
  \node (node_19) at (213.39bp,217.35bp) [draw,draw=none] {$\hat{1}$};
  \node (node_15) at (390.39bp,112.6bp) [draw,draw=none] {$\left(5, 4\right)$};
  \node (node_17) at (437.39bp,112.6bp) [draw,draw=none] {$\left(7, 6\right)$};
  \draw [black,->] (node_0) ..controls (227.11bp,20.506bp) and (243.2bp,34.18bp)  .. (node_1);
  \draw [black,->] (node_0) ..controls (192.34bp,17.867bp) and (148.86bp,36.504bp)  .. (node_4);
  \draw [black,->] (node_0) ..controls (201.01bp,22.05bp) and (189.92bp,33.886bp)  .. (node_7);
  \draw [black,->] (node_0) ..controls (239.1bp,17.272bp) and (305.28bp,38.963bp)  .. (node_13);
  \draw [black,->] (node_1) ..controls (266.1bp,74.826bp) and (261.3bp,85.391bp)  .. (node_2);
  \draw [black,->] (node_1) ..controls (293.32bp,75.908bp) and (310.55bp,88.187bp)  .. (node_3);
  \draw [black,->] (node_1) ..controls (278.95bp,74.826bp) and (283.96bp,85.391bp)  .. (node_14);
  \draw [black,->] (node_2) ..controls (252.35bp,127.2bp) and (254.54bp,137.29bp)  .. (node_16);
  \draw [black,->] (node_3) ..controls (336.83bp,127.43bp) and (331.82bp,137.99bp)  .. (node_18);
  \draw [black,->] (node_4) ..controls (71.854bp,76.134bp) and (51.146bp,88.913bp)  .. (node_5);
  \draw [black,->] (node_4) ..controls (99.617bp,74.6bp) and (102.01bp,84.691bp)  .. (node_6);
  \draw [black,->] (node_4) ..controls (86.665bp,75.053bp) and (79.025bp,86.099bp)  .. (node_8);
  \draw [black,->] (node_5) ..controls (38.926bp,128.74bp) and (59.634bp,141.52bp)  .. (node_10);
  \draw [black,->] (node_6) ..controls (118.12bp,127.66bp) and (125.76bp,138.7bp)  .. (node_12);
  \draw [black,->] (node_7) ..controls (135.24bp,76.004bp) and (106.48bp,89.865bp)  .. (node_8);
  \draw [black,->] (node_7) ..controls (163.43bp,74.6bp) and (161.24bp,84.691bp)  .. (node_9);
  \draw [black,->] (node_7) ..controls (176.39bp,75.053bp) and (184.25bp,86.099bp)  .. (node_11);
  \draw [black,->] (node_8) ..controls (71.115bp,127.66bp) and (78.755bp,138.7bp)  .. (node_10);
  \draw [black,->] (node_8) ..controls (85.926bp,128.74bp) and (106.63bp,141.52bp)  .. (node_12);
  \draw [black,->] (node_9) ..controls (138.26bp,128.29bp) and (124.52bp,140.07bp)  .. (node_10);
  \draw [black,->] (node_10) ..controls (132.67bp,181.75bp) and (174.39bp,199.63bp)  .. (node_19);
  \draw [black,->] (node_11) ..controls (185.26bp,128.29bp) and (171.52bp,140.07bp)  .. (node_12);
  \draw [black,->] (node_12) ..controls (165.23bp,181.85bp) and (184.96bp,195.98bp)  .. (node_19);
  \draw [black,->] (node_13) ..controls (345.76bp,75.908bp) and (328.77bp,88.187bp)  .. (node_14);
  \draw [black,->] (node_13) ..controls (372.95bp,74.826bp) and (377.96bp,85.391bp)  .. (node_15);
  \draw [black,->] (node_13) ..controls (387.32bp,75.908bp) and (404.55bp,88.187bp)  .. (node_17);
  \draw [black,->] (node_14) ..controls (286.39bp,127.66bp) and (278.53bp,138.7bp)  .. (node_16);
  \draw [black,->] (node_14) ..controls (302.68bp,127.43bp) and (307.48bp,137.99bp)  .. (node_18);
  \draw [black,->] (node_15) ..controls (354.35bp,127.63bp) and (312.03bp,144.1bp)  .. (node_16);
  \draw [black,->] (node_16) ..controls (246.89bp,180.6bp) and (235.68bp,192.57bp)  .. (node_19);
  \draw [black,->] (node_17) ..controls (403.89bp,127.96bp) and (368.24bp,143.25bp)  .. (node_18);
  \draw [black,->] (node_18) ..controls (285.35bp,182.31bp) and (249.84bp,199.1bp)  .. (node_19);
\end{tikzpicture} 
}

   \\  \hline  
 \(\twoanti{P}\)  & 
\resizebox{!}{2cm}{
 \begin{tikzpicture}[>=latex,line join=bevel,]
\node (node_0) at (155.39bp,8.3018bp) [draw,draw=none] {$\left(7, 4\right)$};
  \node (node_1) at (155.39bp,60.906bp) [draw,draw=none] {$\left(7, 6\right)$};
  \node (node_4) at (202.39bp,60.906bp) [draw,draw=none] {$\left(5, 4\right)$};
  \node (node_13) at (61.39bp,60.906bp) [draw,draw=none] {$\left(2, 4\right)$};
  \node (node_15) at (61.39bp,113.51bp) [draw,draw=none] {$\left(2, 6\right)$};
  \node (node_2) at (390.39bp,8.3018bp) [draw,draw=none] {$\left(6, 2\right)$};
  \node (node_3) at (390.39bp,60.906bp) [draw,draw=none] {$\left(6, 7\right)$};
  \node (node_5) at (437.39bp,60.906bp) [draw,draw=none] {$\left(5, 2\right)$};
  \node (node_7) at (296.39bp,60.906bp) [draw,draw=none] {$\left(4, 2\right)$};
  \node (node_9) at (296.39bp,113.51bp) [draw,draw=none] {$\left(4, 7\right)$};
  \node (node_17) at (108.39bp,113.51bp) [draw,draw=none] {$\left(1, 4\right)$};
  \node (node_11) at (343.39bp,113.51bp) [draw,draw=none] {$\left(3, 2\right)$};
  \node (node_6) at (296.39bp,8.3018bp) [draw,draw=none] {$\left(4, 1\right)$};
  \node (node_8) at (249.39bp,60.906bp) [draw,draw=none] {$\left(4, 5\right)$};
  \node (node_10) at (343.39bp,60.906bp) [draw,draw=none] {$\left(3, 1\right)$};
  \node (node_12) at (61.39bp,8.3018bp) [draw,draw=none] {$\left(2, 3\right)$};
  \node (node_14) at (14.39bp,60.906bp) [draw,draw=none] {$\left(2, 5\right)$};
  \node (node_16) at (108.39bp,60.906bp) [draw,draw=none] {$\left(1, 3\right)$};
  \draw [black,->] (node_0) ..controls (155.39bp,22.834bp) and (155.39bp,32.769bp)  .. (node_1);
  \draw [black,->] (node_0) ..controls (168.66bp,23.589bp) and (179.37bp,35.125bp)  .. (node_4);
  \draw [black,->] (node_0) ..controls (127.05bp,24.556bp) and (102.85bp,37.589bp)  .. (node_13);
  \draw [black,->] (node_1) ..controls (127.05bp,77.16bp) and (102.85bp,90.192bp)  .. (node_15);
  \draw [black,->] (node_2) ..controls (390.39bp,22.834bp) and (390.39bp,32.769bp)  .. (node_3);
  \draw [black,->] (node_2) ..controls (403.66bp,23.589bp) and (414.37bp,35.125bp)  .. (node_5);
  \draw [black,->] (node_2) ..controls (362.05bp,24.556bp) and (337.85bp,37.589bp)  .. (node_7);
  \draw [black,->] (node_3) ..controls (362.05bp,77.16bp) and (337.85bp,90.192bp)  .. (node_9);
  \draw [black,->] (node_4) ..controls (174.05bp,77.16bp) and (149.85bp,90.192bp)  .. (node_17);
  \draw [black,->] (node_5) ..controls (409.05bp,77.16bp) and (384.85bp,90.192bp)  .. (node_11);
  \draw [black,->] (node_6) ..controls (296.39bp,22.834bp) and (296.39bp,32.769bp)  .. (node_7);
  \draw [black,->] (node_6) ..controls (283.12bp,23.589bp) and (272.41bp,35.125bp)  .. (node_8);
  \draw [black,->] (node_6) ..controls (309.66bp,23.589bp) and (320.37bp,35.125bp)  .. (node_10);
  \draw [black,->] (node_7) ..controls (296.39bp,75.437bp) and (296.39bp,85.372bp)  .. (node_9);
  \draw [black,->] (node_7) ..controls (309.66bp,76.193bp) and (320.37bp,87.729bp)  .. (node_11);
  \draw [black,->] (node_8) ..controls (262.66bp,76.193bp) and (273.37bp,87.729bp)  .. (node_9);
  \draw [black,->] (node_10) ..controls (343.39bp,75.437bp) and (343.39bp,85.372bp)  .. (node_11);
  \draw [black,->] (node_12) ..controls (61.39bp,22.834bp) and (61.39bp,32.769bp)  .. (node_13);
  \draw [black,->] (node_12) ..controls (48.12bp,23.589bp) and (37.405bp,35.125bp)  .. (node_14);
  \draw [black,->] (node_12) ..controls (74.66bp,23.589bp) and (85.375bp,35.125bp)  .. (node_16);
  \draw [black,->] (node_13) ..controls (61.39bp,75.437bp) and (61.39bp,85.372bp)  .. (node_15);
  \draw [black,->] (node_13) ..controls (74.66bp,76.193bp) and (85.375bp,87.729bp)  .. (node_17);
  \draw [black,->] (node_14) ..controls (27.66bp,76.193bp) and (38.375bp,87.729bp)  .. (node_15);
  \draw [black,->] (node_16) ..controls (108.39bp,75.437bp) and (108.39bp,85.372bp)  .. (node_17);
\end{tikzpicture} 
}

\end{tabular}\end{center}
\end{mexample}

\begin{mexample}
Another example, where \(P=C_{2} \times C_{2} \times C_{2}\):
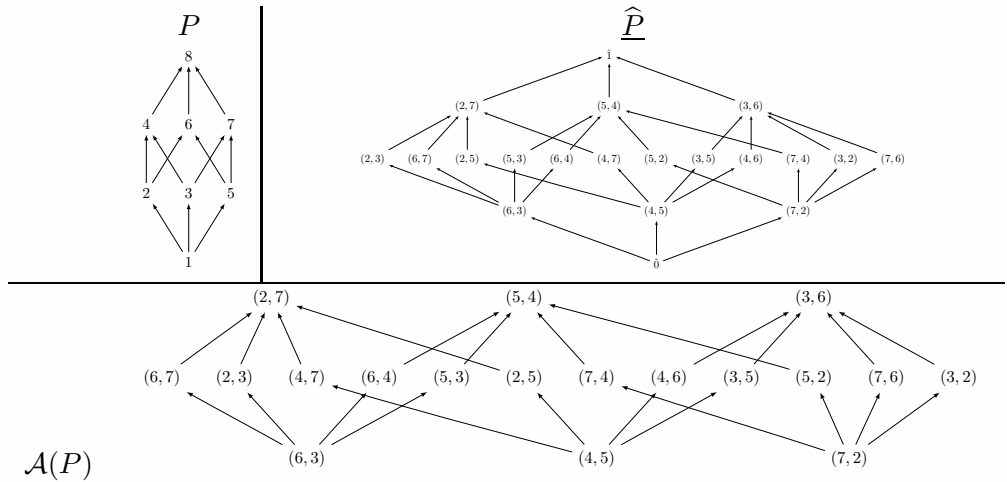
\phantomsection
\label{}
\begin{center}\begin{tabular}{p{1cm}c|c}
 & \(P\) & \(\Pqu\)  \\
 &
\resizebox{!}{3cm}{
 \begin{tikzpicture}[>=latex,line join=bevel,]
\node (node_0) at (35.811bp,6.5307bp) [draw,draw=none] {$1$};
  \node (node_1) at (5.8112bp,55.592bp) [draw,draw=none] {$2$};
  \node (node_2) at (35.811bp,55.592bp) [draw,draw=none] {$3$};
  \node (node_4) at (65.811bp,55.592bp) [draw,draw=none] {$5$};
  \node (node_3) at (5.8112bp,104.65bp) [draw,draw=none] {$4$};
  \node (node_5) at (35.811bp,104.65bp) [draw,draw=none] {$6$};
  \node (node_6) at (65.811bp,104.65bp) [draw,draw=none] {$7$};
  \node (node_7) at (35.811bp,153.71bp) [draw,draw=none] {$8$};
  \draw [black,->] (node_0) ..controls (28.256bp,19.383bp) and (21.198bp,30.455bp)  .. (node_1);
  \draw [black,->] (node_0) ..controls (35.811bp,19.029bp) and (35.811bp,29.252bp)  .. (node_2);
  \draw [black,->] (node_0) ..controls (43.367bp,19.383bp) and (50.424bp,30.455bp)  .. (node_4);
  \draw [black,->] (node_1) ..controls (5.8112bp,68.09bp) and (5.8112bp,78.313bp)  .. (node_3);
  \draw [black,->] (node_1) ..controls (13.367bp,68.445bp) and (20.424bp,79.516bp)  .. (node_5);
  \draw [black,->] (node_2) ..controls (28.256bp,68.445bp) and (21.198bp,79.516bp)  .. (node_3);
  \draw [black,->] (node_2) ..controls (43.367bp,68.445bp) and (50.424bp,79.516bp)  .. (node_6);
  \draw [black,->] (node_3) ..controls (13.367bp,117.51bp) and (20.424bp,128.58bp)  .. (node_7);
  \draw [black,->] (node_4) ..controls (58.256bp,68.445bp) and (51.198bp,79.516bp)  .. (node_5);
  \draw [black,->] (node_4) ..controls (65.811bp,68.09bp) and (65.811bp,78.313bp)  .. (node_6);
  \draw [black,->] (node_5) ..controls (35.811bp,117.15bp) and (35.811bp,127.37bp)  .. (node_7);
  \draw [black,->] (node_6) ..controls (58.256bp,117.51bp) and (51.198bp,128.58bp)  .. (node_7);
\end{tikzpicture}
}

  &
\resizebox{!}{3cm}{
 \begin{tikzpicture}[>=latex,line join=bevel,]
\node (node_0) at (296.39bp,7.8452bp) [draw,draw=none] {$\hat{0}$};
  \node (node_1) at (296.39bp,59.992bp) [draw,draw=none] {$\left(4, 5\right)$};
  \node (node_6) at (155.39bp,59.992bp) [draw,draw=none] {$\left(6, 3\right)$};
  \node (node_12) at (437.39bp,59.992bp) [draw,draw=none] {$\left(7, 2\right)$};
  \node (node_2) at (108.39bp,112.6bp) [draw,draw=none] {$\left(2, 5\right)$};
  \node (node_3) at (343.39bp,112.6bp) [draw,draw=none] {$\left(3, 5\right)$};
  \node (node_4) at (390.39bp,112.6bp) [draw,draw=none] {$\left(4, 6\right)$};
  \node (node_5) at (249.39bp,112.6bp) [draw,draw=none] {$\left(4, 7\right)$};
  \node (node_11) at (108.39bp,165.2bp) [draw,draw=none] {$\left(2, 7\right)$};
  \node (node_18) at (390.39bp,165.2bp) [draw,draw=none] {$\left(3, 6\right)$};
  \node (node_7) at (14.39bp,112.6bp) [draw,draw=none] {$\left(2, 3\right)$};
  \node (node_8) at (155.39bp,112.6bp) [draw,draw=none] {$\left(5, 3\right)$};
  \node (node_9) at (202.39bp,112.6bp) [draw,draw=none] {$\left(6, 4\right)$};
  \node (node_10) at (61.39bp,112.6bp) [draw,draw=none] {$\left(6, 7\right)$};
  \node (node_16) at (249.39bp,165.2bp) [draw,draw=none] {$\left(5, 4\right)$};
  \node (node_19) at (249.39bp,217.35bp) [draw,draw=none] {$\hat{1}$};
  \node (node_13) at (484.39bp,112.6bp) [draw,draw=none] {$\left(3, 2\right)$};
  \node (node_14) at (296.39bp,112.6bp) [draw,draw=none] {$\left(5, 2\right)$};
  \node (node_15) at (437.39bp,112.6bp) [draw,draw=none] {$\left(7, 4\right)$};
  \node (node_17) at (531.39bp,112.6bp) [draw,draw=none] {$\left(7, 6\right)$};
  \draw [black,->] (node_0) ..controls (296.39bp,21.751bp) and (296.39bp,31.682bp)  .. (node_1);
  \draw [black,->] (node_0) ..controls (272.35bp,17.397bp) and (214.09bp,38.115bp)  .. (node_6);
  \draw [black,->] (node_0) ..controls (320.43bp,17.397bp) and (378.69bp,38.115bp)  .. (node_12);
  \draw [black,->] (node_1) ..controls (255.48bp,71.287bp) and (195.01bp,86.72bp)  .. (node_2);
  \draw [black,->] (node_1) ..controls (309.66bp,75.28bp) and (320.37bp,86.816bp)  .. (node_3);
  \draw [black,->] (node_1) ..controls (324.73bp,76.247bp) and (348.94bp,89.279bp)  .. (node_4);
  \draw [black,->] (node_1) ..controls (283.12bp,75.28bp) and (272.41bp,86.816bp)  .. (node_5);
  \draw [black,->] (node_2) ..controls (108.39bp,127.13bp) and (108.39bp,137.06bp)  .. (node_11);
  \draw [black,->] (node_3) ..controls (356.66bp,127.88bp) and (367.37bp,139.42bp)  .. (node_18);
  \draw [black,->] (node_4) ..controls (390.39bp,127.13bp) and (390.39bp,137.06bp)  .. (node_18);
  \draw [black,->] (node_5) ..controls (211.24bp,127.29bp) and (163.01bp,144.6bp)  .. (node_11);
  \draw [black,->] (node_6) ..controls (117.24bp,74.685bp) and (69.01bp,91.993bp)  .. (node_7);
  \draw [black,->] (node_6) ..controls (155.39bp,74.524bp) and (155.39bp,84.459bp)  .. (node_8);
  \draw [black,->] (node_6) ..controls (168.66bp,75.28bp) and (179.37bp,86.816bp)  .. (node_9);
  \draw [black,->] (node_6) ..controls (127.05bp,76.247bp) and (102.85bp,89.279bp)  .. (node_10);
  \draw [black,->] (node_7) ..controls (42.726bp,128.85bp) and (66.935bp,141.88bp)  .. (node_11);
  \draw [black,->] (node_8) ..controls (183.73bp,128.85bp) and (207.94bp,141.88bp)  .. (node_16);
  \draw [black,->] (node_9) ..controls (215.66bp,127.88bp) and (226.37bp,139.42bp)  .. (node_16);
  \draw [black,->] (node_10) ..controls (74.66bp,127.88bp) and (85.375bp,139.42bp)  .. (node_11);
  \draw [black,->] (node_11) ..controls (149.04bp,180.66bp) and (205.11bp,200.6bp)  .. (node_19);
  \draw [black,->] (node_12) ..controls (450.66bp,75.28bp) and (461.37bp,86.816bp)  .. (node_13);
  \draw [black,->] (node_12) ..controls (399.24bp,74.685bp) and (351.01bp,91.993bp)  .. (node_14);
  \draw [black,->] (node_12) ..controls (437.39bp,74.524bp) and (437.39bp,84.459bp)  .. (node_15);
  \draw [black,->] (node_12) ..controls (465.73bp,76.247bp) and (489.94bp,89.279bp)  .. (node_17);
  \draw [black,->] (node_13) ..controls (456.05bp,128.85bp) and (431.85bp,141.88bp)  .. (node_18);
  \draw [black,->] (node_14) ..controls (283.12bp,127.88bp) and (272.41bp,139.42bp)  .. (node_16);
  \draw [black,->] (node_15) ..controls (419.98bp,118.98bp) and (416.57bp,119.99bp)  .. (413.39bp,120.9bp) .. controls (364.41bp,134.81bp) and (306.78bp,149.68bp)  .. (node_16);
  \draw [black,->] (node_16) ..controls (249.39bp,179.68bp) and (249.39bp,189.68bp)  .. (node_19);
  \draw [black,->] (node_17) ..controls (493.24bp,127.29bp) and (445.01bp,144.6bp)  .. (node_18);
  \draw [black,->] (node_18) ..controls (349.74bp,180.66bp) and (293.67bp,200.6bp)  .. (node_19);
\end{tikzpicture}
}

   \\  \hline
 \(\twoanti{P}\)  &
 \multicolumn{2}{c}{
\resizebox{!}{2.5cm}{
 \begin{tikzpicture}[>=latex,line join=bevel,]
\node (node_0) at (460.39bp,8.3018bp) [draw,draw=none] {$\left(7, 2\right)$};
  \node (node_1) at (296.39bp,60.906bp) [draw,draw=none] {$\left(7, 4\right)$};
  \node (node_2) at (484.39bp,60.906bp) [draw,draw=none] {$\left(7, 6\right)$};
  \node (node_7) at (437.39bp,60.906bp) [draw,draw=none] {$\left(5, 2\right)$};
  \node (node_13) at (531.39bp,60.906bp) [draw,draw=none] {$\left(3, 2\right)$};
  \node (node_8) at (249.39bp,113.51bp) [draw,draw=none] {$\left(5, 4\right)$};
  \node (node_14) at (437.39bp,113.51bp) [draw,draw=none] {$\left(3, 6\right)$};
  \node (node_3) at (108.39bp,8.3018bp) [draw,draw=none] {$\left(6, 3\right)$};
  \node (node_4) at (155.39bp,60.906bp) [draw,draw=none] {$\left(6, 4\right)$};
  \node (node_5) at (14.39bp,60.906bp) [draw,draw=none] {$\left(6, 7\right)$};
  \node (node_6) at (202.39bp,60.906bp) [draw,draw=none] {$\left(5, 3\right)$};
  \node (node_16) at (61.39bp,60.906bp) [draw,draw=none] {$\left(2, 3\right)$};
  \node (node_17) at (85.39bp,113.51bp) [draw,draw=none] {$\left(2, 7\right)$};
  \node (node_9) at (296.39bp,8.3018bp) [draw,draw=none] {$\left(4, 5\right)$};
  \node (node_10) at (343.39bp,60.906bp) [draw,draw=none] {$\left(4, 6\right)$};
  \node (node_11) at (108.39bp,60.906bp) [draw,draw=none] {$\left(4, 7\right)$};
  \node (node_12) at (390.39bp,60.906bp) [draw,draw=none] {$\left(3, 5\right)$};
  \node (node_15) at (249.39bp,60.906bp) [draw,draw=none] {$\left(2, 5\right)$};
  \draw [black,->] (node_0) ..controls (418.26bp,22.302bp) and (356.28bp,41.426bp)  .. (node_1);
  \draw [black,->] (node_0) ..controls (466.95bp,23.136bp) and (471.96bp,33.7bp)  .. (node_2);
  \draw [black,->] (node_0) ..controls (454.1bp,23.136bp) and (449.3bp,33.7bp)  .. (node_7);
  \draw [black,->] (node_0) ..controls (481.32bp,24.218bp) and (498.55bp,36.496bp)  .. (node_13);
  \draw [black,->] (node_1) ..controls (283.12bp,76.193bp) and (272.41bp,87.729bp)  .. (node_8);
  \draw [black,->] (node_2) ..controls (471.12bp,76.193bp) and (460.41bp,87.729bp)  .. (node_14);
  \draw [black,->] (node_3) ..controls (121.66bp,23.589bp) and (132.37bp,35.125bp)  .. (node_4);
  \draw [black,->] (node_3) ..controls (80.054bp,24.556bp) and (55.845bp,37.589bp)  .. (node_5);
  \draw [black,->] (node_3) ..controls (136.73bp,24.556bp) and (160.94bp,37.589bp)  .. (node_6);
  \draw [black,->] (node_3) ..controls (95.12bp,23.589bp) and (84.405bp,35.125bp)  .. (node_16);
  \draw [black,->] (node_4) ..controls (183.73bp,77.16bp) and (207.94bp,90.192bp)  .. (node_8);
  \draw [black,->] (node_5) ..controls (35.318bp,76.822bp) and (52.546bp,89.1bp)  .. (node_17);
  \draw [black,->] (node_6) ..controls (215.66bp,76.193bp) and (226.37bp,87.729bp)  .. (node_8);
  \draw [black,->] (node_7) ..controls (419.98bp,67.292bp) and (416.57bp,68.303bp)  .. (413.39bp,69.207bp) .. controls (364.41bp,83.12bp) and (306.78bp,97.993bp)  .. (node_8);
  \draw [black,->] (node_9) ..controls (309.66bp,23.589bp) and (320.37bp,35.125bp)  .. (node_10);
  \draw [black,->] (node_9) ..controls (255.48bp,19.596bp) and (195.01bp,35.03bp)  .. (node_11);
  \draw [black,->] (node_9) ..controls (324.73bp,24.556bp) and (348.94bp,37.589bp)  .. (node_12);
  \draw [black,->] (node_9) ..controls (283.12bp,23.589bp) and (272.41bp,35.125bp)  .. (node_15);
  \draw [black,->] (node_10) ..controls (371.73bp,77.16bp) and (395.94bp,90.192bp)  .. (node_14);
  \draw [black,->] (node_11) ..controls (102.1bp,75.74bp) and (97.301bp,86.304bp)  .. (node_17);
  \draw [black,->] (node_12) ..controls (403.66bp,76.193bp) and (414.37bp,87.729bp)  .. (node_14);
  \draw [black,->] (node_13) ..controls (503.05bp,77.16bp) and (478.85bp,90.192bp)  .. (node_14);
  \draw [black,->] (node_15) ..controls (207.26bp,74.906bp) and (145.28bp,94.029bp)  .. (node_17);
  \draw [black,->] (node_16) ..controls (67.951bp,75.74bp) and (72.962bp,86.304bp)  .. (node_17);
\end{tikzpicture}
}

}
\end{tabular}\end{center}

This examples illustrates that while the edges of the poset \(\twoanti{P}\) are partitioned
into two parts, the poset itself need not be disconnnected. In this example, selecting the elements \((i,j)\) with
\(i < j\) and forming the induced subposet of \(\twoanti{P}\) yields
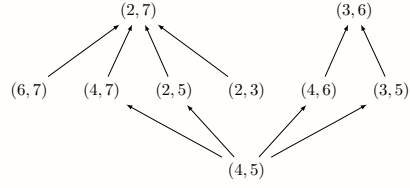
\phantomsection
\label{}
\begin{center}
\resizebox{!}{2.5cm}{
 \begin{tikzpicture}[>=latex,line join=bevel,]
\node (node_0) at (14.39bp,60.906bp) [draw,draw=none] {$\left(6, 7\right)$};
  \node (node_8) at (85.39bp,113.51bp) [draw,draw=none] {$\left(2, 7\right)$};
  \node (node_1) at (155.39bp,8.3018bp) [draw,draw=none] {$\left(4, 5\right)$};
  \node (node_2) at (202.39bp,60.906bp) [draw,draw=none] {$\left(4, 6\right)$};
  \node (node_3) at (61.39bp,60.906bp) [draw,draw=none] {$\left(4, 7\right)$};
  \node (node_4) at (249.39bp,60.906bp) [draw,draw=none] {$\left(3, 5\right)$};
  \node (node_6) at (108.39bp,60.906bp) [draw,draw=none] {$\left(2, 5\right)$};
  \node (node_5) at (225.39bp,113.51bp) [draw,draw=none] {$\left(3, 6\right)$};
  \node (node_7) at (155.39bp,60.906bp) [draw,draw=none] {$\left(2, 3\right)$};
  \draw [black,->] (node_0) ..controls (35.318bp,76.822bp) and (52.546bp,89.1bp)  .. (node_8);
  \draw [black,->] (node_1) ..controls (168.66bp,23.589bp) and (179.37bp,35.125bp)  .. (node_2);
  \draw [black,->] (node_1) ..controls (127.05bp,24.556bp) and (102.85bp,37.589bp)  .. (node_3);
  \draw [black,->] (node_1) ..controls (183.73bp,24.556bp) and (207.94bp,37.589bp)  .. (node_4);
  \draw [black,->] (node_1) ..controls (142.12bp,23.589bp) and (131.41bp,35.125bp)  .. (node_6);
  \draw [black,->] (node_2) ..controls (208.68bp,75.74bp) and (213.48bp,86.304bp)  .. (node_5);
  \draw [black,->] (node_3) ..controls (67.951bp,75.74bp) and (72.962bp,86.304bp)  .. (node_8);
  \draw [black,->] (node_4) ..controls (242.83bp,75.74bp) and (237.82bp,86.304bp)  .. (node_5);
  \draw [black,->] (node_6) ..controls (102.1bp,75.74bp) and (97.301bp,86.304bp)  .. (node_8);
  \draw [black,->] (node_7) ..controls (134.76bp,76.822bp) and (117.77bp,89.1bp)  .. (node_8);
\end{tikzpicture}
}

\end{center}
which does not contain the relation \((6,3) \transterel (2,3)\), nor \((3,2) \transterel (3,6)\),
nor anything related to or derived from
these relations. It is seemingly not possible to discard the ``redundancy'' involved with
\[
(x,y) \transterel (u,v)  \qquad \iff \qquad (v,u) \transterel (y,x).
\]
\label{example-C2C2C2}
\end{mexample}

Even though the poset \(\twoanti{P}\) can not be said to be two identical copies of a smaller poset,
it is nonetheless true that the relations, and also the covering realtions, are naturally split
into two equal parts. We make this precise:
\begin{lemma}
The antitone involution \(\tau\) restricts to an antitone involution on \(\twoanti{P}\);
hence for  \((x,y)\) and \((u,v)\) in \(\twoanti{P}\) we have that
\begin{equation}
\label{tau-on-anti}
(x,y) \transterel (u,v) \quad \iff \quad (v,u) \transterel (y,x)
\end{equation}
\end{lemma}
\begin{proof}
First, recall that \(\tau\) is antitone on \((\Ppairs,\transerel)\) and that
\(\tau = \tau^{-1}\). Secondly, \((x,y) \in \twoanti{P}\) if and only if \((y,x) \in \twoanti{P}\), thus
\(\tau(\twoanti{P}) = \twoanti{P}\). Furthermore, \((\twoanti{P},\transterel)\) is the induced subposet
of \((\Ppairs,\transerel)\). It follows that the restriction of \(\tau\) is antitone on
\((\twoanti{P},\transterel)\) and is equal to its own inverse; thus (\ref{tau-on-anti}) holds.
\end{proof}
\section{The polytope of probability functions.}
\label{sec:org9cf55fd}
\subsection{Definition of the probability functions polytope}
\label{sec:org688a90e}
\begin{theorem}
Let \((P,\le)\) be a finite poset, with \(|P|=n\), and let \(m\) denote the number of
two-element antichains in \(P\), i.e. the cardinality of the set
\[
\{(x,y) \in P \times P \mid x \parallel y\}
\]
is \(2m\). Assume that \(P\) is not a chain, so that \(m > 0\).
Then the set of all probability functions on \(P\)
form a convex polytope which can be realised inside
\(\RR^{^{n^{2}}}\), or inside \(\RR^{^{n^{2}-n}}\),
but also inside \(\RR^{2m}\).
\label{thm-embedd}
\end{theorem}
\begin{proof}
Let \(n = |P|\). Choose some linear extension \(\delta \in \linexts{P}\),
to label the elements in \(P\). For any probability function \(\pi\) on \(P\),
we define an \(n \times n\)-matrix \(A\) by
\[A_{{\delta(x),\delta(y)}} = \pi(x,y).\]
The set of all \(n \times n\)-matrices form an affine space of dimension \(n^{2}\).
By Lemma (\ref{lemma-isconvex}) the subset of matrices arising from
probability functions is convex. But this subset is also contained inside \([0,1]^{n^{2}}\), so it is
bounded. Furthermore, the equalities and inequalities in Definition (\ref{def-probfunc})
are finite in number and cut out the prescribed subset of matrices, so this set
is in fact a polytope \(\subset \RR^{n^{2}}\).

Since \(\pi(x,x)=1\) always, the
corresponding entries in the matrix \(A\) are prescribed. We can thus
project the polytope down to \(\RR^{n^{2} - n}\), by forgetting the coordinates \((i,i)\) .

Similarly, since \(\pi(x,y)=1\) for \(x \le y\), \(\pi(x,y)=0\) for \(x > y\),
we can project the polytope down to \(\QQ^{2m}\), by forgetting the coordinates \((i,j)\) which are not associated
with an \emph{ordered} antichain  \((x,y)\).
\end{proof}

\begin{definition}
Let \((P,\le)\) be a finite poset, with \(|P|=n\)  and containing \(m>0\) \emph{unordered} two-element antichains.
We denote by \(\prpolytope{P} \subset \RR^{2m}\) the probability functions polytope, as described in Theorem (\ref{thm-embedd}).
\end{definition}

\begin{lemma}
Let \(H \subset \RR^{s} \times \RR^{s}\) be the affine subspace
\[H = \{(\vec{x}, \vec{y}) \mid x_{i} + y_{i} = 1 \text{ for } 1 \le i \le s\}.\]
Suppose that \(Q \subset H\) is a polytope. Then
\[
Q  \simeq \mathbb{Proj}(Q),
\] 
where \(\mathbb{Proj}\) denotes the projection \((\vec{x},\vec{y}) \mapsto \vec{x}\),
and \(Q \simeq \mathbb{Proj}(Q)\) means that the two polytopes
are \emph{combinatorially equivalent}, i.e., their face lattices are isomorphic.
\label{lemma-project}
\end{lemma}
\begin{proof}
In fact, the linear maps
\begin{align*}
  \mathbb{Proj}_{\lvert H}: H & \to \RR^{s} \\
  (\vec{x},\vec{y}) & \mapsto \vec{x}
\end{align*}
and
\begin{align*}
  \mathbb{J}: \RR^{s} & \to H \\
  (x_{1}.\dots,x_{n}) & \mapsto (x_{1},\dots, x_{n}; 1 - x_{1},\dots, 1 - x_{n})
\end{align*}
are mutual inverses and thus shows that the polytopes are even \emph{affinely isomorphic}.
\end{proof}
Note however that these maps are not \emph{isometric}, so some combinatorially relevent information, i.e.,
the relative volume of polytopes inside \(H\), are not preserved.
\begin{corollary}
The probability functions polytope \(\prpolytope{P}\) may be defined inside
\([0,1]^{m} \subset \RR^m\) rather that inside \([0,1]^{2m} \subset \RR^{2m}\).
\end{corollary}
\begin{proof}
The equations stemming from \[\pi(a,b) + \pi(b,a)=1\] for all probability functions
show that \[\prpolytope{P} \subset \{\vec{x} \mid x_{a,b} + x_{b,a}=1\},\]
so Lemma \ref{lemma-project} applies. The inequalities \[0 \le \pi(x,y) \le 1\] ensures that the probability functions polytope lies
within the respective unit hypercubes.
\end{proof}

As Example \ref{example-C2C2C2} shows, when \(P=C_{2} \times C_{2} \times C_{2}\),
removing half of the ordered pairs of parallel elements in
\(\twoanti{P}\) does not yield the correct poset. However, intersecting with the affine subspace
\(H\) spanned by the equations \(u_{i} + v_{i}=1\) has the effect of ``removing'' \(v_{i}\)
from the list of inequalities that cut out \(\prpolytope{P}\) by making it redundant.
We will return to \(P=C_{2} \times C_{2} \times C_{2}\) in Example \ref{C2C2C2-vertices} and give a complete list
of irredundant inequalities and equalities for \(\prpolytope{P}\).
\subsection{Examples of probability functions polytopes}
\label{sec:org2c2b1aa}
\subsubsection{The probability functions polytopes of chains}
\label{sec:org8b78ebb}
\begin{lemma}
When \(P\) is a chain, \(\prpolytope{P}\) is a point.
\end{lemma}
\begin{proof}
The value \(\pi(x,y) \in \{0,1\}\), since \(x < y\) or \(x > y\). Thus, there is only a single
probability function on \(P\).

Alternatively, the set of ordered antichains is empty.
\end{proof}
\subsubsection{The probability functions polytopes of antichains}
\label{sec:orgcd61984}
For a two-element antichain \(P\), \(\twoanti{P} \simeq P\),
and \(\prpolytope{P}\) is a line segment.
The 3-element antichain has the cube \([0,1]^{3}\) as its \(\prpolytope{P}\).
More generally:
\begin{lemma}
Suppose that \(\twoanti{P}\) is an antichain poset with \(m\) antichains.
Then \(\prpolytope{P}\) is combinatorially equivalent with the \(m\)-dimensional hypercube.
In particular, the \(d\)-element antichain has a  probability functions polytope that is combinatorially equivalent with the
\({\binom{d}{2}}\)-dimensional hypercube.
\end{lemma}
\begin{proof}
In this case,
there are no restrictions on the value \(\pi(x,y)\) that a probability function
on \(P\) may take on a given (ordered) antichain \((x,y)\) except that
\[0 \le \pi(x,y) \le 1\] and
\[\pi(x,y) + \pi(y,x)=1.\]
Thus \(\prpolytope{P}\) is the unit hypercube \([0,1]^{2m}\)
intersected with the affine subspace \(u_{i} + v_{i}=1\), where
\[U=\{u_{1},\dots,u_{m}\}\]
is a choice of variables (ordered two-element antichains) such that
\[\twoanti{P} = U \sqcup \tau(U).\]
From Lemma \ref{lemma-project} we have that this polytope is combinatorially equivalent with its projection
onto the span of \(U\); this is a hypercube of dimension \(m = d^{2}-d\).
\end{proof}
There are posets \(P\) which are not antichains, but where
\(\twoanti{P}\) is an antichain poset. Examples can be found already for posets with three elements,
see the list in the last chapter.
\subsubsection{The probability functions polytopes of ordinal sums}
\label{sec:org353cf2d}
\begin{lemma}
Let \(P,Q\) be two finite posets. Then the antichains of \(P \oplus Q\) is the disjoint union of the antichains of \(P\) and
the antichains of \(Q\).
\label{lemma-ordinalsum-antichains}
\end{lemma}
\begin{proof}
No \(S \subset P \sqcup Q\) containing elments from both \(P\) and \(Q\) can be an antichain in \(P \oplus Q\).
Conversely, any induced subposet of \(S \subset P\) inside \(P+Q\) is isomorphic to the induced subset inside
\(P\), so it is an antichain of \(P+Q\) iff it is an antichain of \(P\). The same goes for subsets of \(Q\).
\end{proof}
\begin{corollary}
Let \(P,Q\) be two finite posets. Then
\[\twoanti{P \oplus Q} \simeq \twoanti{P} + \twoanti{Q},\]
and
\[\prpolytope{P \oplus Q} \simeq \prpolytope{P} \times \prpolytope{Q},\]
where the relation between the posets is poset isomorphism, and the relations between
the polytopes is combinatoriall equivalence.
\label{corr-ordinalsum-twoanti}
\end{corollary}
\begin{proof}
Suppose that
\[(x,y), (u,v) \in \twoanti{P \oplus Q}.\] By the previous Lemma, we may assume that
either
\[(x,y), (u,v) \in \twoanti{P},\] or
\[(x,y)\in \twoanti{P},\quad (u,v)\in \twoanti{Q}.\].
In the first case, the relation between \((x,y), (u,v)\) is as in \((\twoanti{P}, \transterel)\).
In the second case, \((x,y) \parallel (u,v)\) since no relations of the form (\ref{eq-trel}) apply, given that
\(x,y\) and \(u,v\) belong to disjoint universes.

The second assertion also follows from the Lemma; any probability function
\[\pi: \twoanti{P \oplus Q} \to [0,1]\]
is determined by its restrictions to \(\twoanti{P}\) and to \(\twoanti{Q}\).
\end{proof}

\begin{mexample}
Let \(P_{1}=C_{2} \times C_{3}\), \(P_{2}=C_{2} \times C_{4}\), \(P_{3} = P_{1} \oplus P_{2}\).

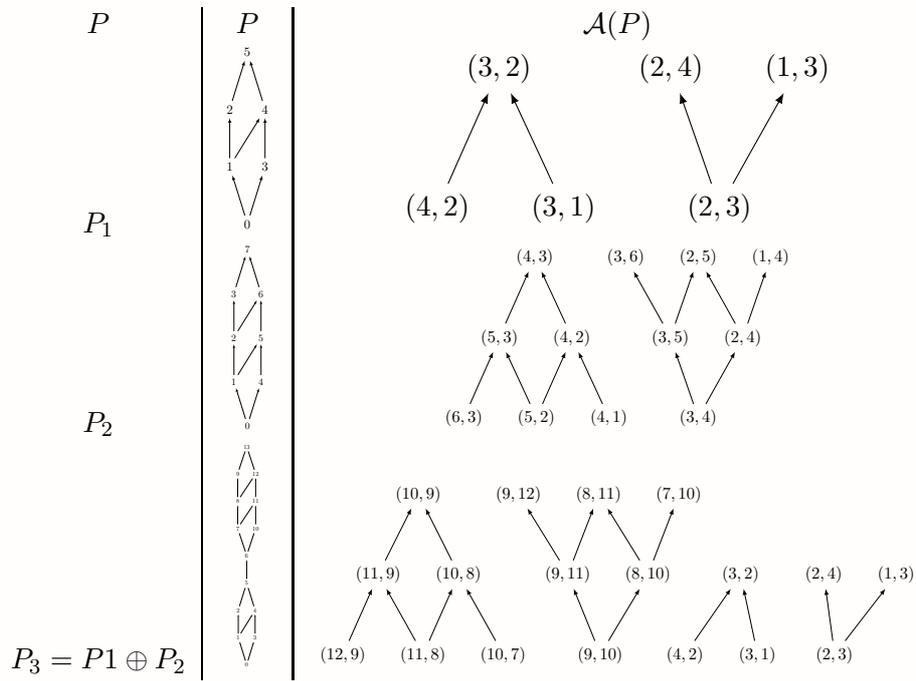
\phantomsection
\label{}
\begin{center}\begin{tabular}{c|c|c}
 \(P\) & \(P\) & \(\twoanti{P}\) \\
 \(P_1\) &
\resizebox{!}{2.5cm}{
 \begin{tikzpicture}[>=latex,line join=bevel,]
\node (node_0) at (20.811bp,6.5307bp) [draw,draw=none] {$0$};
  \node (node_1) at (5.8112bp,55.592bp) [draw,draw=none] {$1$};
  \node (node_3) at (35.811bp,55.592bp) [draw,draw=none] {$3$};
  \node (node_2) at (5.8112bp,104.65bp) [draw,draw=none] {$2$};
  \node (node_4) at (35.811bp,104.65bp) [draw,draw=none] {$4$};
  \node (node_5) at (20.811bp,153.71bp) [draw,draw=none] {$5$};
  \draw [black,->] (node_0) ..controls (17.124bp,19.099bp) and (13.812bp,29.49bp)  .. (node_1);
  \draw [black,->] (node_0) ..controls (24.499bp,19.099bp) and (27.81bp,29.49bp)  .. (node_3);
  \draw [black,->] (node_1) ..controls (5.8112bp,68.09bp) and (5.8112bp,78.313bp)  .. (node_2);
  \draw [black,->] (node_1) ..controls (13.367bp,68.445bp) and (20.424bp,79.516bp)  .. (node_4);
  \draw [black,->] (node_2) ..controls (9.4985bp,117.22bp) and (12.81bp,127.61bp)  .. (node_5);
  \draw [black,->] (node_3) ..controls (35.811bp,68.09bp) and (35.811bp,78.313bp)  .. (node_4);
  \draw [black,->] (node_4) ..controls (32.124bp,117.22bp) and (28.812bp,127.61bp)  .. (node_5);
\end{tikzpicture}
}

 &
\resizebox{!}{2.5cm}{
 \begin{tikzpicture}[>=latex,line join=bevel,]
\node (node_0) at (14.39bp,8.3018bp) [draw,draw=none] {$\left(4, 2\right)$};
  \node (node_2) at (37.39bp,60.906bp) [draw,draw=none] {$\left(3, 2\right)$};
  \node (node_1) at (61.39bp,8.3018bp) [draw,draw=none] {$\left(3, 1\right)$};
  \node (node_3) at (119.39bp,8.3018bp) [draw,draw=none] {$\left(2, 3\right)$};
  \node (node_4) at (101.39bp,60.906bp) [draw,draw=none] {$\left(2, 4\right)$};
  \node (node_5) at (148.39bp,60.906bp) [draw,draw=none] {$\left(1, 3\right)$};
  \draw [black,->] (node_0) ..controls (20.678bp,23.136bp) and (25.479bp,33.7bp)  .. (node_2);
  \draw [black,->] (node_1) ..controls (54.829bp,23.136bp) and (49.819bp,33.7bp)  .. (node_2);
  \draw [black,->] (node_3) ..controls (114.52bp,22.985bp) and (110.88bp,33.233bp)  .. (node_4);
  \draw [black,->] (node_3) ..controls (127.36bp,23.211bp) and (133.51bp,33.935bp)  .. (node_5);
\end{tikzpicture}
}

 \\
 \(P_2\) &
\resizebox{!}{2.5cm}{
 \begin{tikzpicture}[>=latex,line join=bevel,]
\node (node_0) at (20.811bp,6.5307bp) [draw,draw=none] {$0$};
  \node (node_1) at (5.8112bp,55.592bp) [draw,draw=none] {$1$};
  \node (node_4) at (35.811bp,55.592bp) [draw,draw=none] {$4$};
  \node (node_2) at (5.8112bp,104.65bp) [draw,draw=none] {$2$};
  \node (node_5) at (35.811bp,104.65bp) [draw,draw=none] {$5$};
  \node (node_3) at (5.8112bp,153.71bp) [draw,draw=none] {$3$};
  \node (node_6) at (35.811bp,153.71bp) [draw,draw=none] {$6$};
  \node (node_7) at (20.811bp,202.78bp) [draw,draw=none] {$7$};
  \draw [black,->] (node_0) ..controls (17.124bp,19.099bp) and (13.812bp,29.49bp)  .. (node_1);
  \draw [black,->] (node_0) ..controls (24.499bp,19.099bp) and (27.81bp,29.49bp)  .. (node_4);
  \draw [black,->] (node_1) ..controls (5.8112bp,68.09bp) and (5.8112bp,78.313bp)  .. (node_2);
  \draw [black,->] (node_1) ..controls (13.367bp,68.445bp) and (20.424bp,79.516bp)  .. (node_5);
  \draw [black,->] (node_2) ..controls (5.8112bp,117.15bp) and (5.8112bp,127.37bp)  .. (node_3);
  \draw [black,->] (node_2) ..controls (13.367bp,117.51bp) and (20.424bp,128.58bp)  .. (node_6);
  \draw [black,->] (node_3) ..controls (9.4985bp,166.28bp) and (12.81bp,176.67bp)  .. (node_7);
  \draw [black,->] (node_4) ..controls (35.811bp,68.09bp) and (35.811bp,78.313bp)  .. (node_5);
  \draw [black,->] (node_5) ..controls (35.811bp,117.15bp) and (35.811bp,127.37bp)  .. (node_6);
  \draw [black,->] (node_6) ..controls (32.124bp,166.28bp) and (28.812bp,176.67bp)  .. (node_7);
\end{tikzpicture}
}

  &
\resizebox{!}{2.5cm}{
 \begin{tikzpicture}[>=latex,line join=bevel,]
\node (node_0) at (14.39bp,8.3018bp) [draw,draw=none] {$\left(6, 3\right)$};
  \node (node_2) at (37.39bp,60.906bp) [draw,draw=none] {$\left(5, 3\right)$};
  \node (node_1) at (61.39bp,8.3018bp) [draw,draw=none] {$\left(5, 2\right)$};
  \node (node_4) at (84.39bp,60.906bp) [draw,draw=none] {$\left(4, 2\right)$};
  \node (node_5) at (60.39bp,113.51bp) [draw,draw=none] {$\left(4, 3\right)$};
  \node (node_3) at (108.39bp,8.3018bp) [draw,draw=none] {$\left(4, 1\right)$};
  \node (node_6) at (166.39bp,8.3018bp) [draw,draw=none] {$\left(3, 4\right)$};
  \node (node_7) at (148.39bp,60.906bp) [draw,draw=none] {$\left(3, 5\right)$};
  \node (node_9) at (195.39bp,60.906bp) [draw,draw=none] {$\left(2, 4\right)$};
  \node (node_8) at (119.39bp,113.51bp) [draw,draw=none] {$\left(3, 6\right)$};
  \node (node_10) at (166.39bp,113.51bp) [draw,draw=none] {$\left(2, 5\right)$};
  \node (node_11) at (213.39bp,113.51bp) [draw,draw=none] {$\left(1, 4\right)$};
  \draw [black,->] (node_0) ..controls (20.678bp,23.136bp) and (25.479bp,33.7bp)  .. (node_2);
  \draw [black,->] (node_1) ..controls (54.829bp,23.136bp) and (49.819bp,33.7bp)  .. (node_2);
  \draw [black,->] (node_1) ..controls (67.678bp,23.136bp) and (72.479bp,33.7bp)  .. (node_4);
  \draw [black,->] (node_2) ..controls (43.678bp,75.74bp) and (48.479bp,86.304bp)  .. (node_5);
  \draw [black,->] (node_3) ..controls (101.83bp,23.136bp) and (96.819bp,33.7bp)  .. (node_4);
  \draw [black,->] (node_4) ..controls (77.829bp,75.74bp) and (72.819bp,86.304bp)  .. (node_5);
  \draw [black,->] (node_6) ..controls (161.52bp,22.985bp) and (157.88bp,33.233bp)  .. (node_7);
  \draw [black,->] (node_6) ..controls (174.36bp,23.211bp) and (180.51bp,33.935bp)  .. (node_9);
  \draw [black,->] (node_7) ..controls (140.42bp,75.815bp) and (134.27bp,86.539bp)  .. (node_8);
  \draw [black,->] (node_7) ..controls (153.26bp,75.589bp) and (156.9bp,85.836bp)  .. (node_10);
  \draw [black,->] (node_9) ..controls (187.42bp,75.815bp) and (181.27bp,86.539bp)  .. (node_10);
  \draw [black,->] (node_9) ..controls (200.26bp,75.589bp) and (203.9bp,85.836bp)  .. (node_11);
\end{tikzpicture}
}

 \\
 \(P_3 = P1 \oplus P_2\) &
\resizebox{!}{3.0cm}{
 \begin{tikzpicture}[>=latex,line join=bevel,]
\node (node_0) at (21.811bp,6.5307bp) [draw,draw=none] {$0$};
  \node (node_1) at (6.8112bp,55.592bp) [draw,draw=none] {$1$};
  \node (node_3) at (36.811bp,55.592bp) [draw,draw=none] {$3$};
  \node (node_2) at (6.8112bp,104.65bp) [draw,draw=none] {$2$};
  \node (node_4) at (36.811bp,104.65bp) [draw,draw=none] {$4$};
  \node (node_5) at (21.811bp,153.71bp) [draw,draw=none] {$5$};
  \node (node_6) at (21.811bp,202.78bp) [draw,draw=none] {$6$};
  \node (node_7) at (5.8112bp,251.84bp) [draw,draw=none] {$7$};
  \node (node_10) at (37.811bp,251.84bp) [draw,draw=none] {$10$};
  \node (node_8) at (5.8112bp,300.9bp) [draw,draw=none] {$8$};
  \node (node_11) at (37.811bp,300.9bp) [draw,draw=none] {$11$};
  \node (node_9) at (5.8112bp,349.96bp) [draw,draw=none] {$9$};
  \node (node_12) at (37.811bp,349.96bp) [draw,draw=none] {$12$};
  \node (node_13) at (21.811bp,399.02bp) [draw,draw=none] {$13$};
  \draw [black,->] (node_0) ..controls (18.124bp,19.099bp) and (14.812bp,29.49bp)  .. (node_1);
  \draw [black,->] (node_0) ..controls (25.499bp,19.099bp) and (28.81bp,29.49bp)  .. (node_3);
  \draw [black,->] (node_1) ..controls (6.8112bp,68.09bp) and (6.8112bp,78.313bp)  .. (node_2);
  \draw [black,->] (node_1) ..controls (14.367bp,68.445bp) and (21.424bp,79.516bp)  .. (node_4);
  \draw [black,->] (node_2) ..controls (10.499bp,117.22bp) and (13.81bp,127.61bp)  .. (node_5);
  \draw [black,->] (node_3) ..controls (36.811bp,68.09bp) and (36.811bp,78.313bp)  .. (node_4);
  \draw [black,->] (node_4) ..controls (33.124bp,117.22bp) and (29.812bp,127.61bp)  .. (node_5);
  \draw [black,->] (node_5) ..controls (21.811bp,166.21bp) and (21.811bp,176.44bp)  .. (node_6);
  \draw [black,->] (node_6) ..controls (17.854bp,215.42bp) and (14.264bp,225.98bp)  .. (node_7);
  \draw [black,->] (node_6) ..controls (25.768bp,215.42bp) and (29.358bp,225.98bp)  .. (node_10);
  \draw [black,->] (node_7) ..controls (5.8112bp,264.34bp) and (5.8112bp,274.56bp)  .. (node_8);
  \draw [black,->] (node_7) ..controls (13.919bp,264.76bp) and (21.564bp,276.01bp)  .. (node_11);
  \draw [black,->] (node_8) ..controls (5.8112bp,313.4bp) and (5.8112bp,323.62bp)  .. (node_9);
  \draw [black,->] (node_8) ..controls (13.919bp,313.82bp) and (21.564bp,325.07bp)  .. (node_12);
  \draw [black,->] (node_9) ..controls (9.7685bp,362.6bp) and (13.358bp,373.16bp)  .. (node_13);
  \draw [black,->] (node_10) ..controls (37.811bp,264.34bp) and (37.811bp,274.56bp)  .. (node_11);
  \draw [black,->] (node_11) ..controls (37.811bp,313.4bp) and (37.811bp,323.62bp)  .. (node_12);
  \draw [black,->] (node_12) ..controls (33.854bp,362.6bp) and (30.264bp,373.16bp)  .. (node_13);
\end{tikzpicture}
}

 &
\resizebox{!}{2.5cm}{
 \begin{tikzpicture}[>=latex,line join=bevel,]
\node (node_0) at (16.881bp,8.3018bp) [draw,draw=none] {$\left(12, 9\right)$};
  \node (node_2) at (39.881bp,60.906bp) [draw,draw=none] {$\left(11, 9\right)$};
  \node (node_1) at (68.881bp,8.3018bp) [draw,draw=none] {$\left(11, 8\right)$};
  \node (node_4) at (91.881bp,60.906bp) [draw,draw=none] {$\left(10, 8\right)$};
  \node (node_5) at (65.881bp,113.51bp) [draw,draw=none] {$\left(10, 9\right)$};
  \node (node_3) at (120.88bp,8.3018bp) [draw,draw=none] {$\left(10, 7\right)$};
  \node (node_6) at (183.88bp,8.3018bp) [draw,draw=none] {$\left(9, 10\right)$};
  \node (node_7) at (162.88bp,60.906bp) [draw,draw=none] {$\left(9, 11\right)$};
  \node (node_9) at (214.88bp,60.906bp) [draw,draw=none] {$\left(8, 10\right)$};
  \node (node_8) at (130.88bp,113.51bp) [draw,draw=none] {$\left(9, 12\right)$};
  \node (node_10) at (182.88bp,113.51bp) [draw,draw=none] {$\left(8, 11\right)$};
  \node (node_11) at (234.88bp,113.51bp) [draw,draw=none] {$\left(7, 10\right)$};
  \node (node_12) at (238.88bp,8.3018bp) [draw,draw=none] {$\left(4, 2\right)$};
  \node (node_14) at (274.88bp,60.906bp) [draw,draw=none] {$\left(3, 2\right)$};
  \node (node_13) at (285.88bp,8.3018bp) [draw,draw=none] {$\left(3, 1\right)$};
  \node (node_15) at (335.88bp,8.3018bp) [draw,draw=none] {$\left(2, 3\right)$};
  \node (node_16) at (328.88bp,60.906bp) [draw,draw=none] {$\left(2, 4\right)$};
  \node (node_17) at (375.88bp,60.906bp) [draw,draw=none] {$\left(1, 3\right)$};
  \draw [black,->] (node_0) ..controls (23.169bp,23.136bp) and (27.97bp,33.7bp)  .. (node_2);
  \draw [black,->] (node_1) ..controls (60.91bp,23.211bp) and (54.764bp,33.935bp)  .. (node_2);
  \draw [black,->] (node_1) ..controls (75.169bp,23.136bp) and (79.97bp,33.7bp)  .. (node_4);
  \draw [black,->] (node_2) ..controls (46.989bp,75.74bp) and (52.417bp,86.304bp)  .. (node_5);
  \draw [black,->] (node_3) ..controls (112.91bp,23.211bp) and (106.76bp,33.935bp)  .. (node_4);
  \draw [black,->] (node_4) ..controls (84.773bp,75.74bp) and (79.345bp,86.304bp)  .. (node_5);
  \draw [black,->] (node_6) ..controls (178.17bp,23.06bp) and (173.85bp,33.466bp)  .. (node_7);
  \draw [black,->] (node_6) ..controls (192.45bp,23.287bp) and (199.12bp,34.172bp)  .. (node_9);
  \draw [black,->] (node_7) ..controls (154.04bp,75.891bp) and (147.15bp,86.775bp)  .. (node_8);
  \draw [black,->] (node_7) ..controls (168.32bp,75.664bp) and (172.43bp,86.07bp)  .. (node_10);
  \draw [black,->] (node_9) ..controls (206.04bp,75.891bp) and (199.15bp,86.775bp)  .. (node_10);
  \draw [black,->] (node_9) ..controls (220.32bp,75.664bp) and (224.43bp,86.07bp)  .. (node_11);
  \draw [black,->] (node_12) ..controls (248.88bp,23.362bp) and (256.74bp,34.409bp)  .. (node_14);
  \draw [black,->] (node_13) ..controls (282.92bp,22.909bp) and (280.73bp,33.0bp)  .. (node_14);
  \draw [black,->] (node_15) ..controls (334.0bp,22.909bp) and (332.6bp,33.0bp)  .. (node_16);
  \draw [black,->] (node_15) ..controls (347.05bp,23.438bp) and (355.91bp,34.647bp)  .. (node_17);
\end{tikzpicture}
}

\end{tabular}\end{center}

Put \(A_{i} = \prpolytope{P_{i}}\). Then \(A_{1}\) has vertices (columns in the matrix)
\phantomsection
\label{}
\begin{displaymath}
\left(\begin{array}{rrrrr}
0 & 0 & 0 & 1 & 1 \\
0 & 0 & 1 & 0 & 1 \\
1 & 0 & 1 & 1 & 1 \\
0 & 1 & 0 & 0 & 0 \\
1 & 1 & 1 & 0 & 0 \\
1 & 1 & 0 & 1 & 0
\end{array}\right)
\end{displaymath}
and \(A_{2}\)
\phantomsection
\label{}
\begin{displaymath}
\left(\begin{array}{rrrrrrrrrrrrrr}
0 & 0 & 0 & 0 & 0 & 0 & 0 & 0 & 0 & 1 & 1 & 1 & 1 & 1 \\
0 & 0 & 0 & 0 & 0 & 0 & 0 & 1 & 1 & 0 & 1 & 0 & 1 & 0 \\
0 & 1 & 1 & 0 & 0 & 0 & 1 & 1 & 1 & 1 & 1 & 1 & 1 & 1 \\
0 & 0 & 0 & 0 & 0 & 1 & 1 & 0 & 1 & 0 & 0 & 0 & 1 & 1 \\
0 & 0 & 1 & 1 & 0 & 1 & 1 & 1 & 1 & 0 & 1 & 1 & 1 & 1 \\
1 & 1 & 1 & 1 & 0 & 1 & 1 & 1 & 1 & 1 & 1 & 1 & 1 & 1 \\
0 & 0 & 0 & 0 & 1 & 0 & 0 & 0 & 0 & 0 & 0 & 0 & 0 & 0 \\
1 & 0 & 0 & 1 & 1 & 1 & 0 & 0 & 0 & 0 & 0 & 0 & 0 & 0 \\
1 & 1 & 1 & 1 & 1 & 1 & 1 & 1 & 1 & 0 & 0 & 0 & 0 & 0 \\
1 & 1 & 0 & 0 & 1 & 0 & 0 & 0 & 0 & 1 & 0 & 0 & 0 & 0 \\
1 & 1 & 1 & 1 & 1 & 1 & 1 & 0 & 0 & 1 & 0 & 1 & 0 & 1 \\
1 & 1 & 1 & 1 & 1 & 0 & 0 & 1 & 0 & 1 & 1 & 1 & 0 & 0
\end{array}\right)
\end{displaymath}

It indeed holds that \(A_{3} = A_{1} \times A_{2}\); it has \(5 \times 14 = 70\) vertices obtained by combining the vertices from
\(A_{1}\) and \(A_{2}\). We do not list them here.
\end{mexample}
\subsubsection{The probability functions polytopes of disjoint unions}
\label{sec:org0a26e67}
Considering that the \(n\)-element antichain is the disjoint sum of 1-element antichains,
it seems natural to ask:
\begin{question}
If \(P=P_{1} + P_{2}\) is finite poset, can \(\twoanti{P}\) and \(\prpolytope{P}\)
be described in terms of \(\twoanti{P_{1}}, \twoanti{P_{2}}\) and \(\prpolytope{P_{1}}, \prpolytope{P_{2}}\)?
\end{question}

The following example suggests that the answer is ``no''.
\begin{mexample}
Let \(P_{1}=P_{2} = C_{2}\). Then \(\prpolytope{P_{i}}\) is a point, and \(\twoanti{P_{i}}\) is empty.
However, \(\twoanti{P_1 + P_2}\) is  non-trivial:
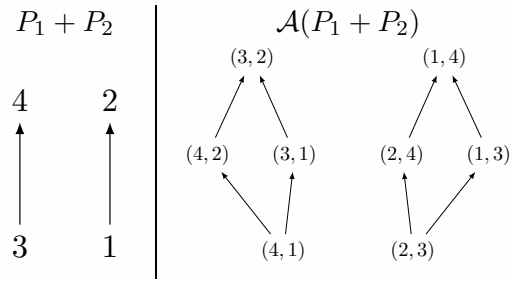
\phantomsection
\label{}
\begin{center}\begin{tabular}{c|c}
 \(P_1 + P_2\) & \(\twoanti{P_1 + P_2}\) \\
\resizebox{!}{2.5cm}{
 \begin{tikzpicture}[>=latex,line join=bevel,]
\node (node_0) at (5.8112bp,6.5307bp) [draw,draw=none] {$3$};
  \node (node_1) at (5.8112bp,55.592bp) [draw,draw=none] {$4$};
  \node (node_2) at (35.811bp,6.5307bp) [draw,draw=none] {$1$};
  \node (node_3) at (35.811bp,55.592bp) [draw,draw=none] {$2$};
  \draw [black,->] (node_0) ..controls (5.8112bp,19.029bp) and (5.8112bp,29.252bp)  .. (node_1);
  \draw [black,->] (node_2) ..controls (35.811bp,19.029bp) and (35.811bp,29.252bp)  .. (node_3);
\end{tikzpicture}
}

 &
\resizebox{!}{3cm}{
 \begin{tikzpicture}[>=latex,line join=bevel,]
\node (node_0) at (55.39bp,8.3018bp) [draw,draw=none] {$\left(4, 1\right)$};
  \node (node_1) at (14.39bp,60.906bp) [draw,draw=none] {$\left(4, 2\right)$};
  \node (node_2) at (61.39bp,60.906bp) [draw,draw=none] {$\left(3, 1\right)$};
  \node (node_3) at (38.39bp,113.51bp) [draw,draw=none] {$\left(3, 2\right)$};
  \node (node_4) at (125.39bp,8.3018bp) [draw,draw=none] {$\left(2, 3\right)$};
  \node (node_5) at (119.39bp,60.906bp) [draw,draw=none] {$\left(2, 4\right)$};
  \node (node_6) at (166.39bp,60.906bp) [draw,draw=none] {$\left(1, 3\right)$};
  \node (node_7) at (142.39bp,113.51bp) [draw,draw=none] {$\left(1, 4\right)$};
  \draw [black,->] (node_0) ..controls (43.876bp,23.514bp) and (34.662bp,34.886bp)  .. (node_1);
  \draw [black,->] (node_0) ..controls (56.995bp,22.834bp) and (58.173bp,32.769bp)  .. (node_2);
  \draw [black,->] (node_1) ..controls (20.951bp,75.74bp) and (25.962bp,86.304bp)  .. (node_3);
  \draw [black,->] (node_2) ..controls (55.102bp,75.74bp) and (50.301bp,86.304bp)  .. (node_3);
  \draw [black,->] (node_4) ..controls (123.79bp,22.834bp) and (122.61bp,32.769bp)  .. (node_5);
  \draw [black,->] (node_4) ..controls (136.9bp,23.514bp) and (146.12bp,34.886bp)  .. (node_6);
  \draw [black,->] (node_5) ..controls (125.68bp,75.74bp) and (130.48bp,86.304bp)  .. (node_7);
  \draw [black,->] (node_6) ..controls (159.83bp,75.74bp) and (154.82bp,86.304bp)  .. (node_7);
\end{tikzpicture}
}

\end{tabular}\end{center}

The polytope \(\prpolytope{P_{1} + P_{2}}\) has dimension 4, so we show only its \textbf{edge graph} and \textbf{face lattice}:
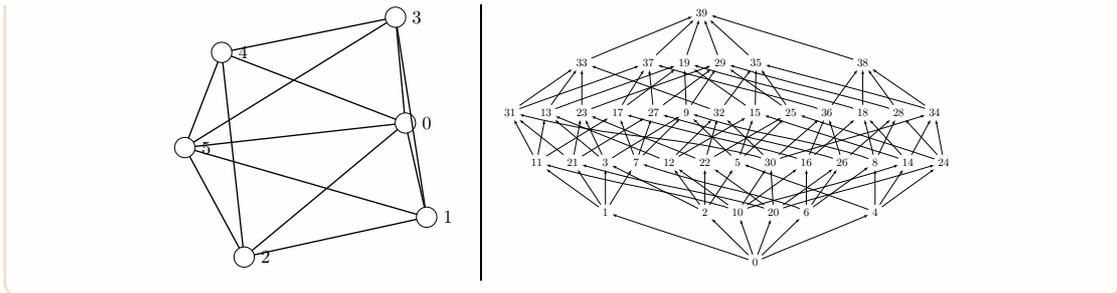
\phantomsection
\label{}
\begin{center}\begin{tabular}{c|c}
\resizebox{!}{3.5cm}{
 \begin{tikzpicture}
\GraphInit[vstyle=Welsh]
\Vertex[L=\hbox{$0$},x=4.5595cm,y=2.7945cm]{v0}
\Vertex[L=\hbox{$1$},x=5.0cm,y=0.8312cm]{v1}
\Vertex[L=\hbox{$2$},x=1.2263cm,y=0.0cm]{v2}
\Vertex[L=\hbox{$3$},x=4.3564cm,y=5.0cm]{v3}
\Vertex[L=\hbox{$4$},x=0.7612cm,y=4.2665cm]{v4}
\Vertex[L=\hbox{$5$},x=0.0cm,y=2.2828cm]{v5}
\Edge[](v0)(v1)
\Edge[](v0)(v2)
\Edge[](v0)(v3)
\Edge[](v0)(v4)
\Edge[](v0)(v5)
\Edge[](v1)(v2)
\Edge[](v1)(v3)
\Edge[](v1)(v5)
\Edge[](v2)(v4)
\Edge[](v2)(v5)
\Edge[](v3)(v4)
\Edge[](v3)(v5)
\Edge[](v4)(v5)
\end{tikzpicture}
}

 &
\resizebox{!}{3.5cm}{
 \begin{tikzpicture}[>=latex,line join=bevel,]
\node (node_0) at (247.3bp,6.5307bp) [draw,draw=none] {$0$};
  \node (node_1) at (101.3bp,55.592bp) [draw,draw=none] {$1$};
  \node (node_2) at (198.3bp,55.592bp) [draw,draw=none] {$2$};
  \node (node_4) at (364.3bp,55.592bp) [draw,draw=none] {$4$};
  \node (node_6) at (297.3bp,55.592bp) [draw,draw=none] {$6$};
  \node (node_10) at (230.3bp,55.592bp) [draw,draw=none] {$10$};
  \node (node_20) at (265.3bp,55.592bp) [draw,draw=none] {$20$};
  \node (node_3) at (101.3bp,104.65bp) [draw,draw=none] {$3$};
  \node (node_7) at (131.3bp,104.65bp) [draw,draw=none] {$7$};
  \node (node_11) at (34.302bp,104.65bp) [draw,draw=none] {$11$};
  \node (node_21) at (69.302bp,104.65bp) [draw,draw=none] {$21$};
  \node (node_5) at (230.3bp,104.65bp) [draw,draw=none] {$5$};
  \node (node_12) at (163.3bp,104.65bp) [draw,draw=none] {$12$};
  \node (node_22) at (198.3bp,104.65bp) [draw,draw=none] {$22$};
  \node (node_9) at (180.3bp,153.71bp) [draw,draw=none] {$9$};
  \node (node_13) at (43.302bp,153.71bp) [draw,draw=none] {$13$};
  \node (node_23) at (78.302bp,153.71bp) [draw,draw=none] {$23$};
  \node (node_8) at (364.3bp,104.65bp) [draw,draw=none] {$8$};
  \node (node_14) at (396.3bp,104.65bp) [draw,draw=none] {$14$};
  \node (node_24) at (431.3bp,104.65bp) [draw,draw=none] {$24$};
  \node (node_15) at (247.3bp,153.71bp) [draw,draw=none] {$15$};
  \node (node_25) at (282.3bp,153.71bp) [draw,draw=none] {$25$};
  \node (node_16) at (297.3bp,104.65bp) [draw,draw=none] {$16$};
  \node (node_26) at (332.3bp,104.65bp) [draw,draw=none] {$26$};
  \node (node_17) at (113.3bp,153.71bp) [draw,draw=none] {$17$};
  \node (node_27) at (148.3bp,153.71bp) [draw,draw=none] {$27$};
  \node (node_18) at (352.3bp,153.71bp) [draw,draw=none] {$18$};
  \node (node_28) at (387.3bp,153.71bp) [draw,draw=none] {$28$};
  \node (node_19) at (178.3bp,202.78bp) [draw,draw=none] {$19$};
  \node (node_29) at (213.3bp,202.78bp) [draw,draw=none] {$29$};
  \node (node_30) at (262.3bp,104.65bp) [draw,draw=none] {$30$};
  \node (node_31) at (8.3019bp,153.71bp) [draw,draw=none] {$31$};
  \node (node_32) at (212.3bp,153.71bp) [draw,draw=none] {$32$};
  \node (node_33) at (78.302bp,202.78bp) [draw,draw=none] {$33$};
  \node (node_34) at (422.3bp,153.71bp) [draw,draw=none] {$34$};
  \node (node_35) at (248.3bp,202.78bp) [draw,draw=none] {$35$};
  \node (node_36) at (317.3bp,153.71bp) [draw,draw=none] {$36$};
  \node (node_37) at (143.3bp,202.78bp) [draw,draw=none] {$37$};
  \node (node_38) at (352.3bp,202.78bp) [draw,draw=none] {$38$};
  \node (node_39) at (195.3bp,251.84bp) [draw,draw=none] {$39$};
  \draw [black,->] (node_0) ..controls (221.34bp,15.899bp) and (149.87bp,38.938bp)  .. (node_1);
  \draw [black,->] (node_0) ..controls (234.52bp,19.809bp) and (221.89bp,31.935bp)  .. (node_2);
  \draw [black,->] (node_0) ..controls (270.05bp,16.68bp) and (321.75bp,37.476bp)  .. (node_4);
  \draw [black,->] (node_0) ..controls (260.42bp,19.88bp) and (273.5bp,32.185bp)  .. (node_6);
  \draw [black,->] (node_0) ..controls (243.1bp,19.17bp) and (239.28bp,29.73bp)  .. (node_10);
  \draw [black,->] (node_0) ..controls (251.75bp,19.17bp) and (255.79bp,29.73bp)  .. (node_20);
  \draw [black,->] (node_1) ..controls (101.3bp,68.09bp) and (101.3bp,78.313bp)  .. (node_3);
  \draw [black,->] (node_1) ..controls (108.86bp,68.445bp) and (115.91bp,79.516bp)  .. (node_7);
  \draw [black,->] (node_1) ..controls (85.913bp,67.401bp) and (65.934bp,81.435bp)  .. (node_11);
  \draw [black,->] (node_1) ..controls (93.194bp,68.515bp) and (85.549bp,79.76bp)  .. (node_21);
  \draw [black,->] (node_2) ..controls (177.93bp,66.474bp) and (139.2bp,85.269bp)  .. (node_3);
  \draw [black,->] (node_2) ..controls (206.41bp,68.515bp) and (214.06bp,79.76bp)  .. (node_5);
  \draw [black,->] (node_2) ..controls (189.38bp,68.586bp) and (180.89bp,80.005bp)  .. (node_12);
  \draw [black,->] (node_2) ..controls (198.3bp,68.09bp) and (198.3bp,78.313bp)  .. (node_22);
  \draw [black,->] (node_3) ..controls (118.94bp,116.16bp) and (147.06bp,132.91bp)  .. (node_9);
  \draw [black,->] (node_3) ..controls (86.937bp,117.31bp) and (70.497bp,130.65bp)  .. (node_13);
  \draw [black,->] (node_3) ..controls (95.579bp,117.36bp) and (90.335bp,128.09bp)  .. (node_23);
  \draw [black,->] (node_4) ..controls (343.04bp,63.871bp) and (295.84bp,80.093bp)  .. (node_5);
  \draw [black,->] (node_4) ..controls (364.3bp,68.09bp) and (364.3bp,78.313bp)  .. (node_8);
  \draw [black,->] (node_4) ..controls (372.41bp,68.515bp) and (380.06bp,79.76bp)  .. (node_14);
  \draw [black,->] (node_4) ..controls (379.69bp,67.401bp) and (399.67bp,81.435bp)  .. (node_24);
  \draw [black,->] (node_5) ..controls (217.18bp,118.0bp) and (204.11bp,130.31bp)  .. (node_9);
  \draw [black,->] (node_5) ..controls (234.51bp,117.29bp) and (238.32bp,127.85bp)  .. (node_15);
  \draw [black,->] (node_5) ..controls (243.49bp,117.59bp) and (257.43bp,130.21bp)  .. (node_25);
  \draw [black,->] (node_6) ..controls (288.74bp,59.926bp) and (285.37bp,61.153bp)  .. (282.3bp,62.123bp) .. controls (225.94bp,79.929bp) and (207.88bp,79.408bp)  .. (node_7);
  \draw [black,->] (node_6) ..controls (313.2bp,67.759bp) and (334.8bp,82.932bp)  .. (node_8);
  \draw [black,->] (node_6) ..controls (297.3bp,68.09bp) and (297.3bp,78.313bp)  .. (node_16);
  \draw [black,->] (node_6) ..controls (306.22bp,68.586bp) and (314.71bp,80.005bp)  .. (node_26);
  \draw [black,->] (node_7) ..controls (144.09bp,117.93bp) and (156.71bp,130.06bp)  .. (node_9);
  \draw [black,->] (node_7) ..controls (126.85bp,117.29bp) and (122.81bp,127.85bp)  .. (node_17);
  \draw [black,->] (node_7) ..controls (135.51bp,117.29bp) and (139.32bp,127.85bp)  .. (node_27);
  \draw [black,->] (node_8) ..controls (355.75bp,109.04bp) and (352.39bp,110.26bp)  .. (349.3bp,111.18bp) .. controls (285.26bp,130.39bp) and (264.63bp,127.36bp)  .. (node_9);
  \draw [black,->] (node_8) ..controls (361.35bp,117.22bp) and (358.7bp,127.61bp)  .. (node_18);
  \draw [black,->] (node_8) ..controls (370.03bp,117.36bp) and (375.27bp,128.09bp)  .. (node_28);
  \draw [black,->] (node_9) ..controls (179.81bp,166.21bp) and (179.38bp,176.44bp)  .. (node_19);
  \draw [black,->] (node_9) ..controls (188.66bp,166.64bp) and (196.55bp,177.88bp)  .. (node_29);
  \draw [black,->] (node_10) ..controls (219.28bp,60.361bp) and (216.18bp,61.336bp)  .. (213.3bp,62.123bp) .. controls (145.9bp,80.579bp) and (125.11bp,79.495bp)  .. (node_11);
  \draw [black,->] (node_10) ..controls (212.32bp,69.225bp) and (193.76bp,82.26bp)  .. (node_12);
  \draw [black,->] (node_10) ..controls (241.45bp,60.082bp) and (245.04bp,61.173bp)  .. (248.3bp,62.123bp) .. controls (302.99bp,78.036bp) and (319.61bp,80.527bp)  .. (node_14);
  \draw [black,->] (node_10) ..controls (248.29bp,69.225bp) and (266.84bp,82.26bp)  .. (node_16);
  \draw [black,->] (node_10) ..controls (238.41bp,68.515bp) and (246.06bp,79.76bp)  .. (node_30);
  \draw [black,->] (node_11) ..controls (36.501bp,117.15bp) and (38.456bp,127.37bp)  .. (node_13);
  \draw [black,->] (node_11) ..controls (54.877bp,117.91bp) and (78.83bp,132.18bp)  .. (node_17);
  \draw [black,->] (node_11) ..controls (27.793bp,117.44bp) and (21.771bp,128.33bp)  .. (node_31);
  \draw [black,->] (node_12) ..controls (135.95bp,116.38bp) and (88.289bp,135.07bp)  .. (node_13);
  \draw [black,->] (node_12) ..controls (184.69bp,117.63bp) and (211.64bp,132.74bp)  .. (node_15);
  \draw [black,->] (node_12) ..controls (176.09bp,117.93bp) and (188.71bp,130.06bp)  .. (node_32);
  \draw [black,->] (node_13) ..controls (72.396bp,164.86bp) and (129.57bp,184.79bp)  .. (node_19);
  \draw [black,->] (node_13) ..controls (52.222bp,166.71bp) and (60.714bp,178.13bp)  .. (node_33);
  \draw [black,->] (node_14) ..controls (385.24bp,109.27bp) and (382.14bp,110.28bp)  .. (379.3bp,111.18bp) .. controls (338.99bp,124.08bp) and (291.65bp,138.9bp)  .. (node_15);
  \draw [black,->] (node_14) ..controls (384.89bp,117.86bp) and (373.72bp,129.81bp)  .. (node_18);
  \draw [black,->] (node_14) ..controls (402.81bp,117.44bp) and (408.83bp,128.33bp)  .. (node_34);
  \draw [black,->] (node_15) ..controls (228.68bp,167.42bp) and (209.29bp,180.64bp)  .. (node_19);
  \draw [black,->] (node_15) ..controls (247.55bp,166.21bp) and (247.76bp,176.44bp)  .. (node_35);
  \draw [black,->] (node_16) ..controls (286.17bp,109.2bp) and (282.58bp,110.28bp)  .. (279.3bp,111.18bp) .. controls (217.34bp,128.22bp) and (198.47bp,129.18bp)  .. (node_17);
  \draw [black,->] (node_16) ..controls (311.82bp,118.07bp) and (326.41bp,130.56bp)  .. (node_18);
  \draw [black,->] (node_16) ..controls (302.28bp,117.36bp) and (306.84bp,128.09bp)  .. (node_36);
  \draw [black,->] (node_17) ..controls (130.75bp,167.35bp) and (148.75bp,180.38bp)  .. (node_19);
  \draw [black,->] (node_17) ..controls (120.86bp,166.57bp) and (127.91bp,177.64bp)  .. (node_37);
  \draw [black,->] (node_18) ..controls (341.16bp,158.23bp) and (337.57bp,159.32bp)  .. (334.3bp,160.25bp) .. controls (276.54bp,176.65bp) and (259.1bp,178.74bp)  .. (node_19);
  \draw [black,->] (node_18) ..controls (352.3bp,166.21bp) and (352.3bp,176.44bp)  .. (node_38);
  \draw [black,->] (node_19) ..controls (182.51bp,215.42bp) and (186.32bp,225.98bp)  .. (node_39);
  \draw [black,->] (node_20) ..controls (254.18bp,60.176bp) and (250.59bp,61.252bp)  .. (247.3bp,62.123bp) .. controls (179.75bp,80.024bp) and (158.95bp,78.914bp)  .. (node_21);
  \draw [black,->] (node_20) ..controls (247.32bp,69.225bp) and (228.76bp,82.26bp)  .. (node_22);
  \draw [black,->] (node_20) ..controls (276.34bp,60.275bp) and (279.45bp,61.267bp)  .. (282.3bp,62.123bp) .. controls (336.86bp,78.485bp) and (353.47bp,81.001bp)  .. (node_24);
  \draw [black,->] (node_20) ..controls (283.29bp,69.225bp) and (301.84bp,82.26bp)  .. (node_26);
  \draw [black,->] (node_20) ..controls (264.57bp,68.09bp) and (263.92bp,78.313bp)  .. (node_30);
  \draw [black,->] (node_21) ..controls (71.501bp,117.15bp) and (73.456bp,127.37bp)  .. (node_23);
  \draw [black,->] (node_21) ..controls (89.877bp,117.91bp) and (113.83bp,132.18bp)  .. (node_27);
  \draw [black,->] (node_21) ..controls (53.02bp,118.22bp) and (36.362bp,131.07bp)  .. (node_31);
  \draw [black,->] (node_22) ..controls (170.95bp,116.38bp) and (123.29bp,135.07bp)  .. (node_23);
  \draw [black,->] (node_22) ..controls (219.69bp,117.63bp) and (246.64bp,132.74bp)  .. (node_25);
  \draw [black,->] (node_22) ..controls (201.74bp,117.22bp) and (204.83bp,127.61bp)  .. (node_32);
  \draw [black,->] (node_23) ..controls (107.4bp,164.86bp) and (164.57bp,184.79bp)  .. (node_29);
  \draw [black,->] (node_23) ..controls (78.302bp,166.21bp) and (78.302bp,176.44bp)  .. (node_33);
  \draw [black,->] (node_24) ..controls (400.35bp,115.43bp) and (333.95bp,136.4bp)  .. (node_25);
  \draw [black,->] (node_24) ..controls (419.89bp,117.86bp) and (408.72bp,129.81bp)  .. (node_28);
  \draw [black,->] (node_24) ..controls (429.1bp,117.15bp) and (427.15bp,127.37bp)  .. (node_34);
  \draw [black,->] (node_25) ..controls (263.68bp,167.42bp) and (244.29bp,180.64bp)  .. (node_29);
  \draw [black,->] (node_25) ..controls (273.64bp,166.71bp) and (265.39bp,178.13bp)  .. (node_35);
  \draw [black,->] (node_26) ..controls (321.17bp,109.2bp) and (317.58bp,110.29bp)  .. (314.3bp,111.18bp) .. controls (251.94bp,128.28bp) and (232.78bp,128.64bp)  .. (node_27);
  \draw [black,->] (node_26) ..controls (346.82bp,118.07bp) and (361.41bp,130.56bp)  .. (node_28);
  \draw [black,->] (node_26) ..controls (328.61bp,117.22bp) and (325.3bp,127.61bp)  .. (node_36);
  \draw [black,->] (node_27) ..controls (165.75bp,167.35bp) and (183.75bp,180.38bp)  .. (node_29);
  \draw [black,->] (node_27) ..controls (147.08bp,166.21bp) and (145.99bp,176.44bp)  .. (node_37);
  \draw [black,->] (node_28) ..controls (376.16bp,158.23bp) and (372.57bp,159.32bp)  .. (369.3bp,160.25bp) .. controls (311.54bp,176.65bp) and (294.1bp,178.74bp)  .. (node_29);
  \draw [black,->] (node_28) ..controls (378.38bp,166.71bp) and (369.89bp,178.13bp)  .. (node_38);
  \draw [black,->] (node_29) ..controls (208.85bp,215.42bp) and (204.81bp,225.98bp)  .. (node_39);
  \draw [black,->] (node_30) ..controls (251.31bp,109.53bp) and (248.2bp,110.48bp)  .. (245.3bp,111.18bp) .. controls (152.92bp,133.67bp) and (124.38bp,125.32bp)  .. (node_31);
  \draw [black,->] (node_30) ..controls (249.18bp,118.0bp) and (236.11bp,130.31bp)  .. (node_32);
  \draw [black,->] (node_30) ..controls (273.46bp,109.12bp) and (277.04bp,110.21bp)  .. (280.3bp,111.18bp) .. controls (332.07bp,126.65bp) and (347.53bp,130.34bp)  .. (node_34);
  \draw [black,->] (node_30) ..controls (276.82bp,118.07bp) and (291.41bp,130.56bp)  .. (node_36);
  \draw [black,->] (node_31) ..controls (27.197bp,167.42bp) and (46.862bp,180.64bp)  .. (node_33);
  \draw [black,->] (node_31) ..controls (37.396bp,164.86bp) and (94.571bp,184.79bp)  .. (node_37);
  \draw [black,->] (node_32) ..controls (183.42bp,164.86bp) and (126.67bp,184.79bp)  .. (node_33);
  \draw [black,->] (node_32) ..controls (221.48bp,166.71bp) and (230.21bp,178.13bp)  .. (node_35);
  \draw [black,->] (node_33) ..controls (104.8bp,214.43bp) and (150.58bp,232.85bp)  .. (node_39);
  \draw [black,->] (node_34) ..controls (411.15bp,158.22bp) and (407.57bp,159.31bp)  .. (404.3bp,160.25bp) .. controls (355.25bp,174.3bp) and (297.07bp,189.37bp)  .. (node_35);
  \draw [black,->] (node_34) ..controls (403.41bp,167.42bp) and (383.74bp,180.64bp)  .. (node_38);
  \draw [black,->] (node_35) ..controls (234.39bp,216.13bp) and (220.54bp,228.43bp)  .. (node_39);
  \draw [black,->] (node_36) ..controls (306.16bp,158.23bp) and (302.57bp,159.32bp)  .. (299.3bp,160.25bp) .. controls (241.54bp,176.65bp) and (224.1bp,178.74bp)  .. (node_37);
  \draw [black,->] (node_36) ..controls (326.22bp,166.71bp) and (334.71bp,178.13bp)  .. (node_38);
  \draw [black,->] (node_37) ..controls (156.95bp,216.13bp) and (170.54bp,228.43bp)  .. (node_39);
  \draw [black,->] (node_38) ..controls (320.37bp,213.35bp) and (248.88bp,234.78bp)  .. (node_39);
\end{tikzpicture}
}

\end{tabular}\end{center}
\end{mexample}
\subsubsection{More exemples of probability functions polytopes}
\label{sec:orgc08e42e}

\begin{mexample}
We take \(P=C_{2} \times C_{2}\) and calculate \(\twoanti{P}\) and \(\prpolytope{P}\).
We do the same for \(P=C_{2} \times C_{3}\) and for the so-called ``Pentagon poset''.
The polytope  \(\prpolytope{C_{2} \times C_{2}}\) is a line segment; we show its edge graph.
The other polytopes live in \(\RR^{6}\) and \(\RR^{4}\), so they are first projected down
to their affine hulls (of dimensions two and three) before the are displayed.
\phantomsection
\label{}
\begin{center}
\begin{tabular}{c|c|c|c|c}
 \(P\) & \(\twoanti{P}\) & \(\Pqu\) & \(\prpolytope{P}\) & \(\prpolytope{P}\), face L. \\
\end{tabular}\end{center}

For the Pentagon poset, the polytope \(\prpolytope{P}\) is two-dimensional, but since there are 4 ordered pairs
of parallel elements, in the most convenient presentation it  lives in \(\RR^{4}\), where it
is defined by (according to SageMath \autocite{sagemath})
\begin{displaymath}
\left(x_{0} + x_{3} = 1, x_{1} + x_{2} = 1, -x_{0} + x_{1} \geq 0, -x_{1} \geq -1, x_{0} \geq 0\right)
\end{displaymath}
Here, the computer have numbered the variables
\[x_{0}= z_{(4,2)}, \, x_{1}= z_{(3,2)}, \, x_{2} = z_{(2,3)}, \, x_{3}= z_{(2,4).}\]
The non-trivial inequality stems from \(\pi(2,3) \le \pi(2,4)\) and \(\pi(3,2) \ge \pi(4,2)\)
for all probability functions, since \(3 \le 4\) in the poset \(P\),
and the equation comes from \[\pi(2,3)+\pi(3,2)=1=\pi(4,2)+\pi(2,4).\]
\end{mexample}
\subsection{The order polytope of a finite poset}
\label{sec:org18135af}
We will describe the probability functions polytope \(\prpolytope{P}\) in terms of the \textbf{order polytope}
of the posets we introduced in the earlier section. Recall
\autocites{StanleyTwoPosetPolytopes}[][]{AhmadOrderChain}[][]{HibiUnimodularOrderpolytopes}[][]{HibiEdgesOrderpolytope}[][]{vonBellTriangulationsOrderpolytopes}
that when \(Q=\{q_{1},\dots,q_{n}\}\) is a finite poset,
the order polytope \(\orderpolytope{Q}\) is the subset of \(\RR^{n} = \RR^{Q}\) consisting of
those \(f: Q \to \RR\) such that
\begin{align}
\label{eq:order-polytope1}
  0 \le f(x) \le 1, & \text{ for all } x \in Q, \\
  f(x) \le f(y), & \text{ if } x \le y \text{ in } Q,
\end{align}
or equivalently
\begin{align}
\label{eq:order-polytope2}
  0 \le f(x), & \text{ if } x \in Q \text{ is a minimal element, } \\
  f(x) \le 1, & \text{ if } x \in Q \text{ is a maximal element, } \\
  f(x) \le f(y), & \text{ if } x \lessdot y \text{ in } Q.
\end{align}

We denote by \(\hat{Q}\) the poset obtained from \(Q\) by
adjoining a minimal element \(\hat{0}\) and a maximal element \(\hat{1}\).
Then a polytope which is combinatorially equivalent to \(\orderpolytope{Q}\) can be obtained by setting \(\orderpolytopebounds{Q}\)
to all functions \(g \in \RR^{\hat{Q}}\) such that
\begin{align}
\label{eq:order-polytope3}
  g(\hat{0}) = 0, \\
  g(\hat{1}) = 1, \\
  g(x) \le g(y), & \text{ if } x \le y \text{ in } Q.
\end{align}
This has the benefit of having just a single type of conditions, with the exceptions of the equations for the
added elements; on the other hand, the ambient dimension increases by two.
The polytopes are \emph{combinatorially equivalent}
in the case under consideration, since the restriction map
\begin{align*}
\rho: \orderpolytopebounds{Q} & \to \orderpolytope{Q} \\
\rho(g)(q) & = g(q)
\end{align*}
is a linear bijection, and hence defines a combinatorial equivalence of polytopes.

\begin{remark}
Be aware that some authors define the order polytope as all \emph{order-reversing} maps
from \(Q\) to \([0,1]\).
\end{remark}

A few more properties of order polytopes:
\begin{itemize}
\item They are 0/1-polytopes, i.e., the coordinates of their vertices are 0 or 1,
\item in particular, they are lattice polytopes,
\item the dimension \(d\) of \(\orderpolytope{Q}\) is the number of elements in \(Q\),
\item the volume of \(\orderpolytope{Q}\) is \(e(Q)/d!\) where \(e(Q)\) counts the number of linear (total)
extensions of \(Q\),
\item they are \emph{distributive polytopes} meaning that componentwise min and max of points in the polytope
remains in the polytope,
\item the facet-inducing inequalities of \(\orderpolytope{Q}\) (or \(\orderpolytopebounds{Q}\)) are those
of (\ref{eq:order-polytope1}) (of (\ref{eq:order-polytope2})) that stem from covering relations.
\end{itemize}
\subsection{The face lattice structure of the order polytope}
\label{sec:orgf1b747c}
The following theorem by Geissinger \autocite{geissingerpolytope} (see also \autocite{StanleyTwoPosetPolytopes}),
and its corollary, describes the face-lattice --- in particular the vertices --- of \(\orderpolytope{Q}\) for
a finite poset \(Q\).

\begin{theorem}[Geissinger]
Let \(Q\) be a finite poset, and let \(\orderpolytope{Q}\) be its order polytope. Then the face lattice
of \(\orderpolytope{Q}\) is anti-isomorphic to the lattice of connected and compatible set partitions on \(Q\).
In particular, the vertices of \(\orderpolytope{Q}\) correspond to characteristic functions on order filters of \(Q\).
\label{thm-geissinger}
\end{theorem}

Hibi and Li \autocite{HibiCutting}\autocite{HibiEdgesOrderpolytope} provided the following description
of the edges of the order polytope \(\orderpolytope{P}\):
\begin{corollary}[Hibi and Li]
Let \(P\) be a finite poset.
Given filters \(S,T\) with \(S \neq T\), denote by \(\chi(S), \chi(T) \in \RR^{P}\) the corresponding characteristic functions.
Then the line segment connecting \(\chi(S)\) and \(\chi(T)\) is an edge of \(\orderpolytope{P}\)
if and only if \(S \subset T\) and \(T \setminus S\) is connected in \(P\)
(possibly after switching \(S\) and \(T\)).
\end{corollary}
\subsection{Relation between the probability functions polytope and the order polytope of the antichains posets}
\label{sec:orgccf695c}
\begin{theorem}
The probability functions polytope \(\prpolytope{P}\) is the intersection of the order polytope
\(\orderpolytope{\twoanti{P}} \subset \RR^{\twoanti{P}}\)
with the affine subspace \(H\) cut out by the equations corresponding to \(\pi(x,y) + \pi(y,x)=1\).
\label{thm-poset2polytope}
\end{theorem}
\begin{proof}
We regard the probability functions polytope  as the polytope \(\prpolytope{P} \subset \RR^{{\twoanti{P}}}\) of all functions
\(f: \twoanti{P} \to \RR\), subject to
\begin{enumerate}
\item \(0 \le f(x,y) \le 1\),
\item \(f((x,y)) \le f((u,v))\), if \(x=u\) and \(y \le v\) in \(P\),
\item \(f((x,y)) \le f((u,v))\), if \(x \ge u\) in \(P\) and \(y =v\),
\item \(f((x,y)) + f((u,v)) = 1\) if \((u,v) = (y,x)\).
\end{enumerate}
The second and third conditions can be summarized as
\(f((x,y)) \le f((u,v))\) if \((x,y) \transerel (u,v)\), hence the first three conditions are simply
the inequalities cutting out the order polytope \(\orderpolytope{\twoanti{P}}\).
The equations in condition 4 are the equations defining the affine subspace \(H\).
\end{proof}
\subsection{The vertices of the probability functions polytope}
\label{sec:org7e0e5be}
Since \(\prpolytope{P} = \orderpolytope{\twoanti{P}} \cap H\), we can say something about the vertices:
\begin{theorem}
The vertices of \(\prpolytope{P}\) which have 0/1 coordinates corresponds to order filters \(S \subset \twoanti{P}\) with the property
that
\[S \ni u \iff \tau(u) \not \in S.\]
\end{theorem}
\begin{proof}
Let the ordered antichains of \(P\) be
\[
\{(p_{i_{1}}, p_{j_{1}}), \dots, (p_{i_{m}}, p_{j_{m}}), (p_{j_{1}}, p_{i_{1}}), \dots, (p_{j_{m}}, p_{j_{1}}) \}.
\]
The order polytope
\[\orderpolytope{\twoanti{P}} \subset [0,1]^{2m}\] have as vertices the characteristic functions on
order filters of \(\twoanti{P}\), by Geissinger's theorem. To get \(\prpolytope{P}\),
we intersect with the affine subspace \(H \subset \RR^{2m}\) cut out by by the \(m\) equations
\[x_{a,b} + x_{b,a} = 1,\]
where the variable \(x_{a,b}\) is associated with the ordered antichain
\[(p_{i_{a}},p_{j_{b}}) \in \twoanti{P}.\]

The resulting polytope \[H \cap \prpolytope{P} \subset H\] has, as its vertices,  those vertices of
\(\prpolytope{P}\) that belong to \(H\), and potentially new vertices arising from
the fact that the hyperplanes defining \(H\) may not ``cut'' the order polytope
in the sense of \autocite{HibiCutting}.

However, any vertex of \(\prpolytope{P}\) with 0/1 coordinates was already a vertex of
\(\orderpolytope{\twoanti{P}}\). The characteristic function/vector of a such vertex in
\(\prpolytope{P}\) is a zero-one vector, and belong to \(H\) iff the coordinate associated to
the variable \(x_{a,b}\) is zero iff \(x_{b,a}\) is one. In other words,
the filter of which the characteristic function is the characteristic function should contain
precisely one of each pair
\[
(p_{i_{a}},p_{j_{b}}), \, (p_{i_{b}},p_{j_{a}}).
\]
\end{proof}

\begin{mexample}
The pentagon poset \(P\) has a probability functions polytope \(\prpolytope{P} \subset \RR^{4}\) whose vertices
happens to be those vertices of \(\orderpolytope{\twoanti{P}}\) that also lie on \(H\).
That is to say, no new vertices arise from intersecting with \(H\).
We have projected the 2-dimensional \(\prpolytope{P} \subset \RR^{4}\) down to its affine hull
\(\RR^{2}\).
\phantomsection
\label{}
\begin{displaymath}\begin{array}{c|c|c|c|c}
 P & \twoanti{P} &  \text{Vertices of } \orderpolytope{\twoanti{P}} & \prpolytope{P}& \text{Vertices of } \prpolytope{P}  \\ \hline
\resizebox{!}{2.5cm}{
 \begin{tikzpicture}[>=latex,line join=bevel,]
\node (node_0) at (20.811bp,6.5307bp) [draw,draw=none] {$1$};
  \node (node_1) at (5.8112bp,104.65bp) [draw,draw=none] {$2$};
  \node (node_2) at (37.811bp,55.592bp) [draw,draw=none] {$3$};
  \node (node_4) at (20.811bp,153.71bp) [draw,draw=none] {$5$};
  \node (node_3) at (36.811bp,104.65bp) [draw,draw=none] {$4$};
  \draw [black,->] (node_0) ..controls (17.765bp,27.051bp) and (11.828bp,65.094bp)  .. (node_1);
  \draw [black,->] (node_0) ..controls (25.016bp,19.17bp) and (28.83bp,29.73bp)  .. (node_2);
  \draw [black,->] (node_1) ..controls (9.4985bp,117.22bp) and (12.81bp,127.61bp)  .. (node_4);
  \draw [black,->] (node_2) ..controls (37.567bp,68.09bp) and (37.35bp,78.313bp)  .. (node_3);
  \draw [black,->] (node_3) ..controls (32.854bp,117.29bp) and (29.264bp,127.85bp)  .. (node_4);
\end{tikzpicture}
}

   &
\resizebox{2cm}{!}{
 \begin{tikzpicture}[>=latex,line join=bevel,]
\node (node_0) at (14.39bp,8.3018bp) [draw,draw=none] {$\left(4, 2\right)$};
  \node (node_1) at (14.39bp,60.906bp) [draw,draw=none] {$\left(3, 2\right)$};
  \node (node_2) at (61.39bp,8.3018bp) [draw,draw=none] {$\left(2, 3\right)$};
  \node (node_3) at (61.39bp,60.906bp) [draw,draw=none] {$\left(2, 4\right)$};
  \draw [black,->] (node_0) ..controls (14.39bp,22.834bp) and (14.39bp,32.769bp)  .. (node_1);
  \draw [black,->] (node_2) ..controls (61.39bp,22.834bp) and (61.39bp,32.769bp)  .. (node_3);
\end{tikzpicture}
}

   &
\left(\begin{array}{rrrr}
1 & 1 & 0 & 1 \\
1 & 1 & 0 & 0 \\
1 & 1 & 1 & 1 \\
0 & 1 & 1 & 1 \\
0 & 0 & 1 & 1 \\
0 & 0 & 0 & 1 \\
0 & 1 & 0 & 1 \\
0 & 0 & 0 & 0 \\
0 & 1 & 0 & 0
\end{array}\right)
   &
\begin{tikzpicture}%
	[scale=1.000000,
	back/.style={loosely dotted, thin},
	edge/.style={color=blue!95!black, thick},
	facet/.style={fill=blue!95!black,fill opacity=0.800000},
	vertex/.style={inner sep=1pt,circle,draw=green!25!black,fill=green!75!black,thick}]
%
%
\coordinate (0.00000, 0.00000) at (0.00000, 0.00000);
\coordinate (0.00000, 1.00000) at (0.00000, 1.00000);
\coordinate (1.00000, 1.00000) at (1.00000, 1.00000);
\fill[facet] (1.00000, 1.00000) -- (0.00000, 0.00000) -- (0.00000, 1.00000) -- cycle {};
\draw[edge] (0.00000, 0.00000) -- (0.00000, 1.00000);
\draw[edge] (0.00000, 0.00000) -- (1.00000, 1.00000);
\draw[edge] (0.00000, 1.00000) -- (1.00000, 1.00000);
\node[vertex] at (0.00000, 0.00000)     {};
\node[vertex] at (0.00000, 1.00000)     {};
\node[vertex] at (1.00000, 1.00000)     {};
\end{tikzpicture}
   &
\left(\begin{array}{rrrr}
0 & 1 & 0 & 1 \\
0 & 0 & 1 & 1 \\
1 & 1 & 0 & 0
\end{array}\right)
 \\
\end{array}\end{displaymath}
\end{mexample}

\begin{mexample}
Surprisingly enough, \(\prpolytope{P}\) need not be a 0/1-polytope; new vertices may
occur in \(H \cap \orderpolytope{\twoanti{P}}\) that do not belong to \(\orderpolytope{\twoanti{P}}\).
The smallest example that we have found where this happens is the
poset \(P=C_{2} \times C_{2} \times C_{2}\) studied in Example \ref{example-C2C2C2}.
Since we use the polytope library of the compututer algebra system SageMath \autocite{sagemath},
the variables, corresponding to the ordered pairs of parallel elements in \(P\), will be zero-indexed.
The ordered antichains and their numberings are
\phantomsection
\label{}
\begin{center}
\begin{tabular}{lrrrrrrrrr}
nr & 0 & 1 & 2 & 3 & 4 & 5 & 6 & 7 & 8\\
antichain & (7, 2) & (7, 4) & (7, 6) & (6, 3) & (6, 4) & (6, 7) & (5, 3) & (5, 2) & (5, 4)\\
nr & 9 & 10 & 11 & 12 & 13 & 14 & 15 & 16 & 17\\
antichain & (4, 5) & (4, 6) & (4, 7) & (3, 5) & (3, 2) & (3, 6) & (2, 5) & (2, 3) & (2, 7)\\
\end{tabular}
\end{center}

The affine subspace \(H\) is cut out by the equalities

\phantomsection
\label{}
\begin{displaymath}\begin{array}{ccccc}
x_{0} + x_{17} = 1
 &
x_{13} + x_{16} = 1
 &
x_{7} + x_{15} = 1
 &
x_{3} + x_{14} = 1
 \\
x_{6} + x_{12} = 1
 &
x_{1} + x_{11} = 1
 &
x_{4} + x_{10} = 1
 &
x_{8} + x_{9} = 1
 \\
x_{2} + x_{5} = 1
 &
\end{array}\end{displaymath}

and the order polytope \(\orderpolytope{\twoanti{C_{2} \times C_{2} \times C_{2}}}\) is defined by
the trivial inequalities \(0 \le x_i \le 1\) and the non-trivial inequalities

\phantomsection
\label{}
\begin{displaymath}\begin{array}{cccc}
-x_{0} + x_{13} \geq 0
 &
-x_{0} + x_{7} \geq 0
 &
-x_{0} + x_{2} \geq 0
 &
-x_{0} + x_{1} \geq 0
 \\
-x_{1} + x_{8} \geq 0
 &
-x_{2} + x_{14} \geq 0
 &
-x_{3} + x_{16} \geq 0
 &
-x_{3} + x_{6} \geq 0
 \\
-x_{3} + x_{5} \geq 0
 &
-x_{3} + x_{4} \geq 0
 &
-x_{4} + x_{8} \geq 0
 &
-x_{5} + x_{17} \geq 0
 \\
-x_{6} + x_{8} \geq 0
 &
-x_{7} + x_{8} \geq 0
 &
-x_{9} + x_{15} \geq 0
 &
-x_{9} + x_{12} \geq 0
 \\
-x_{9} + x_{11} \geq 0
 &
-x_{9} + x_{10} \geq 0
 &
-x_{10} + x_{14} \geq 0
 &
-x_{11} + x_{17} \geq 0
 \\
-x_{12} + x_{14} \geq 0
 &
-x_{13} + x_{14} \geq 0
 &
-x_{15} + x_{17} \geq 0
 &
-x_{16} + x_{17} \geq 0
 \\
\end{array}\end{displaymath}

The probability functions polytope
\[\prpolytope{C_{2} \times C_{2} \times C_{2}} = \orderpolytope{\twoanti{C_{2} \times C_{2} \times C_{2}}} \cap H\]
is defined by the irreduntant inequalities
\phantomsection
\label{}
\begin{displaymath}\begin{array}{cccc}
-x_{0} + x_{13} \geq 0
 &
-x_{0} + x_{7} \geq 0
 &
-x_{0} + x_{2} \geq 0
 &
-x_{0} + x_{1} \geq 0
 \\
-x_{1} + x_{8} \geq 0
 &
-x_{2} - x_{3} \geq -1
 &
-x_{3} - x_{13} \geq -1
 &
-x_{3} + x_{6} \geq 0
 \\
-x_{3} + x_{4} \geq 0
 &
-x_{4} + x_{8} \geq 0
 &
-x_{6} + x_{8} \geq 0
 &
-x_{7} + x_{8} \geq 0
 \\
-x_{8} \geq -1
 &
x_{3} \geq 0
 &
x_{0} \geq 0
 &
\end{array}\end{displaymath}
and equations
\phantomsection
\label{}
\begin{displaymath}\begin{array}{cccc}
x_{0} + x_{17} = 1
 &
x_{13} + x_{16} = 1
 &
x_{7} + x_{15} = 1
 &
x_{3} + x_{14} = 1
 \\
x_{6} + x_{12} = 1
 &
x_{1} + x_{11} = 1
 &
x_{4} + x_{10} = 1
 &
x_{8} + x_{9} = 1
 \\
x_{2} + x_{5} = 1
 &
\end{array}\end{displaymath}

Notice that the variable \(x_{{17}}\) has been removed from the list of inequalities, thanks to the
equation \(x_{0}+x_{17}=1\); in the projection down to \(\RR^{9}\), it disappears altogether.

There are 77 vertices of \(\prpolytope{C_{2} \times C_{2} \times C_{2}}\):

\phantomsection
\label{}
\begin{tiny}\begin{displaymath}\begin{array}{cc}
\left(0,\,1,\,1,\,0,\,0,\,0,\,0,\,0,\,1,\,0,\,1,\,0,\,1,\,0,\,1,\,1,\,1,\,1\right)
 &
\left(0,\,1,\,1,\,0,\,0,\,0,\,0,\,1,\,1,\,0,\,1,\,0,\,1,\,1,\,1,\,0,\,0,\,1\right)
 \\
\left(0,\,1,\,0,\,0,\,0,\,1,\,0,\,0,\,1,\,0,\,1,\,0,\,1,\,0,\,1,\,1,\,1,\,1\right)
 &
\left(0,\,1,\,1,\,0,\,0,\,0,\,0,\,0,\,1,\,0,\,1,\,0,\,1,\,1,\,1,\,1,\,0,\,1\right)
 \\
\left(0,\,1,\,0,\,0,\,0,\,1,\,0,\,1,\,1,\,0,\,1,\,0,\,1,\,1,\,1,\,0,\,0,\,1\right)
 &
\left(0,\,0,\,1,\,0,\,0,\,0,\,0,\,1,\,1,\,0,\,1,\,1,\,1,\,1,\,1,\,0,\,0,\,1\right)
 \\
\left(0,\,0,\,0,\,0,\,0,\,1,\,0,\,0,\,1,\,0,\,1,\,1,\,1,\,1,\,1,\,1,\,0,\,1\right)
 &
\left(0,\,0,\,0,\,0,\,0,\,1,\,0,\,0,\,1,\,0,\,1,\,1,\,1,\,0,\,1,\,1,\,1,\,1\right)
 \\
\left(0,\,0,\,0,\,0,\,0,\,1,\,0,\,1,\,1,\,0,\,1,\,1,\,1,\,1,\,1,\,0,\,0,\,1\right)
 &
\left(0,\,0,\,1,\,0,\,0,\,0,\,0,\,0,\,1,\,0,\,1,\,1,\,1,\,1,\,1,\,1,\,0,\,1\right)
 \\
\left(0,\,1,\,0,\,0,\,0,\,1,\,0,\,0,\,1,\,0,\,1,\,0,\,1,\,1,\,1,\,1,\,0,\,1\right)
 &
\left(0,\,1,\,1,\,0,\,0,\,0,\,0,\,1,\,1,\,0,\,1,\,0,\,1,\,0,\,1,\,0,\,1,\,1\right)
 \\
\left(0,\,1,\,0,\,0,\,0,\,1,\,0,\,1,\,1,\,0,\,1,\,0,\,1,\,0,\,1,\,0,\,1,\,1\right)
 &
\left(0,\,0,\,1,\,0,\,0,\,0,\,0,\,1,\,1,\,0,\,1,\,1,\,1,\,0,\,1,\,0,\,1,\,1\right)
 \\
\left(0,\,0,\,0,\,0,\,0,\,1,\,0,\,1,\,1,\,0,\,1,\,1,\,1,\,0,\,1,\,0,\,1,\,1\right)
 &
\left(0,\,0,\,1,\,0,\,0,\,0,\,0,\,0,\,1,\,0,\,1,\,1,\,1,\,0,\,1,\,1,\,1,\,1\right)
 \\
\left(0,\,0,\,1,\,0,\,1,\,0,\,0,\,0,\,1,\,0,\,0,\,1,\,1,\,0,\,1,\,1,\,1,\,1\right)
 &
\left(0,\,0,\,0,\,0,\,1,\,1,\,0,\,1,\,1,\,0,\,0,\,1,\,1,\,0,\,1,\,0,\,1,\,1\right)
 \\
\left(0,\,0,\,1,\,0,\,1,\,0,\,0,\,1,\,1,\,0,\,0,\,1,\,1,\,0,\,1,\,0,\,1,\,1\right)
 &
\left(0,\,1,\,0,\,0,\,1,\,1,\,0,\,1,\,1,\,0,\,0,\,0,\,1,\,0,\,1,\,0,\,1,\,1\right)
 \\
\left(0,\,1,\,1,\,0,\,1,\,0,\,0,\,1,\,1,\,0,\,0,\,0,\,1,\,0,\,1,\,0,\,1,\,1\right)
 &
\left(0,\,1,\,0,\,0,\,1,\,1,\,0,\,0,\,1,\,0,\,0,\,0,\,1,\,1,\,1,\,1,\,0,\,1\right)
 \\
\left(0,\,0,\,1,\,0,\,1,\,0,\,0,\,0,\,1,\,0,\,0,\,1,\,1,\,1,\,1,\,1,\,0,\,1\right)
 &
\left(0,\,0,\,0,\,0,\,1,\,1,\,0,\,1,\,1,\,0,\,0,\,1,\,1,\,1,\,1,\,0,\,0,\,1\right)
 \\
\left(0,\,0,\,0,\,0,\,1,\,1,\,0,\,0,\,1,\,0,\,0,\,1,\,1,\,0,\,1,\,1,\,1,\,1\right)
 &
\left(0,\,0,\,0,\,0,\,1,\,1,\,0,\,0,\,1,\,0,\,0,\,1,\,1,\,1,\,1,\,1,\,0,\,1\right)
 \\
\left(0,\,0,\,1,\,0,\,1,\,0,\,0,\,1,\,1,\,0,\,0,\,1,\,1,\,1,\,1,\,0,\,0,\,1\right)
 &
\left(0,\,1,\,0,\,0,\,1,\,1,\,0,\,1,\,1,\,0,\,0,\,0,\,1,\,1,\,1,\,0,\,0,\,1\right)
 \\
\left(0,\,1,\,1,\,0,\,1,\,0,\,0,\,0,\,1,\,0,\,0,\,0,\,1,\,1,\,1,\,1,\,0,\,1\right)
 &
\left(0,\,1,\,0,\,0,\,1,\,1,\,0,\,0,\,1,\,0,\,0,\,0,\,1,\,0,\,1,\,1,\,1,\,1\right)
 \\
\left(0,\,1,\,1,\,0,\,1,\,0,\,0,\,1,\,1,\,0,\,0,\,0,\,1,\,1,\,1,\,0,\,0,\,1\right)
 &
\left(0,\,1,\,1,\,0,\,1,\,0,\,0,\,0,\,1,\,0,\,0,\,0,\,1,\,0,\,1,\,1,\,1,\,1\right)
 \\
\left(0,\,0,\,1,\,0,\,0,\,0,\,1,\,0,\,1,\,0,\,1,\,1,\,0,\,0,\,1,\,1,\,1,\,1\right)
 &
\left(0,\,0,\,0,\,0,\,0,\,1,\,1,\,1,\,1,\,0,\,1,\,1,\,0,\,0,\,1,\,0,\,1,\,1\right)
 \\
\left(0,\,0,\,1,\,0,\,0,\,0,\,1,\,1,\,1,\,0,\,1,\,1,\,0,\,0,\,1,\,0,\,1,\,1\right)
 &
\left(0,\,1,\,0,\,0,\,0,\,1,\,1,\,1,\,1,\,0,\,1,\,0,\,0,\,0,\,1,\,0,\,1,\,1\right)
 \\
\left(0,\,1,\,1,\,0,\,0,\,0,\,1,\,1,\,1,\,0,\,1,\,0,\,0,\,0,\,1,\,0,\,1,\,1\right)
 &
\left(0,\,1,\,0,\,0,\,0,\,1,\,1,\,0,\,1,\,0,\,1,\,0,\,0,\,1,\,1,\,1,\,0,\,1\right)
 \\
\left(0,\,0,\,1,\,0,\,0,\,0,\,1,\,0,\,1,\,0,\,1,\,1,\,0,\,1,\,1,\,1,\,0,\,1\right)
 &
\left(0,\,0,\,0,\,0,\,0,\,1,\,1,\,1,\,1,\,0,\,1,\,1,\,0,\,1,\,1,\,0,\,0,\,1\right)
 \\
\left(0,\,0,\,0,\,0,\,0,\,1,\,1,\,0,\,1,\,0,\,1,\,1,\,0,\,0,\,1,\,1,\,1,\,1\right)
 &
\left(0,\,0,\,0,\,0,\,0,\,1,\,1,\,0,\,1,\,0,\,1,\,1,\,0,\,1,\,1,\,1,\,0,\,1\right)
 \\
\left(0,\,0,\,1,\,0,\,0,\,0,\,1,\,1,\,1,\,0,\,1,\,1,\,0,\,1,\,1,\,0,\,0,\,1\right)
 &
\left(0,\,1,\,0,\,0,\,0,\,1,\,1,\,1,\,1,\,0,\,1,\,0,\,0,\,1,\,1,\,0,\,0,\,1\right)
 \\
\left(0,\,1,\,1,\,0,\,0,\,0,\,1,\,0,\,1,\,0,\,1,\,0,\,0,\,1,\,1,\,1,\,0,\,1\right)
 &
\left(0,\,1,\,0,\,0,\,0,\,1,\,1,\,0,\,1,\,0,\,1,\,0,\,0,\,0,\,1,\,1,\,1,\,1\right)
 \\
\left(0,\,1,\,1,\,0,\,0,\,0,\,1,\,1,\,1,\,0,\,1,\,0,\,0,\,1,\,1,\,0,\,0,\,1\right)
 &
\left(0,\,1,\,1,\,0,\,0,\,0,\,1,\,0,\,1,\,0,\,1,\,0,\,0,\,0,\,1,\,1,\,1,\,1\right)
 \\
\left(0,\,0,\,1,\,0,\,1,\,0,\,1,\,0,\,1,\,0,\,0,\,1,\,0,\,0,\,1,\,1,\,1,\,1\right)
 &
\left(0,\,0,\,0,\,0,\,1,\,1,\,1,\,1,\,1,\,0,\,0,\,1,\,0,\,0,\,1,\,0,\,1,\,1\right)
 \\
\left(0,\,0,\,1,\,0,\,1,\,0,\,1,\,1,\,1,\,0,\,0,\,1,\,0,\,0,\,1,\,0,\,1,\,1\right)
 &
\left(0,\,1,\,0,\,0,\,1,\,1,\,1,\,1,\,1,\,0,\,0,\,0,\,0,\,0,\,1,\,0,\,1,\,1\right)
 \\
\left(0,\,1,\,1,\,0,\,1,\,0,\,1,\,1,\,1,\,0,\,0,\,0,\,0,\,0,\,1,\,0,\,1,\,1\right)
 &
\left(0,\,1,\,0,\,0,\,1,\,1,\,1,\,0,\,1,\,0,\,0,\,0,\,0,\,1,\,1,\,1,\,0,\,1\right)
 \\
\left(0,\,0,\,1,\,0,\,1,\,0,\,1,\,0,\,1,\,0,\,0,\,1,\,0,\,1,\,1,\,1,\,0,\,1\right)
 &
\left(0,\,0,\,0,\,0,\,1,\,1,\,1,\,1,\,1,\,0,\,0,\,1,\,0,\,1,\,1,\,0,\,0,\,1\right)
 \\
\left(0,\,0,\,0,\,0,\,1,\,1,\,1,\,0,\,1,\,0,\,0,\,1,\,0,\,0,\,1,\,1,\,1,\,1\right)
 &
\left(0,\,0,\,0,\,0,\,1,\,1,\,1,\,0,\,1,\,0,\,0,\,1,\,0,\,1,\,1,\,1,\,0,\,1\right)
 \\
\left(0,\,0,\,1,\,0,\,1,\,0,\,1,\,1,\,1,\,0,\,0,\,1,\,0,\,1,\,1,\,0,\,0,\,1\right)
 &
\left(0,\,1,\,0,\,0,\,1,\,1,\,1,\,1,\,1,\,0,\,0,\,0,\,0,\,1,\,1,\,0,\,0,\,1\right)
 \\
\left(0,\,1,\,1,\,0,\,1,\,0,\,1,\,0,\,1,\,0,\,0,\,0,\,0,\,1,\,1,\,1,\,0,\,1\right)
 &
\left(0,\,1,\,0,\,0,\,1,\,1,\,1,\,0,\,1,\,0,\,0,\,0,\,0,\,0,\,1,\,1,\,1,\,1\right)
 \\
\left(0,\,1,\,1,\,0,\,1,\,0,\,1,\,1,\,1,\,0,\,0,\,0,\,0,\,1,\,1,\,0,\,0,\,1\right)
 &
\left(0,\,1,\,1,\,0,\,1,\,0,\,1,\,0,\,1,\,0,\,0,\,0,\,0,\,0,\,1,\,1,\,1,\,1\right)
 \\
\left(0,\,0,\,1,\,0,\,0,\,0,\,0,\,0,\,0,\,1,\,1,\,1,\,1,\,0,\,1,\,1,\,1,\,1\right)
 &
\left(0,\,0,\,1,\,0,\,0,\,0,\,0,\,0,\,0,\,1,\,1,\,1,\,1,\,1,\,1,\,1,\,0,\,1\right)
 \\
\left(0,\,0,\,0,\,0,\,0,\,1,\,0,\,0,\,0,\,1,\,1,\,1,\,1,\,1,\,1,\,1,\,0,\,1\right)
 &
\left(0,\,0,\,0,\,0,\,0,\,1,\,0,\,0,\,0,\,1,\,1,\,1,\,1,\,0,\,1,\,1,\,1,\,1\right)
 \\
\left(0,\,0,\,0,\,1,\,1,\,1,\,1,\,0,\,1,\,0,\,0,\,1,\,0,\,0,\,0,\,1,\,1,\,1\right)
 &
\left(0,\,1,\,0,\,1,\,1,\,1,\,1,\,0,\,1,\,0,\,0,\,0,\,0,\,0,\,0,\,1,\,1,\,1\right)
 \\
\left(0,\,0,\,0,\,1,\,1,\,1,\,1,\,1,\,1,\,0,\,0,\,1,\,0,\,0,\,0,\,0,\,1,\,1\right)
 &
\left(0,\,1,\,0,\,1,\,1,\,1,\,1,\,1,\,1,\,0,\,0,\,0,\,0,\,0,\,0,\,0,\,1,\,1\right)
 \\
\left(1,\,1,\,1,\,0,\,1,\,0,\,0,\,1,\,1,\,0,\,0,\,0,\,1,\,1,\,1,\,0,\,0,\,0\right)
 &
\left(1,\,1,\,1,\,0,\,0,\,0,\,1,\,1,\,1,\,0,\,1,\,0,\,0,\,1,\,1,\,0,\,0,\,0\right)
 \\
\left(1,\,1,\,1,\,0,\,1,\,0,\,1,\,1,\,1,\,0,\,0,\,0,\,0,\,1,\,1,\,0,\,0,\,0\right)
 &
\left(\frac{1}{2},\,\frac{1}{2},\,\frac{1}{2},\,\frac{1}{2},\,\frac{1}{2},\,\frac{1}{2},\,\frac{1}{2},\,\frac{1}{2},\,\frac{1}{2},\,\frac{1}{2},\,\frac{1}{2},\,\frac{1}{2},\,\frac{1}{2},\,\frac{1}{2},\,\frac{1}{2},\,\frac{1}{2},\,\frac{1}{2},\,\frac{1}{2}\right)
 \\
\left(1,\,1,\,1,\,0,\,0,\,0,\,0,\,1,\,1,\,0,\,1,\,0,\,1,\,1,\,1,\,0,\,0,\,0\right)
 &
\end{array}\end{displaymath}\end{tiny}

We see that \(\prpolytope{C_{2} \times C_{2} \times C_{2}}\)
has exactly one vertex that does not have 0/1-coordinates, namely \(\frac{1}{2} \mathbf{1}\),
which corresponds to the probability function \(\probnull{}\).
So, \(\prpolytope{C_{2} \times C_{2} \times C_{2}}\) is not a 0/1-polytope, and not a lattice polytope.
\label{C2C2C2-vertices}
\end{mexample}
\section{Summary and directions for further investigations}
\label{sec:org2b04fb3}
In their article, Kim, Kim, and Neggers \autocite{neggers2019probfunc} defined
the notion of prabability functions on posets that we have studied here.
They also introduced some families of probability functions on posets
embedded in the ``ordered plane'', that is to say \(\RR^{2}\) with the product order
\[(x,y) \le (u,v) \quad \iff \quad (x \le u) \, \wedge \, (y \le v).\]
Any poset (not only finite posets) of \emph{order-dimension} at most two can be so embedded, see for instance
\autocite{Trotter:Poset}. Indeed, the constructions by Kim, Kim, and Neggers yield
probability functions on  any induced subposet of the ``ordered plane''. Specializing to finite posets,
the probability functions so obtained depend on the embedding. It would be interesting to
study what subsets of the probability functions polytope \(\prpolytope{P}\) that their probability functions cut out,
for a given finite poset \(P\), as the embedding and the parameters of their construction vary.

The authors reference an earlier work \autocite{neggers2011trends} where probability functions
on certain algebras where defined, once again by inequalities that might be amenable to polyhedral methods.
We have note explored this.

It is natural to ask if probability functions polytopes of ``easy'' classes of finite posets, such as chain products,
boolean lattices, and the like, may be described explicitly. On related note,
if \(P,Q\) are finite posets, and \(\odot\) some natural binary operation on finite posets, how is
\(\prpolytope{P \odot Q}\) related to \(\prpolytope{P}\) and \(\prpolytope{Q}\)? We have answered this
question for the trivial case of ordinal sums, but there are many more options to explore.
The order polytope and the chain polytope enjoys many nice properties in this regard \autocite{Freij-Hollanti_Lundström_2024}.

We have to date found one (1) finite poset \(P\) such that \(\prpolytope{P}\) is not a 0/1-polytope.
It would indeed be interesting to find more.

The linear extensions polytope can be realized inside the probability functions polytope and our intuition is that they should
be ``close'', at least for posets with few elements. There are various variants and relaxations of this
polytope that have been studied in the literature, and the general problem of describing for instance
the vertices of said polytope is possibly intractable; still, it would indeed be interesting to
study how the linear extensions polytope ``sits inside'' the probability functions polytope for some small examples.
\section{The polytope of probability functions for posets with three or four elements.}
\label{sec:org862a568}
We show the poset \(P\) and  probability functions poset \(\twoanti{P}\), as well as the
edge-graph and the \(f\)-vector
of the probability functions polytope \(\prpolytope{P}\),
for \(P\) with three or four elements, excepting the chains and the antichains.

\phantomsection
\label{n-posets-poly}
\begin{center}
\begin{longtable}{|c|c|c|c|}
 \(P\) & \(\twoanti{P}\)  & edg. g. \(\prpolytope{P}\) & f. vect. \(\prpolytope{P}\) \endhead \hline
\resizebox{2cm}{!}{
 \begin{tikzpicture}[>=latex,line join=bevel,]
\node (node_0) at (5.8112bp,6.5307bp) [draw,draw=none] {$1$};
  \node (node_2) at (20.811bp,55.592bp) [draw,draw=none] {$2$};
  \node (node_1) at (35.811bp,6.5307bp) [draw,draw=none] {$0$};
  \draw [black,->] (node_0) ..controls (9.4985bp,19.099bp) and (12.81bp,29.49bp)  .. (node_2);
  \draw [black,->] (node_1) ..controls (32.124bp,19.099bp) and (28.812bp,29.49bp)  .. (node_2);
\end{tikzpicture}
}

   &
\resizebox{3cm}{!}{
 \begin{tikzpicture}[>=latex,line join=bevel,]
\node (node_0) at (14.39bp,8.3018bp) [draw,draw=none] {$\left(1, 0\right)$};
  \node (node_1) at (61.39bp,8.3018bp) [draw,draw=none] {$\left(0, 1\right)$};
\end{tikzpicture}
}

   &
\resizebox{3cm}{!}{
 \begin{tikzpicture}
\GraphInit[vstyle=Welsh]
\Vertex[L=\hbox{$0$},x=5.0cm,y=0.0cm]{v0}
\Vertex[L=\hbox{$1$},x=0.0cm,y=5.0cm]{v1}
\Edge[](v0)(v1)
\end{tikzpicture}
}

   &
 (1, 2, 1)
 \\ \hline
\resizebox{2cm}{!}{
 \begin{tikzpicture}[>=latex,line join=bevel,]
\node (node_0) at (20.811bp,6.5307bp) [draw,draw=none] {$0$};
  \node (node_1) at (5.8112bp,55.592bp) [draw,draw=none] {$1$};
  \node (node_2) at (35.811bp,55.592bp) [draw,draw=none] {$2$};
  \draw [black,->] (node_0) ..controls (17.124bp,19.099bp) and (13.812bp,29.49bp)  .. (node_1);
  \draw [black,->] (node_0) ..controls (24.499bp,19.099bp) and (27.81bp,29.49bp)  .. (node_2);
\end{tikzpicture}
}

   &
\resizebox{3cm}{!}{
 \begin{tikzpicture}[>=latex,line join=bevel,]
\node (node_0) at (14.39bp,8.3018bp) [draw,draw=none] {$\left(2, 1\right)$};
  \node (node_1) at (61.39bp,8.3018bp) [draw,draw=none] {$\left(1, 2\right)$};
\end{tikzpicture}
}

   &
\resizebox{3cm}{!}{
 \begin{tikzpicture}
\GraphInit[vstyle=Welsh]
\Vertex[L=\hbox{$0$},x=5.0cm,y=5.0cm]{v0}
\Vertex[L=\hbox{$1$},x=0.0cm,y=0.0cm]{v1}
\Edge[](v0)(v1)
\end{tikzpicture}
}

   &
 (1, 2, 1)
 \\
\hline
\end{longtable}\end{center}
\section{Bibliography}
\label{sec:org3aa5880}
\printbibliography
\end{document}